\documentclass{amsart}
\usepackage{amsfonts,amssymb,amsmath,amsthm,amscd}
\usepackage{fullpage}
\usepackage{url}
\usepackage{enumerate}
\usepackage{lmodern}

\newtheorem{theorem}{Theorem}[section]
\newtheorem{lemma}[theorem]{Lemma}
\newtheorem{proposition}[theorem]{Proposition}
\newtheorem{corollary}[theorem]{Corollary}
\newtheorem{conjecture}[theorem]{Conjecture}

\theoremstyle{definition}

\newenvironment{definition}[1][Definition]{\begin{trivlist}
\item[\hskip \labelsep {\bfseries #1}]}{\end{trivlist}}
\newenvironment{remark}[1][Remark]{\begin{trivlist}
\item[\hskip \labelsep {\bfseries #1}]}{\end{trivlist}}

\author{Jeanine Van Order}
\address{Fakult\"at f\"ur Mathematik, Universit\"at Bielefeld, Postfach 100131, 33501 Bielefeld}
\email{jvanorder@math.uni-bielefeld.de}
\thanks{This research was supported in part by the Swiss National Science Foundation grant 200021-125291, 
as well as the DFG CRC 701 in the Fakult\"at f\"ur Mathematik at Bielefeld.}
\subjclass{Primary 11F67; Secondary 11G40, 11R23}
\begin{document}

\title{Rankin-Selberg $L$-functions in cyclotomic towers, I}

\begin{abstract}

We formulate and for the most part prove a conjecture in the style of Mazur-Greenberg for the nonvanishing of central values of Rankin-Selberg $L$-functions 
attached to elliptic curves in abelian extensions of imaginary quadratic fields. This in particular generalizes the theorem of Rohrlich on $L$-functions of elliptic 
curves in cyclotomic towers to the setting of abelian extensions of imaginary quadratic fields, corresponding to families of degree-four $L$-functions given by 
$\operatorname{GL}(2)\times\operatorname{GL}(2)$ Rankin-Selberg $L$-functions. It also generalizes the theorems of 
Rohrlich, Greenberg, Vatsal, and Cornut for $L$-functions of elliptic curves in $\operatorname{Z}_p^2$-extensions of imaginary quadratic fields. \\

%Nous formulons et pour la plupart d\'emontrons une conjecture dans le style de Mazur-Greenberg pour la non-annulation des valeurs centrales de 
%fonctions $L$ Rankin-Selberg attach\'ees aux courbes elliptiques dans extensions ab\'eliennes de corps quadratiques imaginaires. Ce r\'esultat en 
%particulier g\'en\'eralise le th\'eor\`eme de Rohrlich sur courbes elliptiques dans tours cyclotomiques \`a la r\'eglage d'extensions ab\'eliennes de corps 
%quadratiques imaginaires qui correspond \`a familles de fonctions $L$ de degr\'e quatre donn\'ees par fonctions $L$ de type Rankin-Selberg pour 
%$\operatorname{GL}(2)\times\operatorname{GL}(2)$. Il g\'en\'eralise aussi les  th\'eor\`emes de Rohrlich, Greenberg, Vatsal, et Cornut pour fonctions 
%$L$ de courbes elliptiques dans $\operatorname{Z}_p^2$-extensions de corps quadratiques imaginaires.

\end{abstract}

\maketitle
\tableofcontents

\section{Introduction} 

We formulate and prove a conjecture in the style of Mazur-Greenberg for central values of Rankin-Selberg $L$-functions in the non self-dual setting, 
motivated by applications to bounding Mordell-Weil ranks in the setting of two-variable main conjectures of Iwasawa theory for elliptic curves without 
complex multiplication (see e.g.~\cite{CFKS}, \cite{MR}, \cite{PR88}, \cite{VO3}). Such applications are explained in the sequel work \cite{VO6}, along 
with how stronger results can be deduced from the existence of a suitable $p$-adic $L$-function (such as \cite{Hi} and \cite{PR88})
to generalize the theorems of Greenberg \cite{Gr}, Rohrlich \cite{Ro} \cite{Ro2}, Vatsal \cite{Va} and Cornut \cite{Cor}. 
The main purpose of this work is to consider the problem from an analytic point of view, and to derive estimates which 
should be of independent interest. In particular, we develop spectral decompositions of the shifted convolution sums 
and Weyl differencing for prime-power modulus in ways that should be applicable to study average central values of 
arbitrary $\operatorname{GL}(4)$-automorphic $L$-functions. 

Let $f$ be a holomorphic cuspidal eigenform of squarefree level $N$ and trivial nebentype character\footnote{In fact, 
we can work with an arbitrary Hecke-Maass eigenform $f$ of level prime to $pD$ for most of our arguments, 
but restrict to this setting (which is relevant to Shimura's rationality theorems and Iwasawa theory) for simplicity.}. 
We shall assume that $f$ is a newform of integral weight $l \geq 2$, with Fourier series expansion at infinity
\begin{align*} f(z) &= \sum_{n \geq 1} n^{\frac{(l-1)}{2}} \lambda(n) e^{2 \pi i n z}, ~~~~~~~z \in \mathfrak{H} \end{align*} 
normalized so that $\lambda(1)=1$. Let $K$ be an imaginary quadratic field of discriminant $D < 0$ prime to $N$ and 
associated quadratic Dirichlet character $\omega = \omega_D$. Fix an odd prime $p$ which is coprime to the product $ND$. 
Let $\mathcal{W}$ be a finite order Hecke character of $K$ which is ramified only at the given prime $p$.
More precisely, we shall consider Hecke characters $\mathcal{W}$ of the form 
\begin{align}\label{factorization}\mathcal{W} &= \rho \psi =\rho \chi \circ {\bf{N}}, \end{align} 
where $\rho$ is a ring class character of $K$ of $p$-power conductor, and $\psi = \chi \circ {\bf{N}}$ is
induced via composition with the norm ${\bf{N}}$ from a Dirichlet character $\chi$ of $p$-power conductor. 
Note that as functions on nonzero ideals $\mathfrak{a} \subset \mathcal{O}_K$, the Hecke characters in this 
decomposition $(\ref{factorization})$ can be characterized by the conditions $\rho(\overline{\mathfrak{a}}) = \overline{\rho(\mathfrak{a})}$ 
(i.e.~ring class characters are equivariant under complex conjugation) and $\psi (\mathfrak{a})= \chi ( {\bf{N}}\mathfrak{a})$
arises via composition with the norm from a Dirichlet character. 
Let us now fix such a Hecke character $\mathcal{W}$ of $K$ as in $(\ref{factorization})$, 
writing $c(\mathcal{W}) = c(\rho)c(\chi) \in {\bf{Z}}$ to denote its conductor. A classical construction 
due to Hecke associates to $\mathcal{W}$ a holomorphic theta series $\Theta(\mathcal{W})$ of weight $1$, level 
$\vert D \vert c(\mathcal{W})^2$ and nebentype character $\omega \mathcal{W} \vert_{\bf{Q}} = \omega\chi^2$. 
We consider the Rankin-Selberg $L$-function $L(s, f \times \mathcal{W})$ of $f$ times $\Theta(\mathcal{W})$,
whose completed $L$-function $\Lambda(s, f \times \mathcal{W})$ satisfies a functional equation of the form 
\begin{align*} \Lambda(s, f \times \mathcal{W}) &= \epsilon(1/2, f \times \mathcal{W}) \Lambda(1-s, f \times \overline{\mathcal{W}}). \end{align*}
Here, $\epsilon(1/2, f \times \mathcal{W}) \in {\bf{S}}^1$ is a complex number of modulus one known as the root number. 
If $\mathcal{W} = \rho$ is a ring class character of $K$, then it is easy to see that the coefficients in the Dirichlet series expansion of 
$L(s, f \times \mathcal{W})$ are real-valued, and hence that $\epsilon(1/2, f \times \mathcal{W})$ takes values in the set $\lbrace \pm 1 \rbrace$. 
In this setting, the $L$-function $L(s, f \times \mathcal{W})$ is said (for representation theoretic reasons) to be self-dual. 
Moreover, when $\mathcal{W} = \rho$ is a ring class character, it is easy to see that the functional equation  
is symmetric, relating values of the same $L$-function $\Lambda(s, f \times \rho)$ on either side (thanks to the equivariance under complex conjugation). 
In particular, when $\mathcal{W} = \rho$ is a ring class character, if $\epsilon(1/2, f \times \rho)$ equals $-1$ (as opposed to $+1$), 
then the central value $\Lambda(1/2, f \times \rho)$ is forced to vanish by the functional equation. 
Let us for future reference distinguish this particular case of forced vanishing from all others as follows: 

\begin{definition} We refer to a pair $(f, \mathcal{W})$ as {\it{exceptional}} if the Hecke character $\mathcal{W} = \rho$ 
is a ring class character and $\epsilon(1/2, f \times \mathcal{W}) = -1$, and as {\it{generic}} otherwise. \end{definition}

We seek to establish nonvanishing properties of the central derivative values $L'(1/2, f \times \rho)$ 
in the exceptional setting on $(f, \mathcal{W})$ and of the central values $L(1/2, f \times \mathcal{W})$ in the generic setting on $(f, \mathcal{W})$.
Here, we determine nonvanishing estimates in the latter setting which allow for $\mathcal{W} = \rho \psi$ to have both
ring class part $\rho$ and cyclotomic part $\psi = \chi \circ {\bf{N}}$ nontrivial, which fills a gap in the literature. 
Our results also generalize the theorems of Cornut \cite{Cor} and Vatsal \cite{Va} in the self-dual setting, as well as 
the older theorems of Greenberg \cite{Gr} and Rohrlich \cite {Ro}, \cite{Ro2}. 
Note that the proofs of Cornut and Vatsal use completely different methods, and in particular rely on the the special value formulae of 
Gross \cite{Gr}, Waldspuger \cite{Wa}, and Gross-Zagier \cite{GZ}, which (for representation theoretic reasons) 
cannot be extended to the generic non self-dual setting. 

The point of departure in the antecedent works of Greenberg \cite{Gr}, Rohrlich \cite{Ro} \cite{Ro2} and Vatsal \cite{Va} for central 
values is the following rationality theorem of Shimura \cite[Theorem 4]{Sh}. To describe this, let $F$ denote the extension of the rational 
number field ${\bf{Q}}$ obtained by adjoining the eigenvalues of $f$. Hence, $F$ is shown by Shimura to be a number field; in fact, 
$F$ is totally real if $f$ is not dihedral, and otherwise CM. Fix a Hecke character $\mathcal{W} = \rho \chi \circ {\bf{N}}$ as in $(\ref{factorization})$ 
above, with $\rho$ a primitive ring class character of conductor $p^{\alpha}$, and $\chi$ a primitive Dirichlet character of  conductor $p^{\beta}$, 
for some integers $\alpha, \beta \geq 0$. Let $F(\mathcal{W})$ denote the cyclotomic extension of $F$ obtained by adjoining the values of this 
character $\mathcal{W}$. Given a complex embedding of $F(\mathcal{W})$ fixing $F$, let $\mathcal{W}^{\sigma}$ denote the Hecke character 
defined on ideals $\mathfrak{a}$ of $K$ by $\mathfrak{a} \mapsto \mathcal{W}(\mathfrak{a})^{\sigma}$. 
Let $\langle f, f \rangle$ denote the Petersson norm of $f$. Shimura \cite[Theorem 4]{Sh} shows that the values 
\begin{align*} \mathcal{L}(1/2, f \times \mathcal{W}) &= L(1/2, f \times \mathcal{W}) /8 \pi^2 \langle f, f \rangle \end{align*} 
are algebraic, and more precisely that they lie in $F(\mathcal{W})$. Moreover, he proves these values are Galois conjugate 
in the sense that any automorphism $\sigma$ of $F(\mathcal{W})$ acts on the value $\mathcal{L}(1/2, f \times \mathcal{W})$ by the rule 
\begin{align*} \mathcal{L}(1/2, f \times \mathcal{W})^{\sigma} &= \mathcal{L}(1/2, f^{\tau}  \times \mathcal{W}^{\sigma}), \end{align*} 
where $\tau$ denotes the restriction of $\sigma$ to $F$. 
It follows $L(1/2, f \times \mathcal{W})$ vanishes if any only if $L(1/2, f \times \mathcal{W}^{\sigma})$ vanishes for each 
automorphism $\sigma$ of ${\operatorname{Aut}}(F(\mathcal{W}))$ fixing $F$. 
A similar notion of Galois conjugacy can be established in the exceptional 
setting via the formulae of Gross-Zagier \cite{GZ} and more generally Zhang \cite{Zh} 
or Yuan-Zhang-Zhang \cite{YZZ} for the central derivative values $L'(1/2, f \times \rho)$ (cf.~\cite[p.~385]{Ro2}), 
at least when $f$ has weight $l=2$. Thus, it can also be deduced in the exceptional setting 
on $(f, \mathcal{W})$ that the value $L'(1/2, f \times \rho)$ vanishes if any only if the value 
$L'(1/2, f \times \rho^{\sigma})$ vanishes for each complex embedding $\sigma$ of $F(\rho)$ fixing $F$. 
Let us therefore define, in either case on the generic root number $k \in \lbrace 0 , 1 \rbrace$, the associated {$k$-th Galois average 
\begin{align}\label{GA} \delta_{[\mathcal{W}]}^{(k)} &= [F(\mathcal{W}): F]^{-1} \sum_{\sigma} L^{(k)}(1/2, f \times \mathcal{W}^{\sigma}). \end{align} 
Here, the sum runs over embeddings $\sigma: F(\mathcal{W}) \rightarrow {\bf{C}}$ fixing $F$. As well, $L^{(0)}(1/2, f \times \mathcal{W})$ 
denotes the central value $L(1/2, f \times \mathcal{W})$, and $L^{(1)}(1/2, f \times \mathcal{W})$ the derivative value $L'(1/2, f \times \mathcal{W})$. 
We make the following conjecture in the spirit of Mazur, Greenberg \cite{Gr} and Coates-Fukaya-Kato-Sujatha \cite{CFKS} 
for these averages, of which the theorems of Greenberg \cite{Gr}, Rohrlich \cite{Ro}, \cite{Ro2}, Vatsal \cite{Va} and 
Cornut \cite{Cor} would be special cases.

\begin{conjecture}\label{MC} 

Let $\mathcal{W} = \rho \psi = \rho \chi \circ {\bf{N}}$ be a Hecke character of $K$ of finite order as in $(\ref{factorization})$ above. 
Let $k=1$ if the pair $(f, \mathcal{W})$ is exceptional, or else $k=0$ if the pair $(f, \mathcal{W})$ is generic. If $c(\mathcal{W})$ is sufficiently large, 
then $\delta_{[\mathcal{W}]}^{(k)} \neq 0$. Equivalently, for all but finitely many such $\mathcal{W}$, the value $L^{(k)}(1/2, f \times \mathcal{W})$ 
does not vanish.

\end{conjecture} 

We establish the following results in the direction of this conjecture, starting with estimates for the following larger averages. 
Fix a Hecke character $\mathcal{W} = \rho \psi = \rho \chi \circ {\bf{N}}$ as in $(\ref{factorization})$ above, with $\rho$ a primitive ring 
class character of conductor $p^{\alpha}$ for some integer $\alpha \geq 0$, and $\chi$ a primitive even\footnote{We average over even
characters as the corresponding ``archimedean component" of the Hecke character corresponding to $\chi \circ {\bf{N}}$ is then trivial, 
and so the archimedean components of the $L$-function do not depend on the choice of character. That is, we average over a 
family of wide ray class characters of $K$, which is simpler for the approximate functional equation below.} Dirichlet character of conductor 
$p^{\beta}$ for some integer $\beta \geq 0$. Writing $k \in \lbrace 0, 1 \rbrace$ as above to denote the condition on the generic root number
$\epsilon(1/2, f \times \mathcal{W})$ for the pair $(f, \mathcal{W})$, we first consider the weighted average of $L$-values 
\begin{align}\label{harmonic} H^{(k)}(\alpha, \beta) &= \frac{2}{\#\operatorname{Pic}(\mathcal{O}_{p^{\alpha}}) \varphi^{\star}(p^{\beta})} 
\sum_{\rho \in \operatorname{Pic}(\mathcal{O}_{p^{\alpha}})^{\vee}} 
\sum\limits_{ \chi \operatorname{mod} p^{\beta} \atop \chi(-1)=1, \operatorname{primitive} } L^{(k)}(1/2, f \times \rho \chi \circ {\bf{N}}). \end{align}
Here, $\operatorname{Pic}(\mathcal{O}_{p^{\alpha}})$ denotes the class group of the order 
$\mathcal{O}_{p^{\alpha}} = {\bf{Z}} + p^{\alpha} \mathcal{O}_K$ of conductor $p^{\alpha}$ in $\mathcal{O}_K$, so
that the character group of $\operatorname{Pic}(\mathcal{O}_{p^{\alpha}})^{\vee}$ of $\operatorname{Pic}(\mathcal{O}_{p^{\alpha}})$ 
contains all of the ring class class characters $\rho$ of conductor $p^{\alpha}$. We also write $\varphi^{\star}(p^{\beta})$ to denote the 
number of primitive Dirichler characters $\chi \operatorname{mod} p^{\beta}$, where the second sum is taken over such characters. 
Given $s \in {\bf{C}}$ (first with $\Re(s) > 1$), we can then define the symmetric square $L$-functions of $f$ by the Dirichlet series 
\begin{align}\label{symm} L(s, \operatorname{Sym^2} f) &= \zeta(2s) \sum_{n \geq 1} \lambda(n^2) n^{-s} \end{align} 
and also that for the corresponding twist by a Dirichlet character $\chi$, 
\begin{align*} L(s, \operatorname{Sym^2} f \otimes \chi) &= L(2s, \chi^2) \sum_{n \geq 1} \lambda(n^2) \chi(n^2) n^{-s} \end{align*} 
These $L$-functions have well-known analytic continuations to the complex plane, with the Euler factors at ramified places described 
by Shimura. We refer to \cite{CM} and \cite{GHL} for more on the analytic properties and bounds satisfied by these $L$-functions. 
Let \begin{align*}\gamma = \lim_{x \rightarrow \infty} \left( \sum_{x \leq n} \frac{1}{n} - \log x \right)\end{align*} 
denote the Euler-Mascheroni constant. Given an integer $M \geq 1$ and an $L$-function $L(s)$, let us also write $L^{(M)}(s)$ to denote 
the $L$-function determined by $L(s)$ after removal of the Euler factors at primes dividing $M$. Let $w = w_K$ denote the number 
of roots of unity in $K$. We first prove the following result. And finally, let us write $q_1 = q_1(x, y)$ to denote the binary quadratic form 
associated to the principal class in $\operatorname{Pic}(\mathcal{O}_K)$. Explicitly, 
\begin{align}\label{qf} q_1(x, y) &= \begin{cases} x^2 - \frac{ D }{4} y^2 &\text{ if $D \equiv 0 \bmod 4 $} \\
x^2 + xy + \left( \frac{1-D}{4} \right) y^2 &\text{ if $D \equiv 1 \bmod 4$}.\end{cases} \end{align}

\begin{theorem}\label{h} 

Let $f$ be a holomorphic eigenform of even weight $l \geq 2$, squarefree level $N$ and trivial nebentype character.
Fix $K$ an imaginary quadratic field of discriminant $D<0$ and quadratic Dirichlet character $\omega = \omega_D$.    
Let $p$ be an odd prime, and assume that $(p, N D) = (N, D) = 1$. Fix integers $\alpha \geq 0$ and $\beta \geq 4$. 
We prove the following estimates in either case on the generic root number $k \in \lbrace 0, 1 \rbrace$: \\

\begin{itemize}

\item[(i)](Theorem \ref{SDave} (i)). If $k= 0$ and $f$ is not dihedral, then for some constant $\eta_0 >0$ we have that  
\begin{align*} H^{(0)}(\alpha, 0) &=  \frac{4}{w} \cdot L(1, \omega) \cdot \frac{L^{(p^{\alpha})}(1, \operatorname{Sym^2} f)}{\zeta^{(p^{\alpha})} (2)} 
+ O \left( (N \vert D \vert p^{2 \alpha})^{-\eta_0} \right). \end{align*}
In particular, if $\alpha$ (or even $N \vert D \vert p^{2 \alpha}$) is sufficiently large, then the average $H^{(0)}(\alpha, 0)$ does not vanish. \\

\item[(ii)](Theorem \ref{SDave} (ii)). If $k=1$ and $f$ is not dihedral, then for some constant $\eta_1 >0$ we have that 
\begin{align*} H^{(1)}(\alpha, 0) =  \frac{4}{w} \cdot L(1, \omega) &\cdot \frac{L^{(p^{\alpha})}(1, \operatorname{Sym^2} f )}{\zeta^{(p^{\alpha})}(2)} 
\left[ \frac{1}{2}\log(N\vert D \vert p^{2 \alpha}) + \frac{L'}{L}(1, \omega)  \right. \\  & \left. 
+ \frac{L'^{(p^{\alpha})}}{L^{(p^{\alpha})}}(1, \operatorname{Sym^2} f) - \frac{\zeta'^{(p^{\alpha})}}{\zeta^{(p^{\alpha})}}(2) 
- \gamma- \log 2 \pi \right] + O((N \vert D \vert p^{2 \alpha})^{-\eta_1} ). \end{align*}
In particular, if $\alpha$ (or even $N \vert D \vert p^{2 \alpha}$) is sufficiently large, then the average $H^{(1)}(\alpha, 0)$ does not vanish. \\

\item[(iii)](Theorem \ref{D1tilde}, Proposition \ref{D1}, Theorem \ref{cycharm} (i)). If $k =0$ with $\alpha \geq 0$ and $\beta \geq 4$, then we have  
\begin{align*} H^{(0)}(\alpha, \beta) &= \frac{2}{w} \left( 1 + \frac{\lambda(q_1(0,1))}{q_1(0,1)^{\frac{1}{2}}} \right)
\cdot \frac{2}{\varphi^{\star}(p^{\beta})} \sum_{\chi \operatorname{mod} p^{\beta} \atop \chi(-1)=1 \operatorname{primitive}} 
L(1, \omega \chi^2) \cdot \frac{ L(1, \operatorname{Sym^2} f \otimes \chi^2) }{L(2, \chi^2)} \\
&+ O_{\varepsilon} \left( (N \vert D \vert p^{2 \alpha})^{\frac{1}{2} + \varepsilon} (p^{\beta})^{-\frac{1}{16} + 30 \varepsilon} \right). \end{align*}
In particular, the average $H^{(0)}(\alpha, \beta)$ does not vanish if the exponent $\beta \geq 4$ is sufficiently large. 

\end{itemize}

\end{theorem} 

\begin{remark}[Two remarks.] (i) It is easy to deduce from Theorem \ref{h} that $H^{(k)}(\alpha, 0) \neq 0$ for $\alpha$ sufficiently large.
Here, we use that the symmetric square values $L^{(N)}(1, \operatorname{Sym^2} f \otimes \chi)$ are nonvanishing, 
and bounded below by $1$ in the style of Goldfeld-Hoffstein-Lieman \cite{GHL}, cf.~\cite[Lemma 4.2]{CM}.
Since $\omega$ is an odd Dirichlet character, we can also use Colmez \cite[Proposition 5]{Col} to derive a lower bound
\begin{align*} \frac{1}{2}\log (N \vert D \vert p^{2(\alpha + \beta)}) + \frac{L'}{L}(1, \omega)  \gg \frac{1}{2} 
\log \left( N \vert D \vert p^{2(\alpha + \beta)} \right) \gg \log p^{\alpha} + O_{f, D}(1). \end{align*} 
In the more general setting of $H^{(0)}(\alpha, \beta)$, we give a contour argument in Theorem \ref{cycharm}
to deduce that the corresponding residual term in the estimate stated above converges to a nonzero constant with $\beta$. 
(ii) The condition $\beta \geq 4$ is made to simplify the explicit evaluation of hyper-Kloosterman sums of dimension $4$ 
in the style of Sali\'e (Proposition \ref{salie}), which we then use to relate to exponential sums of $p$-adic phase which can be bounded using the 
Weyl differencing method (see Proposition \ref{DtildeWS} and Theorem \ref{D1tilde}). \end{remark}

Using the same ideas that appear in the proofs of the Theorems \ref{SDave} and \ref{cycharm} (especially Theorem \ref{SDerror}) below, 
we also obtain the following refinement of Theorem \ref{h} for the Galois averages $\delta_{[\mathcal{W}]}^{(k)}$ introduced in Conjecture 
\ref{MC} above. To show the main estimate that we derive, let us retain the setup leading to the definition $(\ref{harmonic})$ of the averages 
$H^{(k)}(\alpha, \beta)$ above. Fixing $\alpha \geq 1$, let $1 \leq x \leq \alpha$ be an integer, and consider the subaverages defined over all 
ring class characters $\rho$ of exact order $p^x$: 
\begin{align*} \Gamma^{(k)}(\alpha(x)) 
&= \frac{1}{\varphi(p^x)} \sum_{\rho \in \operatorname{Pic}(\mathcal{O}_{p^{\alpha}})^{\vee} \atop \rho^{p^x} ={\bf{1}} } 
L^{(k)}(1/2, f \times \rho) \end{align*} in the self-dual setting (with $\beta =0$), and more generally 
\begin{align*} \Gamma(\alpha(x), \beta) = \Gamma^{(0)}(\alpha(x), \beta) &= 
\frac{2}{\varphi(p^x) \varphi^{\star}(p^{\beta})} \sum_{\rho \in \operatorname{Pic}(\mathcal{O}_{p^{\alpha}})^{\vee} \atop \rho^{p^x} ={\bf{1}} } 
\sum_{\chi \bmod \chi(-1)=1, p^{\beta} \atop \operatorname{primitive}} L^{(k)}(1/2, f \times \rho \chi \circ {\bf{N}}) \end{align*}
in the generic, non-self-dual setting (with $\beta \geq 4$ and $k=0$). Let us write 
\begin{align*} G(\alpha)^{p^x} = \operatorname{Pic}(\mathcal{O}_{p^{\alpha}})^{p^x} 
&= \left\lbrace A^{p^x}: A \in \operatorname{Pic}(\mathcal{O}_{p^{\alpha}}) \right\rbrace
\end{align*} to denote the subgroup of $p^x$-th powers. Given a class $A \in \operatorname{Pic}(\mathcal{O}_{p^{\alpha}})$,
let us also write $q_A(x, y)$ to denote the corresponding positive definite quadratic form, i.e.~as used to define theta series whose 
Mellin transform gives the associated partial Dedekind $L$-function for the ring class extension of conductor $p^{\alpha}$ of $K$.
 
\begin{theorem} Keep the setup of Theorem \ref{h} above. Fix integers $\alpha \geq 1$, $1 \leq x \leq \alpha$, and $\beta \geq 4$. \\

\begin{itemize}

\item[(i)](Theorem \ref{SDGA} (i)). If $k= 0$, then for some constant $\eta_0 >0$ we have  
\begin{align*} \Gamma^{(0)}(\alpha(x)) &=  \frac{4}{w} \cdot L(1, \omega) \cdot \frac{L^{(p^{\alpha})}(1, \operatorname{Sym^2} f)}{\zeta^{(p^{\alpha})} (2)} 
+ O \left( (N \vert D \vert p^{2 \alpha})^{-\eta_0} \right). \end{align*}
In particular, if $\alpha \geq 1$ is sufficiently large, then the average $\Gamma^{(0)}(\alpha(x))$ does not vanish. \\

\item[(ii)](Theorem \ref{SDGA} (ii)). If $k=1$, then for some constant $\eta_1 >0$ we have 
\begin{align*} \Gamma^{(1)}(\alpha(x)) =  \frac{4}{w} \cdot L(1, \omega) &\cdot \frac{L^{(p^{\alpha})}(1, \operatorname{Sym^2} f )}{\zeta^{(p^{\alpha})}(2)} 
\left[ \frac{1}{2}\log(N\vert D \vert p^{2 \alpha}) + \frac{L'}{L}(1, \omega)  \right. \\  & \left. 
+ \frac{L'^{(p^{\alpha})}}{L^{(p^{\alpha})}}(1, \operatorname{Sym^2} f) - \frac{\zeta'^{(p^{\alpha})}}{\zeta^{(p^{\alpha})}}(2) 
- \gamma- \log 2 \pi \right] + O((N \vert D \vert p^{2 \alpha})^{-\eta_1} ). \end{align*}
In particular, if $\alpha \geq 1$ is sufficiently large, then the average $\Gamma^{(1)}(\alpha(x))$ does not vanish. \\

\item[(iii)](Theorem \ref{NSDGA}). If $k =0$ with $\alpha \geq 1$ and $\beta \geq 4$, then we have 
\begin{align*} \Gamma(\alpha(x), \beta) &= \frac{2}{w} \prod_{A \in G(\alpha)^{p^x}} \left( 1 + \frac{\lambda(q_A(0,1))}{q_A(0,1)^{\frac{1}{2}}} \right)
\cdot \frac{2}{\varphi^{\star}(p^{\beta})} 
\sum_{\chi \operatorname{mod} p^{\beta} \atop \operatorname{primitive}} 
L(1, \omega \chi^2) \cdot \frac{ L(1, \operatorname{Sym^2} f \otimes \chi^2) }{L(2, \chi^2)}  \\
&+ O_{\varepsilon} \left( \varphi(p^x) (N \vert D \vert p^{2 \alpha})^{\frac{1}{2} + \varepsilon} (p^{\beta})^{- \frac{1}{16} + 30 \varepsilon} \right). \end{align*}
In particular, the average $\Gamma(\alpha(x), \beta)$ does not vanish if $\beta \geq 4$ is sufficiently large. 

\end{itemize} \end{theorem} 

This goes some way towards showing Conjecture \ref{MC}. In particular, 
we obtain the following immediate application of Shimura's algebraicity theorem \cite{Sh}
in the case of central values (corresponding to $k=0$), together with the central derivative value formulae of Gross-Zagier \cite{GZ}, Zhang \cite{Zh}, and 
Yuan-Zhang-Zahn \cite{YZZ} in the case of central derivative values (corresponding to our exceptional case of $k=1$). 

\begin{theorem} Let $x \geq 0$ be any integer. Let $\rho$ be a primitive ring class character of conductor 
$p^{\alpha}$ for some integer $\alpha \geq 0$ and exact order $p^x$. Assume that the absolute discriminant $\vert D \vert$ of $K$ is sufficiently large. 

\begin{itemize}

\item[(i)] There exists a constant $x_0 \geq 0$
such that for all ring class characters $\rho$ of $K$ of exact order $p^x$ with $x \geq x_0$, the value $L^{(k)}(1/2, f \times \rho)$ does not vanish. 

\item[(ii)] There exists a constant $\beta_0 = \beta_0(x) \geq 4$ depending on $x$ such that 
for each $\beta \geq \beta_0$, the Galois average $\delta^{(0)}_{[\rho \chi \circ {\bf{N}}]}$ does not vanish.
Hence (by Shimura's theorem), for each primitive Dirichlet character $\chi \operatorname{mod} p^{\beta}$ with $\beta \geq \beta_0$, 
the central value $L(1/2, f \times \rho \chi \circ {\bf{N}})$ does not vanish. \end{itemize} \end{theorem}

Again, we remark that the artificial-looking constraint on $\beta$ in the second claim above stems from the fact that we evaluate hyper-Kloosterman 
sums in the style of Sali\'e to derive our our estimates in this setting, and is therefore natural for this approach. However, it is also possible to develop
the approach we take via spectral decompositions of shifted convolution sums for the self-dual estimates leading to (i) to deal with nontrivial nebentype
character (to estimate central values, but not central derivative values), as well as to remove the assumption that $\vert D \vert \gg 1$. 
This latter approach is developed in the more general companion work \cite{VO7}. 
Finally, we explain the natural applications of these estimates to Iwasawa main conjectures and Mordell-Weil ranks in the sequel note \cite{VO6}.

\subsubsection*{Acknowledgements} 

I am grateful to Philippe Michel, as well as to Valentin Blomer, Farrell Brumley, Gergely Harcos, Djordje Milicevic, Paul Nelson, David Rohrlich, 
Peter Sarnak, and Nicolas Templier for various helpful discussions related to this work (at various points). I am also grateful to Christophe 
Cornut and Nike Vatsal for encouragement, as well as to John Coates and Barry Mazur for their interest. Finally, I should like to thank the anonymous 
referees for helpful constructive criticism (and pointing out of gaps) in previous versions of this work, as well as to apologize for the long delay in its completion. 
 
\subsubsection*{Notations and conventions} 

Throughout, we shall fix an imaginary quadratic field $K$ of discriminant $D < 0$ and ring of integers $\mathcal{O}_K$. 
Given an integer $\alpha \geq 0$, we shall write $\mathcal{O}_{p^{\alpha}} = {\bf{Z}} + p^{\alpha} \mathcal{O}_K$ to denote
the ${\bf{Z}}$-order of conductor $p^{\alpha}$ in $K$, with $\operatorname{Pic}(\mathcal{O}_{p^{\alpha}})$ its class group. 
The notation $\rho$ is reserved to denote a ring class character in this setup, so some character of the finite group 
$\operatorname{Pic}(\mathcal{O}_{p^{\alpha}})$ for some $\alpha \geq 0$. We shall also use the notation $\chi$ to 
denote a primitive Dirichlet character of conductor $p^{\beta}$ for some integer $\beta \geq 1$, with $\varphi^{\star}(p^{\beta})$ 
denoting the number of such characters (hence $\varphi^{\star}(p^{\beta}) = (p-1)^2 p^{\beta-2}$ if $\beta \geq 2$). 
Note as well that we write
\begin{align*} \sum_{\rho \in \operatorname{Pic}(\mathcal{O}_{p^{\alpha}})^{\vee}}\end{align*} to denote the sum of all ring 
class characters $\rho$ of $\operatorname{Pic}(\mathcal{O}_{p^{\alpha}})$ in the anticyclotomic direction. 
We also write \begin{align*} \sum_{\chi \bmod p^{\beta} \atop \chi(-1) \operatorname{primitive}}\end{align*} 
to denote the sum over primitive even Dirichlet characters $\chi \operatorname{mod} p^{\beta}$ in the cyclotomic direction, 
noting there that we must in fact take the primitive characters (see Proposition \ref{RT} below). Again, we average over even
Dirichlet characters as these have trivial ``archimedean component", and so do not alter the archimedean component
of the corresponding Rankin-Selberg $L$-function $\Lambda(s, f \times \rho \chi \circ {\bf{N}}) = L_{\infty}(s) L(s, f \times \rho \chi \circ {\bf{N}})$.
 
\section{Some background}

We now give some background on the Rankin-Selberg $L$-functions $L(s, f \times \mathcal{W}) = L(s, f \times \theta(\mathcal{W}))$, 
which we shall take for granted later. Fix a prime number $p$ which is coprime to $D N$, 
and let $\mathcal{W} = \rho \chi \circ {\bf{N}} \in X_{\alpha, \beta}$ denote 
a Hecke character of the form described above, with $\rho$ a ring class character of $K$ of conductor $p^{\alpha}$ for some
integer $\alpha \geq 0$ (hence a class group character if $\alpha = 0$), and $\chi$ a Dirichlet character of
conductor $p^{\beta}$ for some integer $\beta \geq 0$. A classical construction due to Hecke associates to any such character
$\mathcal{W}$ a theta series $\theta(\mathcal{W})$ of weight one, level $\vert D \vert p^{2(\alpha + \beta)}$, 
and nebentype character $\mathcal{W} \vert_{\bf{Q}}^{\times} = \omega \chi^2$.

\subsection{Dirichlet series expansions} 

We consider the Rankin-Selberg $L$-function $L(s, f \times \mathcal{W})$ of $f$ times $\theta(\mathcal{W})$ in the setup described above, 
which for $\Re(s) > 1$ has the Dirichlet series expansion 
\begin{align}\label{idealDirichlet} L(s, f \times \mathcal{W}) &= L^{(N)}(2s, \omega\chi^2) \sum_{\mathfrak{a} \subseteq \mathcal{O}_K} 
\mathcal{W}(\mathfrak{a}) \lambda({\bf{N}}\mathfrak{a}) {\bf{N}}\mathfrak{a}^{-s}. \end{align} 
Here, the sum runs over nonzero integral ideals $\mathfrak{a} \subset \mathcal{O}_K$, with the convention that 
$\mathcal{W}(\mathfrak{a})=0$ for ideals $\mathfrak{a}$ whose image ${\bf{N}}(\mathfrak{a})$ under the norm map ${\bf{N}}$ are not prime to the conductor 
$c(\mathcal{W})$, and $L^{(N)}(s, \omega\chi^2)$ denotes the Dirichlet $L$-function $L(s, \omega\chi^2)$ with Euler factors at primes dividing $N$ removed. 
Since $\rho$ is a ring class character of conductor $p^{\alpha}$, 
we can identify $\rho$ as a character of the class group $\operatorname{Pic}(\mathcal{O}_{p^{\alpha}})$ of the ${\bf{Z}}$-order 
$\mathcal{O}_{p^{\alpha}} = {\bf{Z}} + p^{\alpha} \mathcal{O}_K$ of conductor $p^{\alpha}$ in $K$. 

Given an integer $n \geq 1$ and a class $A \in \operatorname{Pic}(\mathcal{O}_{p^{\alpha}})$, 
let $r_A(n)$ denote the number of ideals of norm $n$ in $A$. 
Expanding out the second sum in $(\ref{idealDirichlet})$ with respect to these counting functions, 
we have the well-known equivalent Dirichlet series expansion over integers
\begin{align}\label{integralDirichlet} L(s, f \times \mathcal{W}) 
&= \sum_{m \geq 1 \atop (m, N)=1} \frac{\omega\chi^2(m)}{m^{2s}} \sum_{n \geq 1} \left( \sum_{A} \rho(A)r_A(n) \right) \frac{\lambda(n)\chi(n)}{n^{s}} \end{align} 
where the third sum runs over all classes $A \in \operatorname{Pic}(\mathcal{O}_{p^{\alpha}})$ of the order $\mathcal{O}_{p^{\alpha}}$. 
Note that to justify the equivalence of $(\ref{idealDirichlet})$ and $(\ref{integralDirichlet})$ properly, 
one can compare Euler factors in each of the corresponding completed $L$-functions 
(the former coming from an $L$-function of degree $2$ over $K$, and the latter an $L$-function of degree $4$ over ${\bf{Q}}$).

Let us now consider the counting functions $r_A(n)$ for $A \in \operatorname{Pic}(\mathcal{O}_{p^{\alpha}})$, 
with emphasis on the counting function $r_{\bf{1}}(n)$ corresponding to the trivial class ${\bf{1}} \in \operatorname{Pic}(\mathcal{O}_{p^{\alpha}})$. 
Observe that whenever $\alpha \beta \geq 1$ in our setup above, the Dirichlet series expansion $(\ref{idealDirichlet})$ is given by a sum over 
nonzero ideals $\mathfrak{a} \subset \mathcal{O}_K$ of absolute norm is prime to $p$. Now it is a classically-known fact, 
as explained in \cite[Proposition 7.20]{Cox} for instance, that $r_{{\bf{1}}}(n) = r_1(n)$ for any integer $n$ prime to $p$, 
where $r_1(n)$ denotes the number of principal ideals in the class group $\operatorname{Pic}(\mathcal{O}_K)$ of $K$ of norm equal to $n$. 
(An analogous statement holds for any $A \in \operatorname{Pic}(\mathcal{O}_{p^{\alpha}})$).
This in particular means that we shall only have to consider the counting function $r_1(n)$ associated to the maximal order $\mathcal{O}_K$ later on. 
Let us lighten notations by writing $r(n) = r_1(n) = r_{\bf{1}}(n)$ to denote this counting function (defined on integers $n \geq 1$ prime to $p$). 
Then, the counting function $r(n)$ can be parametrized in terms of the representations of $n$ by $q_1$ as
\begin{align}\label{count} r(n) &= \frac{1}{w} \cdot \# \left\lbrace (a, b) \in {\bf{Z}}^2: ~ q_1(a, b)= n \right\rbrace, \end{align} 
where $w$ denotes the number of automorphs of $q_1$, equivalently the number of roots of unity in $K$.

\begin{remark} Observe that we could also parametrize $r(n)$ after fixing an integral basis $[1, \vartheta]$ of $\mathcal{O}_K$ as  
\begin{align*} r(n) &= \frac{1}{w} \cdot \# \left\lbrace (a, b) \in {\bf{Z}}^2: ~ a^2 - b^2 \vartheta^2 = n \right\rbrace. \end{align*} 
Although it is typically harmless, one often finds a standard abuse of notation in the literature where 
$\vartheta^2$ is simply replaced with the fundamental discriminant $D$ (e.g.~\cite[$\S 7$]{Te}, \cite{TeT}). 
That is, one can always expand out in terms of a basis this way, then use that $\vartheta^2 = O(D)$ to justify the (technically incorrect) definition  
\begin{align*} r(n) &= \frac{1}{w} \cdot \# \left\lbrace (a, b) \in {\bf{Z}}^2: ~ a^2 - b^2 D = n \right\rbrace \end{align*} 
for various estimates. Either way, this counting function plays an important role in our subsequent estimates, 
allowing us to use spectral decompositions of shifted convolution sums to estimate our off-diagonal terms. \end{remark}

\subsection{Functional equations} 

Keeping with the setup above, we shall take for granted the following important result throughout. 
The result is attributed to Rankin (and Selberg) and Shimura, but the exact form of the root number 
has only been established more recently thanks to the works of \cite{JL}, \cite{Ja}, and \cite{Li}.

\begin{proposition}\label{FE} 

Fix $f$ a normalized newform of weight $2$, level $N$, and trivial nebentype character. 
Fix $K$ an imaginary quadratic field of discriminant $-D$ prime to $N$ and character $\omega = \omega_D$.
Fix a prime $p$ coprime to $ND$, and let $\mathcal{W} = \rho \chi \circ {\bf{N}} \in X_{\alpha, \beta}$ be a Hecke character of 
$K$ of the form described above, i.e.~with $\rho$ a ring class character of conductor $p^{\alpha}$ and $\chi$ a Dirichlet character
of conductor $p^{\beta}$ (with $\alpha, \beta \in {\bf{Z}}_{\geq 0}$ and $\alpha \beta \geq 1$). 
The Rankin-Selberg $L$-function $L(s, f \times \mathcal{W}) = L(s, f \times \theta(\mathcal{W}))$ of $f$ times 
$\theta(\mathcal{W}) \in S_1(\Delta, \xi) = S_1(\vert D \vert p^{2(\alpha + \beta)}, \omega \chi^2)$ has an analytic continuation to ${\bf{C}}$, 
and its completed $L$-function 
\begin{align*} \Lambda(s, f \times \mathcal{W}) &:= (N\Delta)^s  \Gamma_{\bf{R}}(s+1/2)\Gamma_{\bf{R}}(s + 3/2) L(s, f \times \mathcal{W}), 
 ~~~~~\Gamma_{\bf{R}}(s) := \pi^{- \frac{s}{2}} \Gamma \left( \frac{s}{2} \right) \end{align*} 
satisfies the functional equation 
\begin{align*} \Lambda(s, f \times \mathcal{W}) &= \epsilon(1/2, f \times \mathcal{W}) \Lambda(1-s, f \times \overline{\mathcal{W}}).\end{align*} 
Here, the root number $\epsilon(1/2, f \times \mathcal{W}) \in {\bf{S}}^1$ is given explicitly by  the formula
\begin{align*} \epsilon(1/2, f \times \mathcal{W}) &= \omega\chi^2(- N) \left( \frac{\tau(\omega \chi^2)^2}{\vert D \vert p^{\beta}} \right)^2,\end{align*} 
where 
\begin{align*} \tau(\omega\chi^2) &= \sum_{x \operatorname{mod} \vert D \vert p^{\beta}} \omega \chi^2(x) e \left( \frac{x}{ \vert D \vert p^{\beta} }\right)\end{align*} 
denotes the Gauss sum of the Dirichlet character $\omega\chi^2$. 
To be explicit, writing $L_{\infty}(s) := \Gamma_{\bf{R}}(s + \frac{1}{2}) \Gamma_{\bf{R}}(s + \frac{3}{2})$
we have for a given Hecke character $\mathcal{W} = \rho \chi \circ {\bf{N}} \in X_{\alpha, \beta}$ the functional equation
\begin{align}\label{explicitFE} L_{\infty}(s) L(s, f \times \rho \chi \circ {\bf{N}})  &=
\omega \chi^2(-N) \cdot  \frac{\tau(\omega \chi^2)^4}{(\vert D \vert p^{\beta})^2} \cdot (N \vert D \vert p^{2 (\alpha + \beta)})^{1-2s} 
L_{\infty}(1-s) L(1-s, f \times \rho \overline{\chi} \circ {\bf{N}}) . \end{align}
\end{proposition}

\begin{proof} 

See \cite[Theorem 2.2, and Example 2]{Li}. Here, we have used the automorphic normalization, so that $s=1/2$ is the point
of symmetry for the functional equation, and that $\overline{f} = f$. We have also used that $\omega \chi^2$ has conductor 
$\vert D \vert p^{\beta}$, that the Hecke $L$-function $L(s, \mathcal{W}) = L(s, \theta(\mathcal{W}))$ has root number 
$\tau(\omega \chi^2)^2/(\vert D \vert p^{\beta})$, and that ring class characters are equivariant under complex conjugation 
(cf.~the discussion in \cite{Ro2}). \end{proof}
 
\subsection{Approximate functional equations} 

Fix $k \in \lbrace 0, 1 \rbrace$. We now use an approximate functional equation argument to derive an expression for the values 
$L^{(k)}(1/2, f \times \mathcal{W})$, i.e.~outside of the range or absolute convergence for the corresponding Dirichlet series.
Suppose more generally that $f$ is any modular form of level $N$, and that $\mathcal{W}$ is any Hecke character of $K$ 
with theta series $\theta(\mathcal{W})$ of level $\Delta$ prime to $N$. Given $s \in {\bf{C}}$ with $\Re(s) > 0$, 
let us momentarily lighten notations by writing
\begin{align}\label{basicDirichlet} L(s, f \times \mathcal{W})  & = \sum_{n \geq 1} \frac{a_{f \times \mathcal{W}}(n)}{n^s} \end{align} 
to denote the Dirichlet series expansion of the Rankin-Selberg $L$-function $L(s, f \times \mathcal{W}) = L(s, f \times \theta(\mathcal{W}))$ 
of $f$ times the theta series associated to a Hecke character $\mathcal{W}$. 
Fix $g \in \mathcal{C}_c^{\infty}({\bf{R}}_{>0})$ any smooth and compactly supported function on $y \in {\bf{R}}_{>0}$ 
whose Mellin transform $g(s) := \int_0^{\infty} g(y) y^{s} \frac{dy}{y}$ satisfies the condition $g^*(0) = 1$. 
Let us for a given integer $m \geq 1$, write $G_m$ to denote the meromorphic function defined on a 
complex variable $s$ by $G_m(s) = g^*(s) s^{-m}$. Hence, $G_m(s)$ is then holomorphic except at $s = 0$ 
(where it behaves like $s^{-m}$), and bounded in vertical strips.%\footnote{Note that the parenthetical assertion 
%\cite[$\S 7.2$]{Te} that $G(s) = G_2(s)$ could simply be taken to be $1/s^2$ is not valid: The Stirling approximation 
%theorem would then not allow for the derivation of the stated bounds for $y \rightarrow 0$ in \cite[Lemma 7.1]{Te}.} 
Let us also suppose that $g^*(s)$ is even, so that $G_m(-s) = (-1)^m G_m(s)$ for all $s \neq 0$.
Write $V_m \in \mathcal{C}^{\infty}$ to denote the smooth and rapidly decaying function defined on $y \in {\bf{R}}_{>0}$ by 
\begin{align}\label{V} V_m(y) &= \int_{\Re(s)=2} \widehat{V}_m(s)y^{-s}\frac{ds}{2 \pi i }, \end{align}  
where $\widehat{V}_m(s)$ denotes the function defined on $s \in {\bf{C}} - \lbrace 0 \rbrace$ by 
\begin{align}\label{Vhat} \widehat{V}_m(s) &= \frac{L_{\infty}(s + 1/2)}{L_{\infty}(1/2)} G_m(s) =  \pi L_{\infty}(s+1/2) \frac{g^*(s)}{s^m}. \end{align} 
To be clear, recall that we define $L_{\infty}(s) = \Gamma_{\bf{R}} (s + 1/2)\Gamma_{\bf{R}}(s + 3/2)$, and note that $L_{\infty}(1/2) = \pi^{-1}$. 

\begin{lemma}\label{cvformula} 

Fix $k \in \lbrace 0, 1 \rbrace$. Then, for any choice of $Z >0$, the value $ L^{(k)}(1/2, f \times \mathcal{W})$ is given by
\begin{align}\label{value} \sum_{n \geq 1} \frac{a_{f \times \mathcal{W}}(n)}{n^{\frac{1}{2}}} V_{k+1}\left( \frac{ Z n }{ N \Delta } \right) -(-1)^{k+1}
\epsilon(1/2, f \times \mathcal{W}) \sum_{n \geq 1} \frac{a_{f \times \overline{\mathcal{W}}}(n)}{n^{\frac{1}{2}}} V_{k+1}\left( \frac{n}{Z N \Delta} \right).\end{align} 
Note that in the exceptional case of $k=1$, the root number is $\epsilon(1/2, f \times \rho) = \omega(-N) = - \omega(N)$
with $\omega(N)= 1$, and so the formula in our setup is never identically zero. 

\end{lemma}

\begin{proof} See \cite[$\S 5.2$]{IK} (cf.~\cite[$\S 7.2$]{Te}).
Let $m = k+1$, and consider the contour integral
\begin{align*} \int_{\Re(s) = 2} \Lambda(s+1/2, f \times \mathcal{W}) G_m(s) Z^{-s} \frac{ds}{2 \pi i}. \end{align*}
Note that by Stirling's formula, the completed $L$-function $\Lambda(s, f \times \mathcal{W})$ 
decays rapidly as $\Im(s) \rightarrow \infty$, and hence that we are justified in using Cauchy's theorem to shift contours. 
Shifting the contour to $\Re(s) = -2$, we cross a pole at $s = 0$ of residue 
\begin{align*} R:= \operatorname{Res}_{s =0}\left(\Lambda(s+1/2, f \times \mathcal{W})G_m(s) Z^{-s} \right) 
&= \Lambda(1/2, f \times \mathcal{W}) = (N\Delta)^{\frac{1}{2}} L_{\infty}(s) L(1/2, f \times \mathcal{W}), \end{align*} 
which gives us the identification 
\begin{align}\label{cauchy} \int_{\Re(s)=2} \Lambda(s+1/2, f \times \mathcal{W})G_m(s) Z^{-s} \frac{ds}{2\pi i} 
&= R + \int_{\Re(s)=-2} \Lambda(s+1/2, f \times \mathcal{W}) G_m(s)Z^{-s} \frac{ds}{2\pi i}. \end{align}
Now, the integral on the right hand side of $(\ref{cauchy})$ is the same as 
\begin{align*} \int_{\Re(s)=2} \Lambda(-s +1/2 , f \times \mathcal{W}) G_m(-s) Z^s \frac{ds}{2\pi i} .\end{align*} 
Applying the functional equation $(\ref{FE})$ to $\Lambda(-s + 1/2, f \times \mathcal{W})$ in this latter expression, 
and using that $G_m(-s)=(-1)^m G_m(s)$, we then obtain
\begin{align*} (-1)^m \epsilon(1/2, f \times \mathcal{W}) \int_{\Re(s)=2} \Lambda(s+1/2 , f \times \overline{\mathcal{W}})G_m(s)Z^s \frac{ds}{2\pi i} .\end{align*}
Expanding out the absolutely convergent Dirichlet series, it is easy to see that $(\ref{cauchy})$ is equivalent to 
\begin{align*} R &= (N\Delta)^{\frac{1}{2}}  L_{\infty}(1/2) \left( \sum_{n \geq 1} \frac{a_{f \times \mathcal{W}}(n)}{n^{\frac{1}{2}}} V_{m}\left( \frac{Z n }{ N \Delta} \right)
-(-1)^{m} \epsilon(1/2, f \times \mathcal{W}) \sum_{n \geq 1} \frac{ a_{f \times \overline{\mathcal{W}}}(n)}{n^{\frac{1}{2}}} V_{m}\left( \frac{n }{Z N \Delta} \right) \right).\end{align*}
Dividing out by $(N\Delta)^{\frac{1}{2}}L_{\infty}(1/2)$ on each side, we then obtain the stated formula. \end{proof}

The functions appearing in the formula $(\ref{value})$ decay rapidly:

\begin{lemma}\label{7.1} Let $k \in \lbrace 0, 1 \rbrace$ be an integer. Then for each integer $j \geq 0$, we have the estimates
\begin{align*}
V_{k+1}^{(j)}(y) &= \begin{cases} \left( (-1)^k(\log y)^k \right)^{(j)} + O_j(y^{\frac{1}{2} -j}) &\text{when $0 < y \leq 1$} \\
O_{C, j}(y^{-C}) &\text{when $y \geq 1$, for any choice of constant $C >0$.}
\end{cases} \end{align*} \end{lemma}

\begin{proof} 

See \cite[Lemma 7.1]{Te} or \cite[Proposition 5.4]{IK} (for instance). 
To estimate the behaviour as $y \rightarrow 0$ for the first estimate(s), 
we move the line of integration in $(\ref{V})$ to the left, crossing a pole at $s = 0$ of residue 
\begin{align*} \operatorname{Res}_{s=0} \left(\widehat{V}_{k+1}(y) y^{-s}\right)&= \lim_{s \rightarrow 0} \frac{1}{k!} \frac{d^k}{ds^k} \left( \pi L_{\infty}(s) y^{-s} \right). 
\end{align*} 
Note that $\frac{d^k}{ds^k}y^{-s} = (-1)^k y^{-s}(\log y)^k$. Using Stirling's formula to estimate the remaining integral, we derive the stated bound(s). 
To estimate the behaviour of as $y \rightarrow \infty$, we move the line of integration right to $\Re(s) =C$ to obtain the second estimate(s). \end{proof} 

\subsection{Derivation of formulae} 

Fix a Hecke character $\mathcal{W} = \rho \chi \circ {\bf{N}}$ of $K$ as in $(\ref{factorization})$, 
with $\rho$ a primitive ring class character of some conductor $p^{\alpha}$, and $\chi$ a primitive even 
Dirichlet character of some conductor $p^{\beta}$ for integers $\alpha, \beta \geq 0$ (with $\alpha \beta \geq 1$). 
Let us write $X_{\alpha, \beta}$ to denote the set of all such characters of the form $\mathcal{W}' = \rho' \chi \circ {\bf{N}}$, 
where $\rho'$ is a primitive ring class character of some conductor $p^x$ with $0 \leq x \leq \alpha$, and 
$\chi$ is a primitive even Dirichlet  character of conductor $p^{\beta}$. 
Recall too for $k \in \lbrace 0, 1 \rbrace$, we define  
\begin{align*} H^{(k)}(\alpha, \beta) = h_{\alpha, \beta}^{-1} \sum_{\mathcal{W} \in X_{\alpha, \beta}} L^{(k)}(1/2, \mathcal{W}) 
&= \frac{2}{ \# \operatorname{Pic}(\mathcal{O}_{p^{\alpha}}) \varphi^{\star}(p^{\beta})} 
\sum_{\rho \in \operatorname{Pic}(\mathcal{O}_{p^{\alpha}})^{\vee}} \sum_{\chi \operatorname{mod} p^{\beta} \atop \operatorname{primitive}} 
L^{(k)}(1/2, f \times \rho \chi \circ {\bf{N}}),\end{align*} 
where $h_{\alpha, \beta} = \# \operatorname{Pic}(\mathcal{O}_{p^{\alpha}}) \varphi^{\star}(p^{\beta})/2 = \# X_{\alpha, \beta}$ 
denotes the cardinality of the set of characters $X_{\alpha, \beta}$. 
We now use $(\ref{value})$ to derive formulae for these averages in Proposition \ref{haf} below.
Given a primitive ring class character $\rho$ of some conductor $p^{\alpha}$, we also derive a formula for the related average
\begin{align*} C(\rho, \beta) &= \frac{2}{\varphi^{\star}(p^{\beta})} \sum_{\chi \operatorname{mod} p^{\beta} \atop \operatorname{primitive}} 
L(1/2, f \times \rho \chi \circ {\bf{N}}) \end{align*} 
in Proposition \ref{cycformula}.

\subsection{Partial summation decomposition} 

Recall that we write $\Delta = \vert D \vert p^{2(\alpha + \beta)}$ 
to denote the level of the theta series $\theta(\mathcal{W}) \in S_1(\Delta, \mathcal{W} \vert_{{\bf{Q}}^{\times}})$ 
associated to the Hecke character $\mathcal{W}$ of $K$ of $(\ref{factorization})$ above.
In what follows, the imaginary quadratic field $K$ (and hence the discriminant $D<0$) is fixed, 
and the exponents $\alpha$ and $\beta$ in the conductor term $p^{2 (\alpha+ \beta)}$ vary.

Let $P_{\alpha, \beta}$ denote the set of characters of the form $\rho \psi$, where $\rho$ is a primitive 
ring class character of conductor $p^{\alpha}$, and $\psi = \chi \circ {\bf{N}}$ for $\chi$ a primitive even 
Dirichlet character of conductor $p^{\beta}$. 
We can then decompose the set of characters $X_{\alpha, \beta}$ defined above into a disjoint union, 
\begin{align*} X_{\alpha, \beta} & = \bigcup_{0 \leq x \leq \alpha} P_{x, \beta}. \end{align*} 
This allows us to decompose the sums defining the average $H^{(k)}(\alpha, \beta)$ as 
\begin{align*} \sum_{\mathcal{W} \in X_{\alpha, \beta}} L^{(k)}(1/2, f \times \mathcal{W}) &= \sum_{0 \leq x \leq \alpha} 
\sum_{\mathcal{W}' \in P_{x, \beta} \atop \mathcal{W}' = \rho'  \chi \circ {\bf{N}}} L^{(k)}(1/2, f \times \rho' \chi \circ {\bf{N}} ). \end{align*} 
Let us continue to use the shorthand notation $(\ref{basicDirichlet})$ for the Dirichlet series expansion.  
We then have by $(\ref{value})$ the expression $L^{(k)}(1/2, f \times \mathcal{W}) = \sum_{1, Z} + \sum_{2, Z}$ for any choice of 
real parameter $Z >0$, where 
\begin{align*} \sum_{1, Z}
&:= \sum_{n \geq 1} \frac{a_{f \times \mathcal{W} }(n)}{n^{\frac{1}{2}}} V_{k+1}\left( \frac{Z n}{ N \vert D \vert p^{2(\alpha + \beta)}} \right), \\ 
\sum_{2, Z}
&:=  (-1)^{k} \epsilon(1/2, f \times \mathcal{W}) \sum_{n \geq 1} \frac{a_{f \times \overline{\mathcal{W}}}(n)}{n^{\frac{1}{2}}} 
V_{k+1}\left( \frac{n}{Z N \vert D \vert p^{2(\alpha + \beta)}} \right). \end{align*} 
Let us also write 
\begin{align*} \mathfrak{K}_{\mathcal{W} \vert_{{\bf{Q}}^{\times}}} 
&= \frac{ \tau(\omega \chi^2)^4}{(\vert D \vert p^{\beta})^2} = \frac{1}{( \vert D \vert p^{\beta})^2} 
\sum_{z_1, \cdots z_2 \operatorname{mod} \vert D \vert p^{\beta}} \omega \chi^2(z_1 \cdots z_4) e \left(\frac{z_1+ \cdots + z_4}{ \vert D \vert p^{\beta}} \right)\end{align*} 
to lighten notations, so that $-(-1)^{k +1} \epsilon(1/2, f \times \mathcal{W}) = -(-1)^{k+1} \omega\chi^2(-N) \mathfrak{K}_{\mathcal{W} \vert_{{\bf{Q}}^{\times}}}  
= (-1)^{k+1} \omega\chi^2(N) \mathfrak{K}_{\mathcal{W} \vert_{{\bf{Q}}^{\times}}} $ by the description of the root number given in $(\ref{explicitFE})$ above. 

We first consider the sum over characters $\mathcal{W} \in X_{\alpha, \beta}$ of $\sum_{1, Z}$, 
\begin{align*} \sum_{\mathcal{W} \in X_{\alpha, \beta}} \sum_{1, Z} &= \sum_{0 \leq x \leq \alpha} 
\sum_{\mathcal{W}' \in P_{x, \beta}} \sum_{n \geq 1} \frac{a_{f \times \mathcal{W}'}(n)}{n^{\frac{1}{2}}} 
V_{k+1}\left( \frac{Z n}{N\vert D \vert p^{2(x+\beta)}}\right), \end{align*} 
which after interchange of summation is given by the expression 
\begin{align}\label{exp} \sum_{n \geq 1}\frac{1}{n^{\frac{1}{2}}} 
\sum_{0 \leq x \leq \alpha} \sum_{\mathcal{W}' \in P_{x, \beta}} a_{f \times \mathcal{W}'}(n) V_{k+1}\left( \frac{Z n}{N \vert D \vert p^{2(x+ \beta )}}\right).\end{align} 
To evaluate this expression when $\alpha \geq 1$ (i.e.~when $x$ can vary), we use partial summation as follows: 

\begin{lemma}[Partial summation]\label{ps2v} 

Fix integers $\alpha, \beta \geq 0$. Assume that $\alpha \geq 1$. Let us for each choice of $k \in \lbrace 0, 1 \rbrace$ 
define a modified cutoff function $\mathfrak{V}_{k+1}$ on $y \in {\bf{R}}_{>0}$ by the contour integral
\begin{align*} \mathfrak{V}_{k+1}(y) &= \int_{\Re(s) = 2} \widehat{V}_{k+1}(s) y^{-s} \left( p^s-1 \right) \frac{ds}{2 \pi i}. \end{align*}
Keeping with the notations introduced above, we have that 
\begin{align*}
\sum_{\mathcal{W} \in X_{\alpha, \beta}} \sum_{1, Z} &:= \sum_{1 \leq x \leq \alpha} \sum_{\mathcal{W} \in P_{x, \beta}} 
\sum_{n \geq 1} \frac{a_{f \times \mathcal{W} }(n)}{n^{\frac{1}{2}}} V_{k+1} \left( \frac{Z n}{N \vert D \vert p^{2(x + \beta)}} \right) \\ 
= &\sum_{n \geq 1} \left( \sum_{\mathcal{W} \in X_{\alpha, \beta}} \frac{a_{f \times \mathcal{W}}(n)}{n^{\frac{1}{2}}} \right) 
V_{k+1} \left( \frac{Z n}{N \vert D \vert p^{2(\alpha + \beta)}}\right) \\
&- \sum_{n \geq 1} \left( \sum_{1 \leq x \leq \alpha - 1} \sum_{\mathcal{W}' \in P_{x, \beta}} \frac{a_{f \times \mathcal{W}'}(n)}{n^{\frac{1}{2}}} \right) 
\mathfrak{V}_{k+1} \left( \frac{Z n}{N \vert D \vert p^{2(x + \beta)}} \right), \end{align*} 
and that 
\begin{align*} \sum_{ \mathcal{W} \in X_{\alpha, \beta}} \sum_{2, Z} 
&:= \sum_{1 \leq x \leq \alpha} \sum_{\mathcal{W} \in P_{x, \beta}} (-1)^k \epsilon(1/2, f \times \mathcal{W})
\sum_{n \geq 1} \frac{a_{f \times \overline{W}}(n)}{n^{\frac{1}{2}}} V_{k+1} \left( \frac{n}{Z N \vert D \vert p^{2(x + \beta)}} \right) \\ 
= &\sum_{n \geq 1} \left( \sum_{\mathcal{W} \in X_{\alpha, \beta}} (-1)^k \epsilon(1/2, f \times \mathcal{W}) 
\frac{a_{f \times \overline{\mathcal{W}}}(n)}{n^{\frac{1}{2}}} \right) V_{k+1} \left( \frac{n}{Z N \vert D \vert p^{2(\alpha + \beta)}} \right) \\ 
&- \sum_{n \geq 1} \left( \sum_{1 \leq x \leq \alpha - 1} \sum_{\mathcal{W}' \in P_{x, \beta}} (-1)^k \epsilon(1/2, f \times \mathcal{W}') 
\frac{ a_{f \times \overline{\mathcal{W}}'(n)} }{n^{\frac{1}{2}}} \right) \mathfrak{V}_{k+1} \left( \frac{n}{Z N \vert D \vert p^{2(x + \beta)}} \right). \end{align*}
\end{lemma} 

\begin{proof} 

Let us first recall the following simple (discrete) version of the partial summation formula. 
Given a sequence of complex numbers $\lbrace a_x \rbrace_{x \geq 1}$, let us write $A(\alpha) = \sum_{1 \leq x \leq \alpha} a_x$
to denote the partial sum up to an integer $\alpha \geq 1$. Let $h(x)$ be any continuously differentiable function. It is easy to see that 
\begin{align} \label{ps} \sum_{1 \leq x \leq \alpha} a_x h(x) &= \sum_{1 \leq x \leq \alpha} \left( A(x) - A(x-1) \right) h(x) 
= A(\alpha) h(\alpha) - \sum_{1 \leq x \leq \alpha - 1} A(x) \left( h(x) - h(x+1) \right). \end{align}
To apply this to the first sum, we first change the order of summation to obtain 
\begin{align*} \sum_{n \geq 1} \sum_{1 \leq x \leq \alpha} \sum_{\mathcal{W} \in P_{x, \beta}}  
\frac{a_{f \times \mathcal{W}}(n)}{n^{\frac{1}{2}}} V_{k+1} \left( \frac{Z n}{N \vert D \vert p^{2(x + \beta)}} \right). \end{align*}
Let us then for each integer $n \geq 1$ in the sum consider the coefficient
\begin{align*} a_{x, n} &=  \sum_{\mathcal{W} \in P_{x, \beta}}  \frac{a_{f \times \mathcal{W}}(n)}{n^{\frac{1}{2}}} \end{align*}
and the function 
\begin{align*} h_n(x) &= V_{k+1}\left( \frac{Z n}{N \vert D \vert p^{2(x + \beta)}} \right). \end{align*}
Applying $(\ref{ps})$ to each integer $n \geq 1$ in the sum, and using that 
\begin{align*} h_n(x) - h_n(x+1)
&= \int_{\Re(s) = 2} \widehat{V}_{k+1}(s) \left( \frac{Z n}{N \vert D \vert p^{2(x + \beta)}} \right)^{-s} \left( p^s - 1 \right) \frac{ds}{2 \pi i},\end{align*}
the claim is easy to verify. The second stated formula is proven in a completely analogous way. \end{proof} 

\subsection{Evaluation of coefficients} 

We now evaluate the coefficients appearing in the formula of Corollary \ref{ps2v} explicitly via orthogonality relations, using the Dirichlet series 
expansion $(\ref{integralDirichlet})$ to describe the coefficients. To begin, fix integers $\alpha, \beta \geq 0$, as well as a real parameter $Z >0$. 
(We shall later take $Z=1$ when $\beta =0$). If $\alpha \geq 1$, then Lemma \ref{ps2v} gives us the expression 
\begin{align*} \sum_{\mathcal{W} \in X_{\alpha, \beta}} L^{(k)}(1/2, f \times \mathcal{W}) &= S_{1, Z} - S_{2, Z} + S_{3, Z} - S_{4, Z}, \end{align*} 
where 
\begin{align*} S_{1, Z} &=  \sum_{n \geq 1} \left( \sum_{\mathcal{W} \in X_{\alpha, \beta}} \frac{a_{f \times \mathcal{W}}(n)}{n^{\frac{1}{2}}} \right) 
V_{k+1} \left( \frac{Z n}{N \vert D \vert p^{2(\alpha + \beta)}}\right) \\  
S_{2, Z}  &= \sum_{n \geq 1} \left( \sum_{1 \leq x \leq \alpha - 1} \sum_{\mathcal{W}'  \in P_{x, \beta}} \frac{a_{f \times \mathcal{W}'}(n)}{n^{\frac{1}{2}}} \right) 
\mathfrak{V}_{k+1} \left( \frac{Z n}{N \vert D \vert p^{ 2(x + \beta) }} \right) \\ 
S_{3, Z} &= \sum_{n \geq 1} \left( \sum_{\mathcal{W} \in X_{\alpha, \beta}} (-1)^k \epsilon(1/2, f \times \mathcal{W}) 
\frac{a_{f \times \overline{\mathcal{W}}}(n)}{n^{\frac{1}{2}}} \right) V_{k+1} \left( \frac{n}{Z N \vert D \vert p^{2(\alpha + \beta)}} \right)  \\ 
S_{4, Z} &= \sum_{n \geq 1} \left( \sum_{1 \leq x \leq \alpha - 1} \sum_{\mathcal{W}' \in P_{x, \beta}} (-1)^k \epsilon(1/2, f \times \mathcal{W}') 
\frac{a_{f \times \overline{\mathcal{W}}'}(n)}{n^{\frac{1}{2}}} \right) \mathfrak{V}_{k+1} \left( \frac{n}{Z N \vert D \vert p^{2(x + \beta)}} \right). \end{align*} 
Note that if $\alpha = 0$ (so that we average over characters $\rho$ of the ideal class group of $K$), 
then partial summation is not necessary here, and we would have simply the formula 
$\sum_{\mathcal{W} \in X_{\alpha, \beta}}  L^{(k)}(1/2, f \times \mathcal{W})= S_{1, Z} + S_{3, Z}$.

Let us begin with the first sum 
\begin{align*} S_{1, Z} &= \sum_{n \geq 1} \frac{1}{n^{\frac{1}{2}}} \sum_{0 \leq x \leq \alpha } 
\sum_{\mathcal{W} \in P_{x, \beta} \atop \mathcal{W} = \rho \chi \circ {\bf{N}}} 
a_{f \times \mathcal{W} }(n) V_{k+1}\left( \frac{Z n}{N \vert D \vert p^{2(\alpha + \beta)}} \right) \\ &=  \sum_{n \geq 1} \frac{1}{n^{\frac{1}{2}}} 
\sum_{ \rho \in \operatorname{Pic}(\mathcal{O}_{p^{\alpha}})^{\vee}  } \sum_{\chi \operatorname{mod} p^{\beta} \atop \operatorname{primitive}} 
a_{f \times \rho \chi \circ {\bf{N}} }(n)V_{k+1}\left( \frac{Z n}{N\vert D \vert p^{2(\alpha + \beta)}} \right). \end{align*} 
Fixing a ring class character $\rho \in \operatorname{Pic}(\mathcal{O}_{p^{\alpha}})^{\vee}$, 
and using the Dirichlet series expansion $(\ref{integralDirichlet})$,
consider the sum over primitive even Dirichlet characters $\chi \operatorname{mod} p^{\beta}$ in this latter expression:
\begin{align*} &\sum_{\chi \operatorname{mod} p^{\beta} \atop \chi(-1)= 1, \operatorname{primitive}} 
\sum_{n \geq 1} \frac{a_{f \times \rho \chi \circ {\bf{N}} }(n)}{n^{\frac{1}{2}}} 
V_{k+1}\left( \frac{Z n}{N\vert D \vert p^{2(\alpha + \beta)} } \right)  \\ 
&= \sum_{\chi \operatorname{mod} p^{\beta} \atop \chi(-1) = \operatorname{primitive}} 
\sum_{m \geq 1 \atop (m, N p^{\beta})=1} \frac{\omega \chi^2(m)}{m} \sum_{n \geq 1 \atop (n, p^{\alpha + \beta})=1} 
\left(\sum_{A \in \operatorname{Pic}(\mathcal{O}_{p^{\alpha}})}  \rho(A) r_A(n)\right) \frac{\lambda(n) \chi(n)}{n^{\frac{1}{2}}} 
V_{k+1} \left(\frac{ Z m^2 n}{N \vert D \vert p^{2(\alpha + \beta)}} \right).\end{align*} 
Recall that, after using orthogonality with M\"obius inversion, we have for any integer $m \geq 1$ prime to $p$ that 
\begin{align}\label{QO} \sum_{\chi \bmod p^{\beta} \atop \chi(-1) =1, \operatorname{primitive}} \chi(m) 
&= \begin{cases} \frac{1}{2} \varphi^{\star}(p^{\beta}) &\text{ if $m \equiv \pm 1 \bmod p^{\beta}$} \\
- \frac{1}{2} \varphi(p^{\beta-1}) &\text{ if $m \equiv \pm 1 \bmod p^{\beta-1}$ but $m \not\equiv \pm 1 \bmod p^{\beta}$}\\
0 &\text{ otherwise} \end{cases} \end{align}  
if $\beta \geq 2$, and that
\begin{align}\label{QOp} \sum_{\chi \bmod p \atop \chi(-1) = 1, \operatorname{primitive}} \chi(m) 
&= \begin{cases} 0 &\text{ if $m \equiv 0 \bmod p$} \\
\frac{1}{2}\varphi(p) - 1 &\text{ if $m \equiv \pm 1 \bmod p$} \\
-1 &\text{ otherwise}\end{cases}\end{align} 
if $\beta =1$.
Using these relations, the latter sum over primitive even Dirichlet characters $\chi \operatorname{mod} p^{\beta}$ is given by 
\begin{align*} \frac{ \varphi^{\star}(p^{\beta})}{2} &\sum_{m \geq 1 \atop (m, Np)=1} \frac{\omega(m)}{m} 
\sum_{ {n \geq 1 \atop (n, p)=1} \atop m^2 n \equiv \pm 1 \bmod p^{\beta}} 
\left(\sum_{A \in \operatorname{Pic}(\mathcal{O}_{p^{\alpha}})}  \rho(A) r_A(n)\right) \frac{\lambda(n)}{n^{\frac{1}{2}}} 
V_{k+1} \left(\frac{Z m^2 n}{N \vert D \vert p^{2(\alpha + \beta)}} \right) \\
&- \frac{ \varphi(p^{\beta-1})}{2} \sum_{m \geq 1  \atop (m, Np)=1} \frac{\omega(m)}{m} 
\sum_{ {n \geq 1 \atop m^2 n \equiv \pm 1 \bmod p^{\beta - 1}} \atop m^2 n \not\equiv \pm 1 \bmod p^{\beta}} 
\left(\sum_{A \in \operatorname{Pic}(\mathcal{O}_{p^{\alpha}})}  \rho(A) r_A(n)\right) \frac{\lambda(n)}{n^{\frac{1}{2}}} 
V_{k+1} \left(\frac{Z m^2 n}{N \vert D \vert p^{2(\alpha + \beta)}} \right) \end{align*}
if $\beta \geq 2$, and by 
\begin{align*}\left( \frac{\varphi(p)}{2} - 1 \right) &\sum_{m \geq 1 \atop (m, Np)=1} \frac{\omega(m)}{m} 
\sum_{ {n \geq 1 \atop (n, p)=1} \atop m^2 n \equiv \pm 1 \bmod p^{\beta}} 
\left(\sum_{A \in \operatorname{Pic}(\mathcal{O}_{p^{\alpha}})}  \rho(A) r_A(n)\right) \frac{\lambda(n)}{n^{\frac{1}{2}}} 
V_{k+1} \left(\frac{Z m^2 n}{N \vert D \vert p^{2(\alpha + \beta)}} \right) \\
&- \sum_{m \geq 1 \atop (m, Np)=1} \frac{\omega(m)}{m} \sum_{ {n \geq 1 \atop (n, p)=1} \atop m^2 n \not\equiv \pm 1 \bmod p} 
\left(\sum_{A \in \operatorname{Pic}(\mathcal{O}_{p^{\alpha}})}  \rho(A) r_A(n)\right) \frac{\lambda(n)}{n^{\frac{1}{2}}} 
V_{k+1} \left(\frac{Z m^2 n}{N \vert D \vert p^{2(\alpha + \beta)}} \right) \end{align*} if $\beta =1$. 
Taking the sum over all ring class characters $\rho \in \operatorname{Pic}(\mathcal{O}_{p^{\alpha}})^{\vee}$, 
we can then use orthogonality of ring class characters in the $A$-sum to obtain the expression 
\begin{align*} S_{1, Z} = \# \operatorname{Pic}(\mathcal{O}_{p^{\alpha}}) \left( \frac{ \varphi^{\star}(p^{\beta})}{2} \right)
&\sum_{m \geq 1 \atop (m, pN)=1} \frac{\omega(m)}{m} 
\sum_{ {n \geq 1\atop (n, p)=1} \atop m^2 n \equiv \pm 1 \bmod p^{\beta}} \frac{r(n)\lambda(n)}{n^{\frac{1}{2}}} 
V_{k+1}\left(\frac{Z m^2 n }{N\vert D \vert p^{2(\alpha + \beta)} } \right) \\ 
&- \# \operatorname{Pic}(\mathcal{O}_{ p^{\alpha} }) \left( \frac{ \varphi(p^{\beta-1})}{2} \right) \sum_{m \geq 1 \atop (m, p N)=1} \frac{\omega(m)}{m} 
\sum_{ {n \geq 1 \atop m^2 n \equiv 1 \operatorname{mod} p^{\beta-1}} \atop m^2 n \not\equiv 1 \operatorname{mod} p^{\beta}} \frac{r(n)\lambda(n)}{n^{\frac{1}{2}}} 
V_{k+1}\left(\frac{Z m^2 n }{N\vert D \vert p^{2(\alpha + \beta)} } \right) \end{align*} 
if $\beta \geq 2$, and the expression 
\begin{align*} S_{1, Z} = \# \operatorname{Pic}(\mathcal{O}_{p^{\alpha}}) \left( \frac{\varphi(p) }{2} -1 \right) 
&\sum_{m \geq 1 \atop (m, p N)=1} \frac{\omega(m)}{m} \sum_{ {n \geq 1\atop (n, p)=1} \atop m^2 n \equiv 1  p} \frac{r(n)\lambda(n)}{n^{\frac{1}{2}}} 
V_{k+1}\left(\frac{ Z m^2 n }{N\vert D \vert p^{2(\alpha + \beta)} } \right) \\ 
&- \# \operatorname{Pic}(\mathcal{O}_{ p^{\alpha} }) \sum_{m \geq 1 \atop (m, p N)=1} \frac{\omega(m)}{m} 
\sum_{ n \geq 1 \atop m^2 n \not\equiv 1 \operatorname{mod} p} \frac{r(n)\lambda(n)}{n^{\frac{1}{2}}} 
V_{k+1}\left(\frac{Z m^2 n }{N\vert D \vert p^{2(\alpha + \beta)} } \right) \end{align*} if $\beta =1$. 
Here again, we write $r(n)$ to denote the number of principal ideals in the trivial class of $\operatorname{Pic}(\mathcal{O}_{p^{\alpha}})$ of norm $n$,
which for integers $n$ coprime to $p$ is the same as the number of principal ideals in the trivial class of $\operatorname{Pic}(\mathcal{O}_K)$ of norm $n$.
Note that the inner $x$-sums defining $S_{2}$ can be evaluated in an analogous way, as 
\begin{align*} S_{2, Z} = \sum_{1 \leq x \leq \alpha-1} &\# \operatorname{Pic}(\mathcal{O}_{p^{x}}) \left( \frac{ \varphi^{\star}(p^{\beta})}{2} \right)
\sum_{m \geq 1 \atop (m, p N)=1} \frac{\omega(m)}{m} 
\sum_{ {n \geq 1\atop (n, p)=1} \atop m^2 n \equiv 1 \operatorname{mod} p^{\beta}} \frac{r(n)\lambda(n)}{n^{\frac{1}{2}}} 
\mathfrak{V}_{k+1}\left(\frac{Z m^2 n }{N\vert D \vert p^{2(x + \beta)} } \right) \\ 
&- \sum_{1 \leq x \leq \alpha - 1} \# \operatorname{Pic}(\mathcal{O}_{ p^{x} }) \left( \frac{\varphi(p^{\beta-1})}{2} \right) 
\sum_{m \geq 1 \atop (m, p N)=1} \frac{\omega(m)}{m} \sum_{ {n \geq 1 \atop m^2 n \equiv 1 \bmod p^{\beta-1}} \atop m^2 n \not\equiv 1 \bmod p^{\beta}} \frac{r(n)
\lambda(n)}{n^{\frac{1}{2}}} \mathfrak{V}_{k+1}\left(\frac{Z m^2 n }{N\vert D \vert p^{2(x + \beta)}} \right) \end{align*} 
if $\beta \geq 2$, and as 
\begin{align*} S_{2, Z} = \sum_{1 \leq x \leq \alpha - 1} \# \operatorname{Pic}(\mathcal{O}_{p^{x}}) \left( \frac{\varphi(p)}{2} - 1 \right)
&\sum_{m \geq 1 \atop (m, p N)=1} \frac{\omega(m)}{m} 
\sum_{ {n \geq 1\atop (n, p)=1} \atop m^2 n \equiv 1 \operatorname{mod} p} \frac{r(n)\lambda(n)}{n^{\frac{1}{2}}} 
\mathfrak{V}_{k+1}\left(\frac{Z m^2 n }{N\vert D \vert p^{2(x + \beta)}} \right) \\ 
&- \# \operatorname{Pic}(\mathcal{O}_{ p^{x} }) \sum_{m \geq 1 \atop (m, p N)=1} \frac{\omega(m)}{m} 
\sum_{ n \geq 1 \atop m^2 n \not\equiv 1 \operatorname{mod} p} \frac{r(n)\lambda(n)}{n^{\frac{1}{2}}} 
\mathfrak{V}_{k+1}\left(\frac{Z m^2 n }{N\vert D \vert p^{2(x + \beta)} } \right) \end{align*} if $\beta =1$.

Let us now evaluate the third sum $S_{3}$, where the average over the Gauss sums 
contained in the definition of the twisted coefficients is a bit delicate. We therefore make the following preliminary calculations. 

\begin{lemma}[Chinese remainder theorem]\label{CRT} 

Given $D \geq 1$ any integer prime to $p$, and $M$ any coprime class $\bmod p^{\beta}$ (for $\beta \geq 1$), we have the identity of sums
\begin{align*} \sum_{z \bmod D p^{\beta} \atop z^2 \equiv M \bmod p^{\beta}} \omega(z) e \left( \frac{z}{Dp^{\beta}}\right) 
&= \omega(p^{\beta}) \sum_{y \bmod D} \omega(y) e \left( \frac{y}{D} \right) 
\sum_{x \bmod p^{\beta} \atop x \equiv \pm M^{\frac{1}{2}} \overline{D} \bmod  p^{\beta}} 
e \left( \frac{x}{p^{\beta}}\right). \end{align*} Here, $M^{\frac{1}{2}}$ denotes a square root of the class $M \operatorname{mod} p^{\beta}$ 
(if it exists), and $\overline{D}$ the inverse class of $D \operatorname{mod} p^{\beta}$.

If $\beta \geq 2$, then we also have the identity 
\begin{align*} \sum_{  {z \bmod D p^{\beta} \atop z^2 \equiv M \bmod p^{\beta-1}} \atop z^2 \not\equiv M \bmod p^{\beta} } \omega(z) 
e \left( \frac{z}{Dp^{\beta}}\right) &= \omega(p^{\beta}) \sum_{y \bmod D} \omega(y) e \left( \frac{y}{D} \right) 
\sum_{  {x \bmod p^{\beta} \atop x \equiv \pm M^{\frac{1}{2}} \overline{D} \bmod p^{\beta-1}} \atop x \not\equiv \pm M^{\frac{1}{2}} \overline{D} 
\bmod p^{\beta}} e \left( \frac{x}{p^{\beta}}\right). \end{align*} \end{lemma} 

\begin{proof} 

Observe that since $(D, p^{\beta}) =1$, we can write any class $z \bmod D p^{\beta}$ as $z = y p^{\beta} + x D$ for uniquely determined
classes $y \bmod D$ and $x \bmod p^{\beta}$. Hence, the $z$-sum can be decomposed as the double sum
\begin{align*} \sum_{z \bmod D p^{\beta} \atop z^2 \equiv M \bmod p^{\beta}} \omega(z) e \left( \frac{z}{Dp^{\beta}}\right)
&= \sum_{y \bmod D} \sum_{x \bmod p^{\beta} \atop y p^{\beta} + x D \equiv \pm M^{\frac{1}{2}} \bmod p^{\beta}} 
\omega(y p^{\beta} + x D) e \left( \frac{yp^{\beta} +x D}{D p^{\beta}} \right) \\
&= \omega(p^{\beta}) \sum_{y \bmod D} \omega(y) e \left( \frac{y}{D} \right) 
\sum_{x \bmod p^{\beta} \atop y p^{\beta} +  x D \equiv \pm M^{\frac{1}{2}} \bmod p^{\beta}} e \left( \frac{x}{p^{\beta}} \right) \\
 &= \omega(p^{\beta}) \sum_{y \bmod D} \omega(y) e \left( \frac{y}{D} \right)
 \sum_{x \bmod p^{\beta} \atop x  \equiv \pm M^{\frac{1}{2}} \overline{D} p^{\beta}} e \left( \frac{x}{p^{\beta}}\right) . \end{align*} 
The second stated identity is shown in the same way, replacing the congruences in the $z$-sum suitably. \end{proof} 

We can now derive the following result. Let us for an integer $c$ prime to $p^{\beta}$ (hence prime to $p$) write $\operatorname{Kl}_4(c, p^{\beta})$ 
to denote the hyper-Kloosterman sum of dimension $4$ and modulus $p^{\beta}$ (for $\beta \geq 1$) evaluated at $c$:
\begin{align*} \operatorname{Kl}_4(c, p^{\beta}) &= \sum_{x_1, \cdots, x_4 \bmod p^{\beta} \atop x_1 \cdots x_4 \equiv c \bmod p^{\beta}} 
e \left(  \frac{x_1 + \cdots + x_4}{p^{\beta}}\right). \end{align*}
Let us also lighten notation by writing 
\begin{align*} \operatorname{Kl}_4( \pm c, p^{\beta}) &= \sum_{x_1, \cdots, x_4 \bmod p^{\beta} \atop x_1 \cdots x_4 \equiv \pm c \bmod p^{\beta}} 
e \left(  \frac{x_1 + \cdots + x_4}{p^{\beta}}\right) =  \operatorname{Kl}_4(c, p^{\beta}) +  \operatorname{Kl}_4(-c, p^{\beta}).\end{align*}

\begin{proposition}\label{RT}  

Assume (as we do throughout) that $D$ is prime to $p$, so that $\omega\chi^2$ has conductor $\vert D \vert p^{\beta}$. \\

\noindent (i) If $\beta \geq 2$, then we have for each integer $c \geq 1$ prime to $p$ the summation formula
\begin{align*} \sum_{\chi \bmod p^{\beta} \atop \chi(-1)=1, \operatorname{primitive}} \chi(c) \tau(\omega\chi^2)^4 
&=  \tau(\omega)^4 \left( \frac{\varphi(p^{\beta})}{2} \right) \operatorname{Kl}_4( \pm \overline{c}^{\frac{1}{2}} \overline{D}, p^{\beta}). \end{align*} 
Here, $\overline{c}^{\frac{1}{2}}$ denotes a square root of $\overline{c} \bmod p^{\beta}$ (if it exists), 
and $\overline{c}$ the multiplicative inverse of $c \bmod p^{\beta}$. \\

\noindent (ii) If $\beta \geq 2$, then the sum $S_{3, Z}$ is given equivalently by the expression  
\begin{align*} S_{3, Z} &= (-1)^{k+1} \# \operatorname{Pic}(\mathcal{O}_{p^{\alpha}}) \left( \frac{ \varphi(p^{\beta})}{2} \right)
\frac{\omega(N) \tau(\omega)^4}{(\vert D \vert p^{\beta})^2} \\
&\times \sum_{m \geq 1\atop (m, p N)=1} \frac{\omega(m)}{m}  \sum_{ {n \geq 1 \atop (n, p)=1}} 
\frac{r(n) \lambda(n)}{n^{\frac{1}{2}}}V_{k+1} \left( \frac{m^2n}{Z N \vert D \vert p^{2(\alpha + \beta)}} \right) 
\operatorname{Kl}_4( \pm (m^2 n \overline{N}^2)^{\frac{1}{2}} \overline{D}, p^{\beta}).
 \end{align*}  \end{proposition}

\begin{proof} 

Let us start with (i). We have by definition of the Gauss sum $\tau(\omega \chi^2)$ that 
\begin{align*} \sum_{\chi \bmod p^{\beta} \atop \chi(-1) = 1, \operatorname{primitive}} \chi(c) \tau(\omega\chi^2)^4 
&= \sum_{\chi \bmod p^{\beta} \atop \chi(-1)=1, \operatorname{primitive}} \chi(c) \tau(\omega\chi^2)^4 \\
&= \sum_{\chi \bmod p^{\beta} \atop \chi(-1)=1, \operatorname{primitive}}
\sum_{z_1, \ldots, z_4 (\bmod \vert D \vert p^{\beta})} \chi(c) \omega\chi^2 (z_1\cdots z_4) 
e \left( \frac{z_1 + \cdots + z_4} {\vert D \vert p^{\beta}} \right). \end{align*}  
Note that here, we must take a sum over primitive (even) Dirichlet characters 
$\chi \bmod p^{\beta}$ ($\chi \neq \chi_0$) as otherwise $\tau(\omega \chi^2) = 0$. 
To evaluate the sum, we switch the order of summation and use the relation $(\ref{QO})$ to obtain 
\begin{align*}  &\sum_{z_1, \ldots, z_4 \bmod \vert D \vert p^{\beta}} \omega(z_1 \ldots z_4) e \left( \frac{z_1 + \cdots + z_4}{\vert D \vert p^{\beta}} \right)
\sum_{\chi \bmod p^{\beta} \atop \chi(-1) = 1, \operatorname{primitive}} \chi(z_1^2 \cdots z_4^2 c) \\ 
&=\frac{ \varphi^{\star}(p^{\beta})}{2}  \sum_{z_1, \ldots, z_4 \bmod \vert D \vert p^{\beta} \atop z_1^2 \cdots z_4^2 c \equiv \pm 1 \bmod p^{\beta} }
e \left( \frac{z_1 + \cdots + z_4}{ \vert D \vert p^{\beta}} \right) - \frac{ \varphi(p^{\beta - 1})}{2} 
\sum_{ {z_1, \ldots, z_4 \bmod \vert D \vert p^{\beta} \atop z_1^2 \cdots z_4^2 c \equiv \pm 1 \bmod p^{\beta-1} } 
\atop z_1^2 \cdots z_4^2 c \not\equiv \pm 1 \bmod p^{\beta} }
e \left( \frac{z_1 + \cdots + z_4}{ \vert D \vert p^{\beta}} \right),\end{align*} 
which after applying the results of Lemma \ref{CRT} to each of the sums gives the expression 
\begin{align}\label{crtgs} \tau(\omega)^4 \left(  \frac{\varphi^{\star}(p^{\beta})}{2}  
\sum_{x_1, \ldots, x_4 \bmod p^{\beta} \atop x_1 \cdots x_4  \equiv \pm \overline{c}\frac{1}{2} \overline{D} \bmod p^{\beta} }
e \left( \frac{x_1 + \cdots + x_4}{ p^{\beta}} \right) - \frac{\varphi(p^{\beta - 1})}{2} 
\sum_{ {x_1, \ldots, x_4 \bmod p^{\beta} \atop x_1 \cdots x_4  \equiv \pm \overline{c}^{\frac{1}{2}}\overline{D} \bmod p^{\beta-1} } 
\atop x_1 \cdots x_4  \not\equiv \pm \overline{c}^{\frac{1}{2}} \overline{D} \bmod p^{\beta} }
e \left( \frac{x_1 + \cdots + x_4}{  p^{\beta}} \right) \right). \end{align} 
Here, we have used that $\omega(p^{\beta})^4 = 1$. 
Let us consider the second inner sum in this latter expression,  
\begin{align*} \sum_{ {x_1, \ldots, x_4 \bmod p^{\beta} \atop x_1 \cdots x_4 \equiv \pm \overline{c}^{\frac{1}{2}} \overline{D} \bmod p^{\beta-1} } 
\atop x_1 \cdots x_4  \not\equiv \pm \overline{c}^{\frac{1}{2}} \overline{D} \bmod p^{\beta} }
e \left( \frac{x_1 + \cdots + x_4}{  p^{\beta}} \right) &=
\sum_{x_1, x_2, x_3 \bmod p^{\beta}} e \left( \frac{x_1 + x_2 + x_3}{p^{\beta}} \right) 
\sum_{ {x_4 \bmod p^{\beta}  \atop x_4 \equiv \pm \overline{x_1 x_2 x_3 D c}^{\frac{1}{2}} \bmod p^{\beta-1} } 
\atop  x_4 \not\equiv \pm \overline{x_1 x_2 x_3 D c}^{\frac{1}{2}} \bmod p^{\beta} } e \left( \frac{x_4}{p^{\beta}}\right). \end{align*}
Observe that we can express each class $x_4 \bmod p^{\beta}$ in the second sum as 
$\pm \overline{x_1 x_2 x_3 D c}^{\frac{1}{2}} + lp^{\beta - 1}$ for some integer $1 \leq l \leq p-1$, so that 
\begin{align*} \sum_{ {x_4 \bmod p^{\beta}  \atop x_4 \equiv \pm \overline{x_1 x_2 x_3 D c}^{\frac{1}{2}} \bmod p^{\beta-1} } 
\atop  x_4 \not\equiv \pm \overline{x_1 x_2 x_3 D c}^{\frac{1}{2}} \bmod p^{\beta} } e \left( \frac{x_4}{p^{\beta}}\right) 
&= \left( e \left(  \frac{\overline{x_1 x_2 x_3 D c}^{\frac{1}{2}}}{p^{\beta}}\right) 
+ e \left(  \frac{-\overline{x_1 x_2 x_3 D c}^{\frac{1}{2}}}{p^{\beta}}\right) \right)
\sum_{l=1}^{p-1} e \left( \frac{l p^{\beta-1}}{p^{\beta}} \right) \\
&= -  \left( e \left(  \frac{\overline{x_1 x_2 x_3 D c}^{\frac{1}{2}}}{p^{\beta}}\right) 
+ e \left(  \frac{-\overline{x_1 x_2 x_3 D c}^{\frac{1}{2}}}{p^{\beta}}\right) \right) \end{align*}
by the well-known identity $\sum_{1 \leq l \leq p-1} e \left(  \frac{l}{p}\right) = -1$. Hence, we derive the relation
\begin{align*} \sum_{ {x_1, \ldots, x_4 \bmod p^{\beta} \atop x_1 \cdots x_4 \equiv \pm \overline{c}^{\frac{1}{2}} \overline{D}  \bmod p^{\beta-1} } 
\atop x_1 \cdots x_4  \not\equiv \pm \overline{c}^{\frac{1}{2}} \overline{D} \bmod p^{\beta} } e \left( \frac{x_1 + \cdots + x_4}{  p^{\beta}} \right) 
&= \sum_{x_1, x_2, x_3 \bmod p^{\beta}} e \left( \frac{ x_1 + x_2 + x_3 \pm  \overline{x_1 x_2 x_3 D c}^{\frac{1}{2}} }{p^{\beta}} \right) \\
&= - \sum_{ x_1, \ldots, x_4 \bmod p^{\beta} \atop x_1 \cdots x_4 \equiv \pm \overline{c}^{\frac{1}{2}} \overline{D} \bmod p^{\beta} } 
e \left( \frac{x_1 + \cdots + x_4}{  p^{\beta}} \right), \end{align*} 
from which it follows that $(\ref{crtgs})$ is equivalent to the expression
\begin{align*} \tau(\omega)^4 \left(  \frac{\varphi^{\star}(p^{\beta})}{2}  
\sum_{x_1, \ldots, x_4 \bmod p^{\beta} \atop x_1 \cdots x_4  \equiv \pm \overline{c}^{\frac{1}{2}} \overline{D} \bmod p^{\beta} }
e \left( \frac{x_1 + \cdots + x_4}{ p^{\beta}} \right) + \frac{\varphi(p^{\beta - 1})}{2} 
\sum_{ x_1, \ldots, x_4 \bmod p^{\beta} \atop x_1 \cdots x_4  \equiv \pm \overline{c}^{\frac{1}{2}} \overline{D} \bmod p^{\beta} }
e \left( \frac{x_1 + \cdots + x_4}{  p^{\beta}} \right) \right). \end{align*}
Using that $\varphi^{\star}(p^{\beta}) = \varphi(p^{\beta}) - \varphi(p^{\beta-1})$, we then derive the stated identity
\begin{align*} \sum_{\chi \bmod p^{\beta} \atop \chi(-1) = 1, \operatorname{primitive}} 
\chi(c) \tau(\omega\chi^2)^4 &= \tau(\omega)^4 \left( \frac{\varphi(p^{\beta})}{2} \right)  
\sum_{x_1, \ldots, x_4 \bmod p^{\beta} \atop x_1 \cdots x_4 \equiv \pm \overline{c}^{\frac{1}{2}} \overline{D} \bmod p^{\beta}} 
e\left( \frac{x_1 + \cdots + x_4}{ p^{\beta} }\right). \end{align*} 

To show (ii), we start with the definition of the sum, which after unraveling notations has the expansion 
\begin{align*} S_{3} &= (-1)^{k+1} \sum_{n \geq 1}  \frac{1}{n^{\frac{1}{2}}} 
 \sum_{\chi \bmod p^{\beta} \atop \chi(-1) = 1, \operatorname{ primitive}} \sum_{\rho \in \operatorname{Pic}(\mathcal{O}_{p^{\alpha}})^{\vee}} 
  \frac{\tau(\omega \chi^2)^4}{( \vert D \vert p^{\beta})^2 } \cdot \omega\chi^2(N) \cdot 
a_{f \times \overline{\rho \chi} \circ {\bf{N}}}(n) V_{k+1} \left( \frac{ n }{Z N \vert D \vert p^{2(\alpha + \beta)}}\right). \end{align*} 
Fix a ring class character $\rho \in \operatorname{Pic}(\mathcal{O}_{p^{\alpha}})^{\vee}$. Consider the corresponding sum in this latter 
expression, which after expanding out in terms of $(\ref{integralDirichlet})$ and switching the order of summation is the same as:
\begin{align*} &(-1)^{k+1} \sum_{\chi \mod p^{\beta} \atop \chi(-1) = 1, \operatorname{ primitive}} 
\left(\frac{\tau(\omega \chi^2)^4}{(\vert D \vert p^{\beta})^2 } \right) \omega\chi^2(N) 
\sum_{m \geq 1 \atop (m, p N)=1} \frac{\omega\overline{\chi}^2(m)}{m} \sum_{n \geq 1 \atop (n, p)=1} 
 \sum_{A \in \operatorname{Pic}(\mathcal{O}_{p^{\alpha}})} r_A(n) \rho(A) \frac{\lambda(n) \overline{\chi}(n)}{n^{\frac{1}{2}}} 
V_{k+1} \left( \frac{m^2n }{Z N \vert D \vert p^{2(\alpha + \beta)}}\right) \\
&= \frac{(-1)^{k+1}}{(\vert D \vert p^{\beta})^2} \sum_{m \geq 1 \atop (m, p N)=1} \frac{\omega(m)}{m} \sum_{n \geq 1 \atop (n, p)=1} 
 \sum_{A \in \operatorname{Pic}(\mathcal{O}_{p^{\alpha}})} r_A(n) \rho(A) \frac{\lambda(n)}{n^{\frac{1}{2}}} 
V_{k+1} \left( \frac{m^2n }{Z N \vert D \vert p^{2(\alpha + \beta)}}\right) 
\sum_{\chi \bmod p^{\beta} \atop \chi(-1) = 1, \operatorname{ primitive}} \chi(\overline{m}^2 \overline{n} N^2) \tau(\omega \chi^2)^4. \end{align*}
Using (i), we have for each pair of integers $m, n$ in the sum the identity 
\begin{align*} \sum_{ \chi \bmod p^{\beta} \atop \chi(-1)=1, \operatorname{ primitive} } \tau( \omega \chi^2 )^4 \chi( \overline{m}^2 \overline{n} N^2) 
&= \tau(\omega)^4 \left( \frac{\varphi(p^{\beta})}{2}\right) \operatorname{Kl}_4( \pm (m^2 n \overline{N}^2)^{\frac{1}{2}} \overline{D}, p^{\beta}).\end{align*} 
Substituting this back into our previous expression, we find that the sum over primitive even $\chi$ is given by 
\begin{align*} (-1)^{k+1} \left( \frac{\varphi(p^{\beta})}{2} \right)\frac{\omega(N) \tau(\omega)^4}{(\vert D \vert p^{\beta})^2} 
 \sum_{m \geq 1\atop (m, p N)=1} \frac{\omega(m)}{m}  \sum_{ {n \geq 1 \atop (n, p)=1}} 
&\left(\sum_A \rho(A)r_A(n) \right)\frac{\lambda(n)}{n^{\frac{1}{2}}}V_{k+1} \left( \frac{m^2 n }{N \vert D \vert p^{2(\alpha + \beta)}} \right) \\
&\times \operatorname{Kl}_4( \pm (m^2 n \overline{N}^2)^{\frac{1}{2}} \overline{D}, p^{\beta}). \end{align*} 
Now, taking the sum over all ring class characters $\rho$ of conductor $p^{\alpha}$, 
and applying $\rho$-orthogonality relations to the inner $A$-sum, we obtain the identity  
\begin{align*} S_{3, Z} &= (-1)^{k+1} \left( \frac{\varphi(p^{\beta})}{2} \right) \# \operatorname{Pic}(\mathcal{O}_{p^{\alpha}}) 
&\\ &\times \frac{\omega(N) \tau(\omega)^4}{(\vert D \vert p^{\beta})^2} 
 \sum_{m \geq 1\atop (m, p N)=1} \frac{\omega(m)}{m}  \sum_{ {n \geq 1 \atop (n, p)=1}} 
\frac{r(n) \lambda(n)}{n^{\frac{1}{2}}}V_{k+1} \left( \frac{m^2 n }{Z N \vert D \vert p^{2(\alpha + \beta)}} \right) 
\operatorname{Kl}_4( \pm (m^2 n \overline{N}^2)^{\frac{1}{2}} \overline{D}, p^{\beta}). \end{align*} 
Here again, $r(n)$ denotes the number of ideals in the principal class ${\bf{1}} \in \operatorname{Pic}(\mathcal{O}_{p^{\alpha}})$ of norm $n$, which 
for integers $n$ coprime to $p$ coincides with the number of ideals in the principal class of $\operatorname{Pic}(\mathcal{O}_K)$ of norm $n$. \end{proof} 

\begin{corollary}\label{RTp}

Assume again that the discriminant $D$ is coprime to $p$, so that $\omega\chi^2$ has conductor $\vert D \vert p^{\beta}$. \\

\noindent (i) If $\beta =1$, then we have for each integer $c \geq 1$ prime to $p$ the summation formula
\begin{align*} \sum_{\chi \bmod p^{\beta} \atop \chi(-1) = 1, \operatorname{ primitive}} \chi(c) \tau(\omega\chi^2)^4 
&= \tau(\omega)^4  \left( \frac{\varphi(p)}{2} - 1 \right) \left(  \operatorname{Kl}_4( \pm \overline{c}^{\frac{1}{2}} \overline{D}, p)  -1 \right). \end{align*} 

\noindent (ii) If $\beta =1$, then the sum $S_{3, Z}$ is given equivalently by the expression  
\begin{align*} &(-1)^{k+1} \# \operatorname{Pic}(\mathcal{O}_{p^{\alpha}}) \frac{\omega(N) \tau(\omega)^4}{(\vert D \vert p^{\beta})^2} 
\sum_{m \geq 1\atop (m, p N)=1} \frac{\omega(m)}{m}  \sum_{ {n \geq 1 \atop (n, p)=1}} 
\frac{r(n) \lambda(n)}{n^{\frac{1}{2}}} V_{k+1} \left( \frac{m^2 n}{Z N \vert D \vert p^{2(\alpha + \beta)}} \right) \\ &\times 
\left(  \left( \frac{\varphi(p)}{2} - 1 \right) \operatorname{Kl}_4( \pm (m^2 n \overline{N}^2)^{\frac{1}{2}} \overline{D}, p) - 1 \right). \end{align*} 
\end{corollary}

\begin{proof} 

The calculations are given by a minor variation of those for Proposition \ref{RT} above as follows. 
For (i), we open up the sum over primitive even Dirichlet characters $\chi \bmod p$,
\begin{align*} \sum_{\chi \bmod p \atop \chi(-1) = 1, \operatorname{primitive}} \chi(c) \tau(\omega \chi^2)^4 
&= \sum_{\chi \mod p \atop \chi(-1)=1, \operatorname{primitive}} \chi(c) 
\sum_{z_1, \cdots, z_4 \bmod \vert D \vert p } \omega\chi^2(z_1 \cdots z_4) e \left( \frac{z_1 + \cdots + z_4}{p}\right) \\
&= \sum_{z_1, \cdots, z_4 \bmod \vert D \vert p } \omega (z_1 \cdots z_4) e \left( \frac{z_1 + \cdots + z_4}{p}\right) 
\sum_{\chi \bmod p \atop \chi(-1) = 1, \operatorname{primitive}} (z_1^2 \cdots z_4^2 c). \end{align*}
Applying the relation $(\ref{QOp})$ to the inner $\chi$-sum in latter expression, we then obtain 
\begin{align*} \left( \frac{\varphi(p)}{2} - 1 \right) \sum_{z_1, \cdots, z_4 \operatorname{mod} \vert D \vert p \atop z_1^2 \cdots z_4^2 c \equiv \pm 1 \bmod p } 
\omega (z_1 \cdots z_4) e \left( \frac{z_1 + \cdots + z_4}{ \vert D \vert p}\right)
-\sum_{z_1, \cdots, z_4 \bmod \vert D \vert p \atop z_1^2 \cdots z_4^2 c \not\equiv \pm1 \bmod p} 
\omega (z_1 \cdots z_4) e \left( \frac{z_1 + \cdots + z_4}{\vert D \vert p}\right), \end{align*}
which after Lemma \ref{CRT} can be expressed as the stated formula 
\begin{align*} &\tau(\omega)^4 \left( \left( \frac{\varphi(p)}{2} -1 \right)
\sum_{x_1, \cdots, x_4 \operatorname{mod} p \atop x_1 \cdots x_4 \equiv \pm \overline{c}^{\frac{1}{2}}\overline{D} \operatorname{mod} p } 
e \left( \frac{x_1 + \cdots + x_4}{p}\right) 
-\sum_{x_1, \cdots, x_4 \operatorname{mod} p \atop x_1 \cdots x_4 \not\equiv \pm \overline{c}^{\frac{1}{2}} \overline{D} \operatorname{mod} p} 
e \left( \frac{x_1 + \cdots + x_4}{p}\right) \right) \\ 
&= \tau(\omega)^4 \left( \left( \frac{\varphi(p)}{2} -1 \right) \operatorname{Kl}_4( \pm \overline{c}^{\frac{1}{2}} \overline{D}, p) - (-1)^4 \right). \end{align*}

To derive the stated formula for (ii), we simply use this formula to compute of the corresponding $\chi$-sum in the previous discussion 
(in the proof of Proposition \ref{RT} (ii)), with all other steps being the same. \end{proof}

Again, we can compute the sum $S_{4, Z}$ in a completely analogous way as
\begin{align*} S_{4, Z} &= (-1)^{k+1} \left( \frac{\varphi(p^{\beta})}{2} \right) \frac{\omega(N) \tau(\omega)^4}{(\vert D \vert p^{\beta})^2} 
\sum_{1 \leq x \leq \alpha - 1} \# \operatorname{Pic}(\mathcal{O}_{p^{x}})  \\
&\times \sum_{m \geq 1\atop (m, p N)=1} \frac{\omega(m)}{m}  \sum_{ {n \geq 1 \atop (n, p)=1}} 
\frac{r(n) \lambda(n)}{n^{\frac{1}{2}}} \mathfrak{V}_{k+1} \left( \frac{m^2n }{Z N \vert D \vert p^{2(x + \beta)}} \right) 
\operatorname{Kl}_4( \pm (m^2 n \overline{N}^2)^{\frac{1}{2}} \overline{D}, p^{\beta}) . \end{align*}  
if $\beta \geq 2$, as as 
\begin{align*} S_{4, Z}&= (-1)^{k+1} \frac{\omega(N) \tau(\omega)^4}{(\vert D \vert p^{\beta})^2} 
\sum_{1 \leq x \leq \alpha - 1} \# \operatorname{Pic}(\mathcal{O}_{p^{x}})  \\
&\times \sum_{m \geq 1\atop (m, p N)=1} \frac{\omega(m)}{m}  \sum_{ {n \geq 1 \atop (n, p)=1}} 
\frac{r(n) \lambda(n)}{n^{\frac{1}{2}}} \mathfrak{V}_{k+1} \left( \frac{m^2 n }{Z N \vert D \vert p^{2(x + \beta)}} \right) 
\left( \left( \frac{\varphi(p)}{2} -1 \right) \operatorname{Kl}_4( \pm (m^2 n \overline{N}^2)^{\frac{1}{2}} \overline{D}, p) - 1 \right) \end{align*} 
if $\beta = 1$. Putting together these expressions for $S_{1, Z} - S_{2, Z} + S_{3, Z} - S_{4, Z}$, we derive the following result:

\begin{proposition}\label{haf} 

Fix integers $\alpha \geq 0$ and $ \beta \geq 0$.
Let $\mathcal{W} = \rho \chi \circ {\bf{N}} \in X_{\alpha, \beta}$ be a Hecke character of $K$, as in $(\ref{factorization})$, 
with $\rho$ a primitive ring class character of conductor $p^{\alpha}$, and $\chi$ a primitive even Dirichlet character of conductor $p^{\beta}$.
We have for either choice of $k \in \lbrace 0, 1 \rbrace$ the following formula for the corresponding average 
\begin{align*} H^{(k)}(\alpha, \beta) &= \frac{2}{\# \operatorname{Pic}(\mathcal{O}_{p^{\alpha}}) \varphi^{\star}(p^{\beta})} 
\sum_{\rho \in \operatorname{Pic}(\mathcal{O}_{p^{\alpha}})^{\vee}} 
\sum_{\chi \bmod p^{\beta} \atop \chi(-1) = 1, \operatorname{primitive}} L^{(k)}(1/2, f \times \rho \chi \circ {\bf{N}}): \end{align*}
For any choice of real parameter $Z >1$, we have the formula 
\begin{align*} H^{(k)}(\alpha, \beta) = D_{k+1}(\alpha, \beta; Z) &+(-1)^{k+1} \widetilde{D}_{k+1}(\alpha, \beta; Z) \\
&- \sum_{0 \leq x \leq \alpha -1}  \frac{\# \operatorname{Pic} (\mathcal{O}_{p^x})}{\# \operatorname{Pic}(\mathcal{O}_{p^{\alpha}})}
 \left( \mathfrak{D}_{k+1}(x, \beta; Z) + (-1)^{k+1} \widetilde{\mathfrak{D}}_{k+1}(x, \beta; Z) \right). \end{align*}
   
Here, we define the leading sums 
\begin{align*} D_{k+1}(\alpha, \beta; Z) = \sum_{m \geq 1 \atop (m, Np^{\beta})=1} \frac{\omega(m)}{m} 
&\sum_{ {n \geq 1 \atop (n, p) =1} \atop m^2 n \equiv 1 \bmod p^{\beta}} \frac{r(n)\lambda(n)}{n^{\frac{1}{2}}} 
V_{k+1} \left(\frac{Z m^2 n }{ N \vert D \vert p^{2(\alpha + \beta)} } \right) \\ & - \frac{1}{\varphi(p)} \sum_{m \geq 1 \atop (m, N)=1} \frac{\omega(m)}{m} 
\sum_{ {n \geq 1 \atop m^2 n \equiv 1 \bmod p^{\beta-1}} \atop m^2 n \not \equiv 1 \bmod p^{\beta}} \frac{r(n)\lambda(n)}{n^{\frac{1}{2}}} 
V_{k+1}\left(\frac{Z m^2 n }{ N \vert D \vert p^{2(\alpha + \beta)} } \right) \end{align*}
if $\beta \geq 2$, and 
\begin{align*} \sum_{m \geq 1 \atop (m, N p^{\beta})=1} \frac{\omega(m)}{m} 
&\sum_{ {n \geq 1 \atop (n, p^{\alpha + \beta}) =1} \atop m^2 n \equiv 1 \bmod p} \frac{r(n)\lambda(n)}{n^{\frac{1}{2}}} 
V_{k+1}\left(\frac{ Z m^2 n }{ N \vert D \vert p^{2(\alpha + \beta)} } \right) \\ &- \frac{2}{p-3} \sum_{m \geq 1 \atop (m, N)=1} \frac{\omega(m)}{m} 
\sum_{ n \geq 1 \atop m^2 n \not\equiv 1 \bmod p} \frac{r(n)\lambda(n)}{n^{\frac{1}{2}}} 
V_{k+1}\left(\frac{Z m^2 n }{ N \vert D \vert p^{2(\alpha + \beta)} } \right) \end{align*} 
if $\beta =1$, with no congruence conditions $\bmod p^{\beta}$ if $\beta =0$.
Similarly, for $1 \leq x \leq \alpha - 1$, we define 
\begin{align*} \mathfrak{D}_{k+1}(x, \beta; Z) = \sum_{m \geq 1 \atop (m, N p^{\beta})=1} \frac{\omega(m)}{m} 
&\sum_{ {n \geq 1 \atop (n, p) =1} \atop m^2 n \equiv 1 \bmod p^{\beta}} \frac{r(n)\lambda(n)}{n^{\frac{1}{2}}} 
\mathfrak{V}_{k+1}\left(\frac{Z m^2 n }{ N \vert D \vert p^{2(x + \beta)} } \right) \\ &- \frac{1}{\varphi(p)} \sum_{m \geq 1 \atop (m, N)=1} \frac{\omega(m)}{m} 
\sum_{ {n \geq 1 \atop m^2 n \equiv 1 \bmod p^{\beta-1}} \atop m^2 n \not \equiv 1 \bmod p^{\beta}} \frac{r(n)\lambda(n)}{n^{\frac{1}{2}}} 
\mathfrak{V}_{k+1}\left(\frac{Z m^2 n }{ N \vert D \vert p^{2(x + \beta)} } \right)\end{align*} 
if $\beta \geq 2$, and 
\begin{align*} \mathfrak{D}_{k+1}(x, \beta; Z) = \sum_{m \geq 1 \atop (m, N p^{\beta})=1} \frac{\omega(m)}{m} 
&\sum_{ {n \geq 1 \atop (n, p) =1} \atop m^2 n \equiv 1 \bmod p} \frac{r(n)\lambda(n)}{n^{\frac{1}{2}}} 
\mathfrak{V}_{k+1}\left(\frac{Z m^2 n }{ N \vert D \vert p^{2(x +\beta)} } \right) \\ &- \frac{2}{p-3} \sum_{m \geq 1 \atop (m, N)=1} \frac{\omega(m)}{m} 
\sum_{ n \geq 1 \atop m^2 n \not\equiv 1 \bmod p} \frac{r(n)\lambda(n)}{n^{\frac{1}{2}}} 
\mathfrak{V}_{k+1}\left(\frac{Z m^2 n }{ N \vert D \vert p^{2(x+\beta)} } \right)\end{align*} 
if $\beta =1$, again with no congruence conditions $\operatorname{mod} p^{\beta}$ if $\beta =0$. We also define the twisted sums 
\begin{align*} \widetilde{D}_{k+1}(\alpha, \beta; Z) = \frac{\omega(N) \tau(\omega)^4}{(\vert D \vert p^{\beta})^2} \left( \frac{p}{\varphi(p)} \right)
\sum_{m \geq 1 \atop (m, N p^{\beta})=1} &\frac{\omega(m)}{m} 
\sum_{ n \geq 1 \atop (n, p^{\alpha + \beta})=1} \frac{r(n)\lambda(n)}{n^{\frac{1}{2}}} V_{k+1} \left(\frac{m^2 n }{Z N \vert D \vert p^{2(\alpha + \beta)}} \right) \\
&\times  \operatorname{Kl}_4( \pm (m^2 n \overline{N}^2)^{\frac{1}{2}} \overline{D}, p^{\beta}) \end{align*} 
if $\beta \geq 2$, 
\begin{align*} \widetilde{D}_{k+1}(\alpha, \beta; Z) = \frac{\omega(N) \tau(\omega)^4}{(\vert D \vert p^{\beta})^2} 
\sum_{m \geq 1 \atop (m, N p^{\beta})=1} &\frac{\omega(m)}{m} 
\sum_{ n \geq 1 \atop (n, p^{\alpha + \beta})=1} \frac{r(n)\lambda(n)}{n^{\frac{1}{2}}} V_{k+1} \left(\frac{m^2 n }{N \vert D \vert p^{2(\alpha + \beta)}} \right) \\
&\times \left(  \operatorname{Kl}_4( \pm (m^2 n \overline{N}^2)^{\frac{1}{2}} \overline{D}, p) - \left( \frac{2}{p-3} \right) \right) \end{align*} 
for $\beta = 1$, and $\widetilde{D}_{k+1}(\alpha, 0; Z) = D_{k+1}(\alpha, 0; Z)$ for $\beta =0$. Similarly for $1 \leq x \leq \alpha - 1$, we define 
\begin{align*} \widetilde{\mathfrak{D}}_{k+1}(x, \beta; Z)  =
 \frac{\omega(N) \tau(\omega)^4}{(\vert D \vert p^{\beta})^2} \left( \frac{p}{\varphi(p)} \right) \sum_{m \geq 1 \atop (m, N p^{\beta})=1} &\frac{\omega(m)}{m} 
\sum_{ n \geq 1 \atop (n, p)=1} \frac{r(n)\lambda(n)}{n^{\frac{1}{2}}} \mathfrak{V}_{k+1} \left(\frac{m^2 n }{Z N \vert D \vert p^{2(x+\beta)}} \right) \\
&\times \operatorname{Kl}_4( \pm (m^2 n \overline{N}^2)^{\frac{1}{2}} \overline{D}, p^{\beta}) \end{align*} 
if $\beta \geq 2$, and 
\begin{align*} \widetilde{\mathfrak{D}}_{k+1}(x, \beta; Z) = \frac{\omega(N) \tau(\omega)^4}{(\vert D \vert p^{\beta})^2} 
 \sum_{m \geq 1 \atop (m, N p^{\beta})=1} &\frac{\omega(m)}{m} 
\sum_{ n \geq 1 \atop (n, p)=1} \frac{r(n)\lambda(n)}{n^{\frac{1}{2}}} \mathfrak{V}_{k+1} \left(\frac{m^2 n }{Z N \vert D \vert p^{2(x+ \beta)}} \right) \\
&\times \left( \operatorname{Kl}_4( \pm (m^2 n \overline{N}^2)^{\frac{1}{2}} \overline{D}, p) -\left( \frac{2}{p-3} \right)  \right) \end{align*} 
if $\beta = 1$, and $\widetilde{\mathfrak{D}}_{k+1}(\alpha, 0) = \mathfrak{D}_{k+1}(\alpha, 0)$ if $\beta = 0$. 

\end{proposition}\label{cycformula} 

\begin{proof} We simply evaluate  
\begin{align*} \frac{2}{\# \operatorname{Pic}(\mathcal{O}_{p^{\alpha}}) \varphi^{\star}(p^{\beta})} 
\sum_{\mathcal{W} \in X_{\alpha, \beta}} L^{(k)}(1/2, f \times \mathcal{W}) 
&= \frac{2}{\# \operatorname{Pic}(\mathcal{O}_{p^{\alpha}}) \varphi^{\star}(p^{\beta})} \left( S_{1, Z} - S_{2,Z} + S_{3, Z} - S_{4, Z} \right). \end{align*} \end{proof}

Skipping over the partial summation for ring class exponents $x$, we also derive the following simpler result: 

\begin{proposition}\label{CAF} 

Fix integers $\alpha \geq 0$ and $ \beta \geq 1$. Fix a primitive ring class character $\rho$ of conductor $p^{\alpha}$.
Given an integer $n \geq 1$ prime to $p$, let us then write 
\begin{align*} c_{\rho}(n) &= \sum_{A \in \operatorname{Pic}(\mathcal{O}_{p^{\alpha}})} r_A(n) \rho(A) \end{align*}
to denote the corresponding coefficient in the Dirichlet series expansion $(\ref{integralDirichlet})$. The one-variable average
\begin{align*} C(\rho, \beta) 
&:= \frac{2}{\varphi^{\star}(p^{\beta})} \sum_{\chi \bmod p^{\beta} \atop \chi(-1) = 1, \operatorname{primitive}} L(1/2, f \times \rho \chi \circ {\bf{N}}) \end{align*}
over primitive even Dirichlet characters $\chi \bmod p^{\beta}$ is given for any choice of real parameter $Z >0$ by the formula 
\begin{align*} C(\rho, \beta) &= D_1(\rho, \beta; Z) + \widetilde{D}_1(\rho, \beta; Z). \end{align*}
Here, the leading sum is defined by  
\begin{align*} D_1(\rho, \beta; Z) = \sum_{m \geq 1 \atop (m, pN)=1} \frac{\omega(m)}{m} 
&\sum_{ {n \geq 1 \atop (n, p) =1} \atop m^2 n \equiv 1 \bmod p^{\beta}} \frac{c_{\rho}(n)\lambda(n)}{n^{\frac{1}{2}}} 
V_{1}\left(\frac{Z m^2 n }{ N \vert D \vert p^{2(\alpha + \beta)} } \right) \\ 
& - \frac{1}{\varphi(p)} \sum_{m \geq 1 \atop (m, pN)=1} \frac{\omega(m)}{m} 
\sum_{ {n \geq 1 \atop m^2 n \equiv 1 \bmod p^{\beta-1}} \atop m^2 n \not \equiv 1 \bmod p^{\beta}} \frac{c_{\rho}(n)\lambda(n)}{n^{\frac{1}{2}}} 
V_{1}\left(\frac{Z m^2 n }{ N \vert D \vert p^{2(\alpha + \beta)} } \right) \end{align*} if $\beta \geq 2$, and 
\begin{align*} D_1(\rho, \beta; Z)  = \sum_{m \geq 1 \atop (m, pN)=1} \frac{\omega(m)}{m} 
&\sum_{ {n \geq 1 \atop (n, p^{\alpha + \beta}) =1} \atop m^2 n \equiv 1 \bmod p} \frac{c_{\rho}(n)\lambda(n)}{n^{\frac{1}{2}}} 
V_{1}\left(\frac{Z m^2 n }{ N \vert D \vert p^{2(\alpha + \beta)}} \right) \\ 
&- \frac{2}{p-3} \sum_{m \geq 1 \atop (m, pN)=1} \frac{\omega(m)}{m} 
\sum_{ n \geq 1 \atop m^2 n \not\equiv 1 \bmod p} \frac{r(n)\lambda(n)}{n^{\frac{1}{2}}} 
V_{1}\left(\frac{Z m^2 n }{ N \vert D \vert p^{2(\alpha + \beta)} } \right) \end{align*}
if $\beta =1$; the twisted sum is defined by 
\begin{align*} \widetilde{D}_1(\rho, \beta; Z) = \frac{\omega(N) \tau(\omega)^4}{(\vert D \vert p^{\beta})^2}  \left( \frac{p}{\varphi(p)} \right)
\sum_{m \geq 1 \atop (m, pN)=1} &\frac{\omega(m)}{m} 
\sum_{ n \geq 1 \atop (n, p)=1} \frac{c_{\rho}(n)\lambda(n)}{n^{\frac{1}{2}}} V_{1} \left(\frac{Z m^2 n }{N \vert D \vert p^{2(\alpha + \beta)}} \right) \\
&\times  \operatorname{Kl}_4( \pm (m^2 n \overline{N}^2)^{\frac{1}{2}} \overline{D}, p^{\beta}) \end{align*} 
if $\beta \geq 2$, and by 
\begin{align*} \widetilde{D}_1(\rho, \beta) = \frac{\omega(N) \tau(\omega)^4}{(\vert D \vert p^{\beta})^2} 
\sum_{m \geq 1 \atop (m, pN)=1} &\frac{\omega(m)}{m} 
\sum_{ n \geq 1 \atop (n, p)=1} \frac{c_{\rho}(n)\lambda(n)}{n^{\frac{1}{2}}} V_{1} \left(\frac{m^2 n }{N \vert D \vert p^{2(\alpha + \beta)}} \right) \\
&\times \left( \operatorname{Kl}_4( \pm (m^2 n \overline{N}^2)^{\frac{1}{2}} \overline{D}, p) - \left( \frac{2}{p-3} \right) \right) \end{align*} if $\beta = 1$.

\end{proposition}

\begin{proof} The proof works in the same way as for Proposition \ref{haf}, using the the approximate functional equation $(\ref{cvformula})$ with the 
orthogonality relations of $(\ref{QO})$ and $(\ref{QOp})$, but without taking the average over all ring class characters of conductor $p^{\alpha}$
(and hence without the need to use partial summation as in Lemma \ref{ps2v}). \end{proof}
  
\section{Self-dual estimates} 

Let us first estimate the average $H^{(k)}(\alpha, 0)$ for any $\alpha \geq 0$, taking the cyclotomic part to be trivial $\beta =0$. 
Hence, we average over ring class characters of conductor of a given conductor $p^{\alpha}$. Note that we could in fact fix any 
integer $\beta \geq 1$ in this discussion, and allow for $\alpha \geq 0$ to vary using the same method of approach. 
We treat this simpler setting to keep the exposition light, leaving the more general setting for the companion work \cite{VO3}.
We shall also take the unbalancing parameter $Z > 0$ in the approximate functional equation to be $Z=1$, i.e.~so that the 
approximate functional equation is balanced, and hence suppress the $Z$ from all notations for this section.
Our starting point here is the corresponding average formula of Proposition \ref{haf}.

\subsection{Strategy} 

We start with the following preliminary reduction, which is perhaps more of an observation:

\begin{lemma}\label{haf2} 

Keep the setup of Proposition \ref{haf}. Assume $\alpha \geq 1$. There exists a constant $\eta_0>0$ such that 
\begin{align*} H^{(k)}(\alpha, 0) = D_{k+1}(\alpha, 0) +(-1)^{k+1} \widetilde{D}_{k+1}(\alpha, 0) + O_{f, D}\left((p^{2\alpha})^{-\eta_0} \right)
&= 2D(\alpha, 0) + O((p^{-2 \alpha \eta_0})). \end{align*}  \end{lemma} 

\begin{proof} 

Let us first consider the weighting factors appearing in the extra term of Proposition \ref{haf}. Thus, fix an integral exponent 
$0 \leq x \leq \alpha -1$. Recall that $\# \operatorname{Pic}(\mathcal{O}_{p^x})$ is given by Dedekind's classical formula 
\begin{align}\label{Dedekind} \# \operatorname{Pic}(\mathcal{O}_{p^x}) &= \frac{h(\mathcal{O}_K) p^x }{[\mathcal{O}_K^{\times}: 
\mathcal{O}_{p^x}^{\times}]} \left( 1 - \left( \frac{D}{p}\right) \frac{1}{p}\right). \end{align} 
Here, $\mathcal{O}_K$ is the ring of integers of $K$, and $h(\mathcal{O}_K)$ the cardinality of its associated ideal class group. 
Since the unit index $[\mathcal{O}_K^{\times}:  \mathcal{O}_{p^x}^{\times}]$ is equal to one by our hypotheses, we find that 
\begin{align*} \frac{\# \operatorname{Pic}(\mathcal{O}_{p^{x}}) }{ \# \operatorname{Pic}(\mathcal{O}_{p^{\alpha}})} 
&= \left( \frac{p^{x}}{p^{\alpha}}\right) < 1. \end{align*} 
Using these observations, we claim that it will suffice will suffice to estimate the sums of terms 
\begin{align*} \mathfrak{D}_{k+1}(\alpha-1, 0) + (-1)^{k+1}\widetilde{\mathfrak{D}}_{k+1,}(\alpha - 1, 0) x= 2\mathfrak{D}_{k+1}(\alpha-1, 0). \end{align*} 
where
\begin{align}\label{sde} 2 \mathfrak{D}_{k+1}(\alpha - 1, 0) 
&= \sum_{m \geq 1 \atop (m, N)=1} \frac{\omega(m)}{m} \sum_{n \geq 1\atop (n, p^{\alpha-1}) =1} \frac{r(n) \lambda(n)}
{n^{\frac{1}{2}}} \mathfrak{V}_{k+1} \left( \frac{m^2 n}{N \vert D \vert p^{2(\alpha - 1)}}\right). \end{align} 
Now, we can estimate $(\ref{sde})$ in the same way as we shall estimate the leading sum $2 D_{k+1}(\alpha, 0; 1)$ 
below to deduce the result. In brief, we shall first consider the contribution from the $b=0$ terms in the parametrization 
$(\ref{count})$ of the counting function $r(n)$ described above. Expanding out these contributions, we obtain  
\begin{align*}2 \sum_{m \geq 1 \atop (m, N)=1} \frac{\omega(m)}{m} 
\sum_{a \geq 1 \atop (a, p^{\alpha - 1})=1} \frac{\lambda(a^2)}{a} \mathfrak{V}_{k+1} \left( \frac{m^2 a^2}{N \vert D \vert p^{2(\alpha -1)} }\right). \end{align*} 
By definition of the function $\mathfrak{V}_{k+1}$, this contribution is equivalent to
\begin{align*} \int_{\Re(s)=2}  \sum_{m \geq 1 \atop (m, N)=1} \frac{\omega(m)}{m^{2s+1}} \sum_{a \geq 1 \atop (a, p^{\alpha - 1})=1} 
\frac{\lambda(a^2)}{a^{2s+1}} \widehat{V}_{k+1}(s)(N\vert D \vert p^{2 (\alpha - 1)})^{s} \left(p^{-2s} -1 \right) \frac{ds}{2 \pi i}, \end{align*} 
and after identification with appropriate $L$-values to the integral 
\begin{align}\label{resint0} \int_{\Re(s)=2} \frac{L^{(N)}(2s+1, \omega)}
{\zeta^{(N)}(4s+2)} L^{(p^{\alpha-1})}(2s+1, \operatorname{Sym^2} f) \widehat{V}_{k+1}(s)
(N \vert D \vert p^{2(\alpha - 1)})^{s} \left(p^{-2s} -1 \right) \frac{ds}{2\pi i}. \end{align} 
Suppose now that we shift the range of integration to $\Re(s) = -2$, crossing a pole at $s=0$.
Observe that the residue of this pole must vanish thanks to the factor of $(p^{-2s} -1)$, which is not the case for the leading sums.
We can then deduce from the argument of Lemma \ref{residue} below that the remaining integrals are bounded, using
the fact that the cutoff functions $\mathfrak{V}_{k+1, j}$ away from this pole have the same decay behaviour as described 
in Lemma \ref{7.1} above for the cutoff functions $V_{k+1}$. The contribution from the $b \neq 0$ 
terms in the counting function $r(n)$ can then be bounded using the result of Proposition \ref{SDerror} below. \end{proof} 

So, Lemma \ref{haf2} gives us the preliminary estimate 
\begin{align*} H^{(k)}(\alpha, 0) = 2D_{k+1}(\alpha, 0) + O \left( p^{- \alpha\eta_0} \right). \end{align*}
Since $O \left(p^{-\alpha\eta_0} \right)$ tends to zero with $\alpha$, we reduce to estimating the leading sum
\begin{align*} 2 D_{k+1}(\alpha, 0) &= 2 \sum_{m \geq 1 \atop (m, N)=1} \frac{\omega(m)}{m} \sum_{n \geq 1 \atop (n, p^{\alpha})=1}
\frac{\lambda(n) r(n)}{n^{\frac{1}{2}}} V_{k+1} \left( \frac{m^2 n}{N \vert D \vert p^{2 \alpha}} \right) \end{align*} 
with $\alpha$. To do this, we first expand our sums in terms of the parametrization of $r(n)$ described in $(\ref{count})$, 
\begin{align*} r(n) &= \frac{1}{w} \# \left\lbrace (a, b) \in {\bf{Z}}^2 : q_1(a,b) = n \right\rbrace \end{align*}
for $q_1(a,b)$ the quadratic form defined in $(\ref{qf})$ above (see also the subsequent remark). Hence, we study
\begin{align}\label{Dexp} D_{k+1}(\alpha, 0) &= \frac{1}{w} \sum_{m \geq 1} \frac{\omega(m)}{m} \sum_{a, b \in {\bf{Z}} \atop q_1(a, b) \neq 0} 
\frac{\lambda(q_1(a, b))}{q_1(a, b)^{\frac{1}{2}}} V_{k+1} \left( \frac{m^2 q_1(a, b)}{N \vert D \vert p^{2 \alpha}} \right) \end{align}
Here, we first compute the contributions coming from the $b=0$ terms (cf.~\cite{Te} and \cite{TeT}),
these being the canonical residue terms for each of our estimates.  
The remaining contributions from $b \neq 0$ terms are then identified as certain Fourier-Whittaker coefficients
of certain other automorphic forms on $\operatorname{GL}_2({\bf{A}}_{\bf{Q}})$ or its metaplectic cover, 
and estimated using techniques from the spectral theory of shifted convolutions sums, in the style of \cite{TeT}.
Note that to derive bounds for these remaining sums, we shall require the best known approximations towards both the generalized
Ramanujan conjecture for $\operatorname{GL}_2({\bf{A}}_{\bf{Q}})$-automorphic forms, as well as the generalized Lindel\"of hypothesis
for $\operatorname{GL}_2({\bf{A}}_{\bf{Q}})$-automorphic forms in the level aspect. To be clear, we shall have to use the theorems of 
Kim-Sarnak \cite{KSh} and Blomer-Harcos \cite[Theorem 2]{BH07} for these respective approximations to derive suitable bounds for our estimates. 

\subsection{Calculation of the residue terms}  

Let us first describe the terms coming from the $b=0$ contributions in the counting functions $r(n)$ in the expansion $(\ref{Dexp})$ 
of $D_{k+1}(\alpha, 0)$. Recall that since the level $N$ of $f$ is assumed to be squarefree, we can define the symmetric square 
$L$-function $L(s, \operatorname{Sym^2} f)$ of $f$ by the Dirichlet series $(\ref{symm})$ above (first for $s \in {\bf{C}}$ with $\Re(s)>1$). 
Recall too that given an integer $M \geq 2$, we write $L^{(M)}(s, \operatorname{Sym^2} f)$ to denote the $L$-function $L(s, \operatorname{Sym^2} f)$ 
with the Euler factors at primes dividing $M$ removed.

\begin{lemma}\label{residue} We have the following estimates for the $b=0$ contributions in sums $2 D_{k+1}(\alpha, 0)$, i.e.~
after expanding out the counting function $r(n)$ according to the parametrization $(\ref{count})$ above: \\

\begin{itemize}

\item[(i)] If the pair $(f, \rho)$ is generic (i.e.~$k=0$), then for any choice of constant $C>0$, 
\begin{align*} \sum_{m \geq 1 \atop (m, N)=1} \frac{\omega(m)}{m} \sum_{a \in {\bf{Z}} }
\frac{\lambda(q_1(a, 0))}{q_1(a, 0)^{\frac{1}{2}}} V_{1} \left( \frac{m^2 q_1(a, 0)}{N \vert D \vert p^{2 \alpha}} \right) 
&= 2 L(1, \omega) \cdot \frac{L^{(p^{\alpha})}(1, \operatorname{Sym^2} f)}{\zeta^{(p^{\alpha})}(2)} + O_{C}\left( (\vert D \vert p^{2 \alpha})^{-C} \right). \end{align*}

\item[(ii)] If the pair $(f, \rho)$ is exceptional (i.e.~$k=1$), then for any choice of constant $C>0$,
\begin{align*} \sum_{m \geq 1 \atop (m, N)=1} \frac{\omega(m)}{m} \sum_{a \in {\bf{Z}} }
\frac{\lambda(q_1(a, 0))}{q_1(a, 0)^{\frac{1}{2}}} &V_{2} \left( \frac{m^2 q_1(a, 0)}{N \vert D \vert p^{2 \alpha}} \right) 
= 2 L(1, \omega) \cdot \frac{L^{(p^{\alpha})}(1, \operatorname{Sym^2} f )}{\zeta^{(p^{\alpha})}(2)} 
\left[ \frac{1}{2}\log(N \vert D \vert p^{2 \alpha}) + \frac{L'}{L}(1, \omega)  \right. \\  & \left. + \frac{L'^{(p^{\alpha})}}{L^{(p^{\alpha})}}(1, \operatorname{Sym^2} f) 
- \frac{\zeta'^{(p^{\alpha})}}{\zeta^{(p^{\alpha})}}(2) - \gamma- \log 2 \pi \right] + O_{C}( (\vert D \vert p^{2 \alpha})^{-C} ).\end{align*}

\end{itemize}

\end{lemma} 

\begin{proof} 

In either case on $k \in \lbrace 0, 1 \rbrace$, the contribution from $b=0$ to $2D_{k+1}(\alpha, 0)$ is given by the expression
\begin{align*} 2 \sum_{m \geq 1} \frac{\omega(m)}{m} \sum_{a \geq 1 \atop (a, p^{\alpha})=1} 
\frac{\lambda(a^2)}{a} V_{k+1}\left( \frac{m^2 a^2}{N\vert D \vert p^{2 \alpha}}\right), \end{align*} 
which after opening up the definition of the contour integral $V_{k+1}(y)$ is the same as the contour integral
\begin{align*} 2 \int_{\Re(s)=2}  \sum_{m \geq 1} \frac{\omega(m)}{m^{2s+1}}\sum_{a \geq 1 \atop (a, p^{\alpha})=1} 
\frac{\lambda(a^2)}{a^{2s+1}} \widehat{V}_{k+1}(s)(N\vert D \vert p^{2 \alpha})^{s} \frac{ds}{2 \pi i}, \end{align*} 
and which by $(\ref{symm})$ is the same as  
\begin{align}\label{resint1} 2 \int_{\Re(s)=2} \frac{L(2s+1, \omega)}{\zeta^{(p^{\alpha})}(4s+2)} L^{(p^{\alpha})}(2s+1, \operatorname{Sym^2} f) 
\widehat{V}_{k+1}(s)(N\vert D \vert p^{2 \alpha})^{s}\frac{ds}{2\pi i}. \end{align} 
Now, recall that we define 
\begin{align*} \widehat{V}_{k+1}(s) = \pi L_{\infty}(s+1/2)G_{k+1}(s) = \frac{L_{\infty}(s+1/2)}{L_{\infty}(1/2)}G_{k+1}(s),\end{align*}
where $G_{k+1}(s)= g^*(s)s^{-(k+1)}$ (with $g^*(0)=1$). Moving the line of integration leftward to $\Re(s) = -2$, 
we cross a pole at $s=0$ of the stated residue(s) given above. We refer to \cite[Lemma 7.2]{Te} for more details in the $k=1$ case,
where we use the fact that the Mellin transform $\widehat{V}_2(s)$ in the definition of $V_2(y)$ behaves as
\begin{align*} \widehat{V}_2(s) &= \frac{1}{s^2} - 2 \cdot \frac{\gamma + \log 2 \pi}{s} + O \left( 1 \right) \text{   as $s \rightarrow 0$}. \end{align*}
The remaining integral is seen easily to be bounded as $O_C \left( \left( N \vert D \vert p^{2 \alpha} \right)^{-C} \right)$ in either case using the 
Stirling approximation formula to describe $\widehat{V}_{k+1}(s)$ as $\Im(s) \rightarrow  \pm \infty$. \end{proof} 

\subsection{Main estimates}

Let us now estimate the contributions of $D_{k+1}(\alpha, 0)$ coming from $b \neq 0$ terms appearing in our parametrization 
$(\ref{count})$ of the counting function $r(n)$. To be more explicit, we now estimate 
\begin{align*} \sum_{m \geq 1} \frac{\omega(m)}{m} \sum_{a, b \in {\bf{Z}} \atop b \neq 0} 
\frac{\lambda(q_1(a, b))}{q_1(a, b)^{\frac{1}{2}}} V_{k+1} \left( \frac{m^2 q_1(a, b)}{N \vert D \vert p^{2 \alpha}} \right). \end{align*}
Observe that by the rapid decay of $V_{k+1}$, it will suffice to estimate the truncated sum defined by 
\begin{align}\label{trunc} D_{k+1}^{\dagger}(\alpha, 0) &= \sum_{m \geq 1} \frac{\omega(m)}{m} \sum_{b \in {\bf{Z}}, b \neq 0 \atop 
m^2 q_1(0, b) \leq (N \vert D \vert p^{2 \alpha})} \sum_{a \in {\bf{Z}}} 
\frac{\lambda(q_1(a, b))}{q_1(a, b)^{\frac{1}{2}}} V_{k+1} \left( \frac{m^2 q_1(a, b)}{N \vert D \vert p^{2 \alpha}} \right). \end{align}
To be clear, the rapid decay of the cutoff function $V_{k+1}$ reduces us to bounding the contribution from integers $m \geq 1$
and pairs of integers $(a, b)$ (with $b \neq 0$) such that $m^2 q_1(a, b) \leq N \vert D \vert p^{2 \alpha}$. As we shall see however, 
there is extra cancellation to detect in taking the sum over all integers $a \in {\bf{Z}}$ for each corresponding pair of admissible integers 
$(m, b)$ in this truncated sum. In this spirit, we shall bound the corresponding contributions in the style of Templier \cite[Proposition 7.2]{Te}, 
using the more general representation theoretic framework of Templier-Tsimerman \cite[Theorem 1, (6.13)]{TeT}. This allows us to derive suitable 
bounds in the setting of ring class exponent $\alpha = 0$ corresponding to class group characters. To deal with the setting of $\alpha \geq 1$ corresponding 
to ramified ring class characters, we shall have to exploit additional structure inherent in the $b$-sum to relate to constant coefficients of Eisenstein series. 

\begin{theorem}\label{SDerror} 

Suppose now that $f$ is any non-dihedral (possibly non-holomorphic) cuspidal modular form of arbitrary weight, level $N$, and nebentype character $\xi$,
with Fourier coefficients denoted by $\lambda(n) = \lambda_f(n)$. We have the following uniform estimates in either case on the generic root number 
$k \in \lbrace 0, 1 \rbrace$: \\

\begin{itemize}

\item[(i)] $D_{k+1}^{\dagger}(\alpha, 0) \ll_{f, \varepsilon} \left( N \vert D \vert p^{2 \alpha} \right)^{\frac{7}{16} + \varepsilon} \vert D \vert ^{-\frac{1}{2} + \varepsilon}$. \\

\item[(ii)] We have for some constant $\kappa >0$ the estimate 

\begin{align*} \sum_{m \geq 1} \frac{\omega \xi(m)}{m} \sum_{a, b \in {\bf{Z}} \atop q_1(a, b) \neq 0} \frac{\lambda(q_1(a, b))}{q_1(a, b)^{\frac{1}{2}}}
V_{k+1} \left( \frac{m^2 q_1(a, b)}{Y} \right) &= L(1, \omega \xi) \cdot 2 \mathfrak{L}^{(k)}(1, f) + O \left( Y^{-\kappa} \right), \end{align*} where
\begin{align*} \mathfrak{L}^{(k)}(1, f) &= \frac{2}{w} \cdot \begin{cases} \frac{ L(1, \operatorname{Sym}^2 f)}{L(2, \xi)} &\text{ if $k=0$} \\ 
\frac{L(1, \operatorname{Sym}^2 f )}{L(2, \xi)} \left[ \frac{1}{2}\log(N \vert D \vert p^{2 \alpha}) 
+ \frac{L'}{L}(1, \operatorname{Sym}^2 f) - \frac{L'}{L}(2, \xi) - \gamma- \log 2 \pi \right] &\text{ if $k=1$} \end{cases} \end{align*}
Here, the error term depends on the conductor of the form $f$ (and hence that of $\xi$) unless we assume that the absolute value 
of the fundamental discriminant $\vert D \vert$ is sufficiently large. 
\end{itemize}

\end{theorem} 

\begin{proof} 

Let us start with (i). Here, we modify the arguments of \cite[Proposition 7.2]{Te} and \cite[Theorem 2]{TeT} as follows, writing 
$\pi = \pi_f$ to denote the cuspidal $\operatorname{GL}_2({\bf{A}}_{\bf{Q}})$-automorphic form generated by $f$. Let us write $Y = N \vert D \vert p^{2 \alpha}$. 
Let us also fix an integral basis $[1, \vartheta]$ of $\mathcal{O}_K$, so that the truncated sum $D_{k+1}^{\dagger}(\alpha, 0)$ can be presented equivalently as 
\begin{align}\label{alttrunc} D_{k+1}^{\dagger}(\alpha, 0) &= \sum_{m \geq 1} \frac{\omega(m)}{m} 
\sum_{b \in {\bf{Z}}, b \neq 0 \atop b^2 m^2 \leq \frac{Y}{\vert \vartheta^2 \vert}}
\sum_{a \in {\bf{Z}}} \frac{\lambda(a^2 - b^2 \vartheta^2)}{(a^2 - b^2 \vartheta^2)^{\frac{1}{2}}} V_{k+1} \left( \frac{m^2 (a^2 - b^2 \vartheta^2)}{Y} \right).\end{align}
Fixing a pair of integers $(m, b)$ in the hyperboloid region $m b \leq \left( \frac{Y}{\vert \vartheta^2 \vert} \right)^{\frac{1}{2}}$ defining the second sum in this 
latter expression, the corresponding $a$-sum can be bounded by the argument of \cite[Theorem 1 and $\S 6.7$ (6.13)]{TeT} as 
\begin{align} \label{abound} \sum_{a \in {\bf{Z}}} \frac{\lambda(a^2 - b^2 \vartheta^2)}{(a^2 - b^2 \vartheta^2)^{\frac{1}{2}}} 
V_{k+1} \left( \frac{m^2 (a^2 - b^2 \vartheta^2)}{Y} \right) \ll_{\pi} \left( \frac{Y}{m^2} \right)^{\frac{1}{4}} \vert b^2 \vartheta^2 \vert^{\delta_0 - \frac{1}{2}}
\left( \frac{ \vert b^2 \vartheta^2 \vert m^2}{Y} \right)^{\frac{1}{2} - \frac{\theta_0}{2}}. \end{align}
Here, $0 \leq \theta_0 \leq 1/2$ denotes the best known approximation towards the generalized Ramanujan conjecture for 
$\operatorname{GL}_2({\bf{A}}_{\bf{Q}})$-automorphic forms (with $\theta_0 = 0$ conjectured). Note that we can take $\theta_0 = 7/64$
thanks to Kim-Sarnak \cite{KSh}. As well, $0 \leq \delta_0 \leq 1/4$ denotes the best known approximation towards the generalized 
Lindel\"of hypothesis for $\operatorname{GL}_2({\bf{A}}_{\bf{Q}})$-automorphic forms in the level aspect (with $\delta_0 = 0$ conjectured).
We can and later do take this exponent to be $\delta_0 = 3/16$ thanks to Blomer-Harcos \cite{BH07}. 
Using $(\ref{abound})$ in the presentation $(\ref{alttrunc})$, we then obtain the bound 
\begin{align*} D_{k+1}^{\dagger}(\alpha, 0)  \ll_{\pi} Y^{ \frac{\theta_0}{2} - \frac{1}{4} } 
\vert \vartheta^2 \vert^{\delta_0 - \frac{\theta_0}{2} } \sum_{m \geq 1} 
\sum_{b \neq 0 \in {\bf{Z}} \atop m b \leq \left(  \frac{Y}{\vert \vartheta^2 \vert } \right)^{ \frac{1}{2} } } 
\frac{b^{2 \left( \delta_0 - \frac{\theta_0}{2} \right)}}{m^{\frac{1}{2} + \frac{\theta_0}{2}}},\end{align*}
which after using that $\vartheta^2 = O( \vert D \vert )$ and that the number of lattice points under the hyperbola $m b = r$ 
is bounded above by $r \log r$ gives 
\begin{align*} D_{k+1}^{\dagger}(\alpha, 0)  \ll_{\pi, \varepsilon} Y^{ \frac{\theta_0}{2} - \frac{1}{4} } 
\vert D \vert^{\delta_0 - \frac{\theta_0}{2} }
\left( \frac{Y}{ \vert D \vert} \right)^{\frac{1}{2} + (\delta_0 - \frac{\theta_0}{2}) + \varepsilon}
= Y^{\delta_0 + \frac{1}{4} + \varepsilon} \vert D \vert^{-\frac{1}{2} + \varepsilon}. \end{align*} 
Hence, expanding out $Y = N \vert D \vert p^{2 \alpha}$, and using the admissible exponent $\delta_0 = 3/16$, we obtain 
\begin{align*} D_{k+1}^{\dagger}(\alpha, 0) 
\ll_{\pi, \varepsilon} \left( N \vert D \vert p^{2 \alpha} \right)^{\frac{7}{16} + \varepsilon} \vert D \vert^{-\frac{1}{2} + \varepsilon}. \end{align*}

To deal with the more general case of ramification ($p^{2 \alpha} \gg \vert D \vert$) for (ii), we must first recall how the proof of 
$(\ref{abound})$ works, following the arguments of \cite[Theorem 1]{TeT} and \cite[$\S 6$, (6.13)]{TeT}. In fact, we provide  
additional justification for the stated estimate of \cite[Theorem 1]{TeT} by explaining how the integral presentations work for cuspidal 
$\operatorname{GL}_2({\bf{A}}_{\bf{Q}})$-automorphic forms, as the proofs in \cite{TeT} in this respect are only given 
for cuspidal forms on $\operatorname{SL}_2({\bf{R}})$. In doing this, we shall also explain how to relate both the truncated 
sum $D_{k+1}^{\dagger}(\alpha, 0)$ and the full sum $D_{k+1}(\alpha, 0)$ to Fourier-Whittaker coefficients of certain forms 
forms on $\operatorname{GL}_2({\bf{A}}_{\bf{Q}})$ and its metaplectic cover. To begin, recall that a decomposable new vector 
$\phi \in V_{\pi}$ can be viewed as a cuspidal $\operatorname{GL}_2({\bf{A}}_{\bf{Q}})$-automorphic form,
and as such has the following expansion (cf.~\cite[$\S 2.5$]{TeT} with \cite[$\S 2.5$]{BH10}): Given $x \in {\bf{A}}_{\bf{Q}}$ an adele, 
and $y = y_f y_{\infty} \in {\bf{A}}_{\bf{Q}}^{\times}$ an idele split into its finite component $y_f \in {\bf{A}}_{ {\bf{Q}}, f }^{\times}$ times its 
real component $y_{\infty} \in {\bf{R}}^{\times}$, and writing $\vert y \vert$ to denote the idele norm, we have
\begin{align*} \phi \left(  \left(\begin{array} {cc} y &  x \\ 0 & 1 \end{array}\right) \right) 
&= \sum_{\gamma \in {\bf{Q}}^{\times}} \rho_{\phi}(\gamma y_f) W_{\phi}(\gamma y_{\infty}) e(\gamma x) =
\rho_{\phi}(1) \sum_{\gamma \in {\bf{Z}} \atop \gamma \neq 0} \frac{ \lambda (\gamma y_f)}{ \vert \gamma y_f \vert^{\frac{1}{2}} } 
W_{\phi}(\gamma y_{\infty}) e(\gamma x), \end{align*}
where the sum is supported on rational integers $\gamma \neq 0 \in {\bf{Z}}$, and the Whittaker integral $W_{\infty}(y_{\infty})$ is given by 
\begin{align*} W_{\phi}(y_{\infty}) &= W_{\phi} \left( \left(\begin{array} {cc} y_{\infty} &  ~ \\ ~ & 1 \end{array}\right) \right)
= \int_{  {\bf{A}}_{\bf{Q}} / {\bf{Q}} } \phi \left( \left(\begin{array} {cc} y_{\infty} &  x \\ 0 & 1 \end{array}\right) \right) e(-x) dx. \end{align*}
Let us now consider the theta series $\theta_Q$ attached to the quadratic form $Q(\gamma) = \gamma^2$. 
Viewed as a genuine automorphic form on the metaplectic cover of $\operatorname{GL}_2({\bf{A}}_{\bf{Q}})$, i.e.~corresponding 
to a modular form of half-integral weight, $\theta_{Q}$ has the Fourier-Whittaker expansion 
\begin{align*} \theta_Q \left( \left(\begin{array} {cc} y &  x \\ 0 & 1 \end{array}\right) \right) 
&= \vert y \vert^{\frac{1}{4}} \sum_{\gamma \in {\bf{Q}}} e \left( Q(\gamma)(x + i y) \right) 
= \vert y \vert^{\frac{1}{4}} \sum_{\gamma \in {\bf{Z}} } e \left( Q(\gamma)(x + i y) \right). \end{align*}
Using the orthogonality of additive characters on the compact abelian group ${\bf{A}}_{\bf{Q}} / {\bf{Q}}$, 
these characters being parametrized as $x \mapsto e(\gamma x)$ for $\gamma \in {\bf{Q}}$ varying, 
it is easy to see that for any vector $\phi \in V_{\pi}$, we have  
\begin{align*} \int_{ {\bf{A}}_{\bf{Q}} / {\bf{Q}} } &\phi \overline{\theta}_Q 
\left( \left(\begin{array} {cc} y_{\infty} &  x \\ 0 & 1 \end{array}\right) \right) e \left( - b^2 \vartheta^2 x \right) dx 
= \int_{ {\bf{A}}_{\bf{Q}} / {\bf{Q}} } \phi \left( \left(\begin{array} {cc} y_{\infty} &  x \\ 0 & 1 \end{array}\right) \right) 
\overline{\theta}_Q \left( \left(\begin{array} {cc} y_{\infty} &  x \\ 0 & 1 \end{array}\right) \right) e \left( - b^2 \vartheta^2 x \right) dx \\ 
&= \vert y_{\infty} \vert^{\frac{1}{4}} \rho_{\phi}(1) \int_{ {\bf{A}}_{\bf{Q}} / {\bf{Q}} } 
\sum_{\gamma_1 \in {\bf{Z}} \atop \gamma_1 \neq 0} \frac{ \lambda (\gamma_1)}{ \vert \gamma_1 \vert^{\frac{1}{2}} } 
W_{\phi}(\gamma_1 y_{\infty}) e(\gamma_1 x) \sum_{\gamma_2 \in {\bf{Z}}} e \left( i Q(\gamma_2) y_{\infty} \right) e(-Q(\gamma_2) x) e(-b^2 \vartheta^2 x) dx \\
&=  \vert y_{\infty} \vert^{\frac{1}{4}} \rho_{\phi}(1) \sum_{\gamma_1 \in {\bf{Z}} \atop \gamma_1 \neq 0} \frac{ \lambda (\gamma_1)}{ \vert \gamma_1 \vert^{\frac{1}{2}} } 
W_{\phi}(\gamma_1 y_{\infty}) \sum_{\gamma_2 \in {\bf{Z}}} e \left( i Q(\gamma_2) y_{\infty} \right)  
\int_{ {\bf{A}}_{\bf{Q}} / {\bf{Q}} } e \left( \gamma_1 x - Q(\gamma_2)x - b^2 \vartheta^2 x \right) dx \\ &= 
\vert y_{\infty} \vert^{\frac{1}{4}} \rho_{\phi}(1) \sum_{\gamma \in {\bf{Z}} \atop Q(\gamma) - b^2 \vartheta^2 \neq 0} 
\frac{\lambda (Q(\gamma) - b^2 \vartheta^2)}{ \vert Q(\gamma) - b^2 \vartheta^2 \vert^{\frac{1}{2}}} 
W_{\phi}\left( (Q(\gamma) - b^2 \vartheta^2) y_{\infty} \right) e (i y _{\infty} Q(\gamma)). \end{align*}
Now, the surjectivity of the archimedean Kirillov map can be used to show the following useful consequence 
(see \cite[Proposition 2.1]{TeT} and \cite[(37) and Lemma 3]{BH10}): Given an arbitrary smooth function $W \in L^2({\bf{R}}_{>0})$, 
there exists a new vector $\phi \in V_{\infty}$ such that $W_{\infty}(y_{\infty}) = W(y_{\infty})$ for any $y_{\infty} \in {\bf{R}}_{>0}$. 
We use this property as follows to derive a suitable integral presentation for each $a$-sum corresponding to each admissible 
pair $(m, b)$ in the truncated sum $D_{k+1}^{\dagger}(\alpha, 0)$ as follows. Fixing an integer $b$ in this sum, choosing 
$\phi_{b} \in V_{\pi}$ in such a way that $\rho_{\phi_b}(1)W_{\phi}(y_{\infty}) = V_{k+1}(y_{\infty}) e(-i y_{\infty}) e \left( i \cdot \frac{b^2 \vartheta^2}{Y} \right)$, 
we can then specialize the identity above to $y_{\infty} = \frac{m^2}{Y}$ to derive the identity 
\begin{align}\label{metID} \int_{ {\bf{A}}_{\bf{Q}} / {\bf{Q}} } \phi_b \overline{\theta}_Q 
\left( \left(\begin{array} {cc} \frac{m^2}{Y} &  x \\ 0 & 1 \end{array}\right) \right) e \left( - b^2 \vartheta^2 x \right) dx &=
\left( \frac{m^2}{Y} \right)^{\frac{1}{4}} \sum_{a \in {\bf{Z}} \atop a^2 - b^2 \vartheta^2 \neq 0} 
\frac{\lambda(a^2 - b^2 \vartheta^2)}{(a^2 - b^2 \vartheta^2)^{\frac{1}{2}}} V_{k+1} \left( \frac{m^2(a^2 - b^2 \vartheta^2)}{Y} \right).\end{align}
The bound of \cite[Theorem 1]{TeT} and more specifically \cite[$\S 6.7$, (6.13)]{TeT} applies to the integral on the left hand side 
of this identity $(\ref{metID})$, which has a natural identification as the Fourier-Whittaker coefficient at $ b^2 \vartheta^2$ of
the genuine metaplectic form $\phi_b \overline{\theta}_Q$ (which corresponds to a modular form of half-integral weight). 
In this way, it is clear that the truncated sum $D_{k+1}^{\dagger}(\alpha, 0)$ could be viewed as a finite sum of these coefficients,
although (as shown in (i)) the bounds we derive do not suffice to handle ramification when $\alpha \geq 1$. To deal with this ramified
setting, we exploit extra structure inherent in the $b$-sum, which is not seen via the metaplectic theta series $\theta_Q$. Rather, we 
shall now consider the binary theta series $\theta_{q_1}$ corresponding to the binary quadratic form $q_1$ defined in $(\ref{qf})$ above. 
Note that this theta series appears throughout the classical literature, and gives rise to a holomorphic modular form of weight one, level $\vert D \vert$,
and nebentype character $\omega = \omega_D$. Viewed as a $\operatorname{GL}_2({\bf{A}}_{\bf{Q}})$-automorphic form, 
it has (in the notations defined above) the Fourier-Whittaker expansion
\begin{align*} \theta_{q_1} \left( \left(\begin{array} {cc} y &  x \\ 0 & 1 \end{array}\right) \right) 
&= \vert y \vert^{\frac{1}{2}} \frac{1}{w} \sum_{\gamma_1, \gamma_2 \in {\bf{Q}}} e \left( q_1(\gamma_1, \gamma_2)(x + i y) \right) 
= \vert y \vert^{\frac{1}{2}} \frac{1}{w} \sum_{\gamma_1, \gamma_2 \in {\bf{Z}} } e \left( q_1(\gamma_1, \gamma_2)(x + i y) \right). \end{align*}
Using orthogonality of additive characters again, it is easy to see that for any $\phi \in V_{\pi}$ and $y_{\infty} \in {\bf{R}}^{\times}$ we have 
\begin{align*} \int_{ {\bf{A}}_{\bf{Q}} / {\bf{Q}} } \phi \overline{\theta}_{q_1} \left( \left(\begin{array} {cc} y_{\infty} &  x \\ 0 & 1 \end{array}\right) \right) dx  
= \rho_{\phi}(1) \vert y_{\infty} \vert^{\frac{1}{2}} \frac{1}{w} \sum_{a, b \in {\bf{Z}} \atop q_1(a, b) \neq 0}\frac{\lambda(q_1(a, b))}{q_1(a, b)^{\frac{1}{2}}} 
W_{\phi}(q_1(a, b) y_{\infty}) e(i q_1(a, b) y_{\infty}).\end{align*}
Choosing $\phi \in V_{\pi}$ so that $\rho_{\phi}(1) W_{\phi}(y_{\infty}) 
= V_{k+1}(y_{\infty}) e(-i y_{\infty})$, we then derive for any $y_{\infty} \in {\bf{R}}_{>0}$ the identity 
\begin{align}\label{doubleIP} \int_{ {\bf{A}}_{\bf{Q}} / {\bf{Q}} } \phi \overline{\theta}_{q_1} \left( \left(\begin{array} {cc} y_{\infty} &  x \\ 0 & 1 \end{array}\right) \right) dx  
= \vert y_{\infty} \vert^{\frac{1}{2}} \frac{1}{w} \sum_{a, b \in {\bf{Z}} \atop q_1(a, b) \neq 0}\frac{\lambda (q_1(a, b))}{q_1(a, b)^{\frac{1}{2}}} V_{k+1}(q_1(a, b) y_{\infty}).\end{align}
Now, observe that the integral in $(\ref{doubleIP})$ is simply the constant coefficient $\rho_{\phi \overline{\theta}_{q_1}, 0}(y)$ 
of the $\operatorname{GL}_2({\bf{A}}_{\bf{Q}})$-automorphic form $\phi \overline{\theta}_{q_1}$ evaluated at $y = y_{\infty}$. 
That is, the Fourier-Whittaker expansion of $\phi \overline{\theta}_{q_1}$ is given by 
\begin{align*} \phi \overline{\theta}_{q_1} \left( \left(\begin{array} {cc} y &  x \\ 0 & 1 \end{array}\right) \right) 
&= \rho_{\phi \overline{\theta}_{q_1}, 0}(y) + \sum_{\gamma \in {\bf{Q}}^{\times}} \rho_{\phi \overline{\theta}_{q_1}}(\gamma y_f) 
W_{\phi \overline{\theta}_{q_1}}(\gamma y_{\infty}) e(\gamma x), \end{align*} where 
\begin{align*} \rho_{\phi \overline{\theta}_{q_1}, 0}(y) 
&= \int_{ {\bf{A}}_{\bf{Q}} / {\bf{Q}} } \phi \overline{\theta}_{q_1} \left( \left(\begin{array} {cc} y &  x \\ 0 & 1 \end{array}\right) \right) dx.\end{align*}
As such, $\rho_{\phi \overline{\theta}_{q_1}, 0}(y)$ is a smooth function of $y$, of moderate growth. 
Now, observe that if we specialize $(\ref{doubleIP})$ to $y_{\infty} = \frac{m^2}{Y}$ for {\it{any}} real parameter $Y >0$, 
then take the (obviously-twisted) sum over $m$, we obtain
\begin{align*} \sum_{m \geq 1} \frac{\omega \xi(m)}{m^2} \left( Y^{\frac{1}{2}} \rho_{\phi \overline{\theta}_{q_1}, 0} \left( \frac{m^2}{Y} \right) \right)
&= \sum_{m \geq 1} \frac{\omega \xi (m)}{m} \frac{1}{w} \sum_{a, b \in {\bf{Z}} \atop q_1(a, b) \neq 0}
\frac{\lambda(q_1(a, b))}{q_1(a, b)^{\frac{1}{2}}} V_{k+1} \left( \frac{m^2 q_1(a, b)}{Y} \right). \end{align*} 
For instance, if we specialize to $Y = N \vert D \vert$ and use the estimates of Lemma \ref{residue} and (i), we then derive
\begin{align*} Y^{\frac{1}{2}} \rho_{\phi \overline{\theta}_{q_1}, 0} \left( \frac{1}{Y} \right) &= 2 \mathfrak{L}^{(k)}(1, f)
+ O_{\pi, \varepsilon} \left( N^{\frac{7}{16} + \varepsilon} \vert D \vert^{-\frac{1}{16} - \varepsilon} \right). \end{align*}
Beware however that the error term in this latter estimate depends on the conductor our given form $\pi$, 
which would limit its applicability to our subsequent arguments when $f$ has nontrivial nebentype 
character\footnote{However, we can and do use this for our estimates over ring class characters. In the setting 
where we average over Dirichlet or cyclotomic characters, we need to assume additionally 
that $\vert D \vert$ is sufficiently large to derive a suitable nonvanishing estimate.}. 
To get around this, we can derive an estimate by hand. To be clear, we wish to show that as a function of a nonzero real variable $y_{\infty}$, 
the coefficient $\rho_{\phi \overline{\theta}_{q_1}, 0}(y_{\infty})$ cannot be identically zero, and moreover is nonvanishing for all $y_{\infty}$ of 
norm $\vert y_{\infty} \vert$ in some region of ${\bf{R}}_{>0}$. We can then use that this nonzero function is smooth and of moderate growth to 
derive its limiting behaviour as desired. In particular, we are not restricted to a specific choice of of $y_{\infty}$ in this approach, so long as we 
derive a nonvanishing estimate in some suitable region. To this end, we start with the Epstein-like expansion $(\ref{doubleIP})$, 
\begin{align*} \rho_{\phi \overline{\theta}_{q_1}, 0} \left( y_{\infty} \right) 
&= \vert y_{\infty} \vert^{\frac{1}{2}}  \frac{1}{w} \sum_{a, b \in {\bf{Z}} \atop q_1(a, b) \neq 0} 
\frac{\lambda(q_1(a, b))}{q_1(a, b)^{\frac{1}{2}}} V_{k+1} \left( q_1(a, b) \vert y_{\infty} \vert \right)  \\ 
&= \vert y_{\infty} \vert^{\frac{1}{2}} \int_{\Re(s) = 2} \widehat{V}_{k+1}(s) \frac{1}{w} 
\sum_{a, b \in {\bf{Z}} \atop q_1(a, b) \neq 0} \frac{\lambda(q_1(a,b))}{q_1(a,b)^{\frac{1}{2}}} \left( q_1(a, b) \vert y_{\infty} \vert \right)^{-s} \frac{ds}{2 \pi i}.\end{align*}
Let us first consider the contribution $\rho_{\phi \overline{\theta}_{q_1}, 0} \left( y_{\infty} \right)\vert_{b = 0}$ from $b =0$ terms in in this expression for the coefficient 
$\rho_{\phi \overline{\theta}_{q_1}, 0} \left( y_{\infty} \right)$, which is seen easily by the residue calculations above to have the integral presentation 
\begin{align*} \rho_{\phi \overline{\theta}_{q_1}, 0} \left( y_{\infty} \right)\vert_{b = 0} 
&= \vert y_{\infty} \vert^{\frac{1}{2}} \int_{\Re(s) = 2} \widehat{V}_{k+1}(s) 
\frac{2}{w} \frac{L(1 + 2s, \operatorname{Sym}^2 f)}{L(2 + 4s, \xi)} \vert y_{\infty} \vert^{-s} \frac{ds}{2 \pi i}. \end{align*} 
Shifting the contour to the right to estimate the behaviour as $\vert y_{\infty} \vert \rightarrow \infty$, and to the left to estimate the 
behaviour as $\vert y_{\infty} \vert \rightarrow 0$, we can derive the familiar estimate 
\begin{align*} \rho_{\phi \overline{\theta}_{q_1}, 0}(y_{\infty})\vert_{b=0} &= \vert y_{\infty} \vert^{\frac{1}{2}} \cdot \begin{cases} 
\mathfrak{L}^{(k)}(1, f) + O_A\left(\vert y_{\infty} \vert^A \right) &\text{ for any $A \geq 1$ as $\vert y_{\infty} \vert \rightarrow 0$} \\
O_C \left( \vert y_{\infty} \vert^{-C} \right) &\text{ for any $C >0$ as $\vert y_{\infty} \vert \rightarrow \infty$}. \end{cases} \end{align*}
To derive similar contour estimates for the remaining contributions of $b \neq 0$ terms 
\begin{align*} \rho_{\phi \overline{\theta}_{q_1},0}(y_{\infty})\vert_{b \neq 0} &= \vert y_{\infty} \vert^{\frac{1}{2}}  \int_{\Re(s) = 2} \widehat{V}_{k+1}(s) 
\frac{1}{w} \sum_{a, b \in {\bf{Z}} \atop q_1(a, b) \neq 0} \frac{\lambda(q_1(a, b))}{q_1(a, b)^{\frac{1}{2}}} \left( q_1(a, b) \vert y_{\infty} \vert \right)^{-s} \frac{ds}{2 \pi i}, \end{align*}
we first observe that shifting contours to the right gives us the unsurprising bound of $O_C\left( \vert y_{\infty} \vert^{\frac{1}{2}} \cdot \vert D y_{\infty} \vert^{-C} \right)$ 
for any choice of constant $C >0$ as $\vert y_{\infty} \vert \rightarrow \infty$. In fact, opening up the contribution of the quadratic form $q_1$ as described in $(\ref{qf})$, 
we can refine this latter estimate to include any $y_{\infty}$ region $\vert D \vert^{-1} < \vert y_{\infty} \vert < 1$. Putting this together with our estimate for the $b=0$ 
contributions in this same region, we then derive the following nonvanishing estimate for the coefficient: 
For any choices of constants $A \geq 1$ and $C > 0$, we have for $y_{\infty}$ in the region $\vert D \vert^{-1} < \vert y_{\infty} \vert < 1$ that 
\begin{align*}  \vert y_{\infty} \vert^{-\frac{1}{2}} \cdot \rho_{\phi \overline{\theta}_{q_1},0}(y_{\infty}) 
&= \mathfrak{L}^{(k)}(1, f) + O_A(\vert y_{\infty} \vert^A) + O_C(\vert D y_{\infty} \vert^{-C}). \end{align*} 
In particular, taking $Y = \vert D \vert^{-1/2}$ (for instance), we deduce that the constant coefficient is nonvanishing for $\vert D \vert \gg 1$. 
In this way, we can deduce from the smoothness and moderate growth of $\rho_{\phi \overline{\theta}_{q_1}}$ that 
\begin{align*} \lim_{Y \rightarrow \infty} Y^{\frac{1}{2}} \rho_{\phi \overline{\theta}_{q_1}, 0} \left( \frac{1}{Y} \right)  = \mathfrak{L}^{(k)}(1, f), \end{align*}
and by Abel summation that 
\begin{align*} \lim_{Y \rightarrow \infty} \sum_{m \geq 1} \frac{\omega \xi(m)}{m} \left( \left( \frac{Y}{m^2} \right)^{\frac{1}{2}}
\rho_{\phi \overline{\theta}_{q_1}, 0} \left( \frac{m^2}{Y} \right)  \right) = L(1, \omega \xi) \cdot \mathfrak{L}^{(k)}(1, f). \end{align*}
It is then easy to deduce that for $Y = N \vert D \vert p^{2 \alpha}$, there exists a constant $\kappa >0$ for which
\begin{align*} \frac{1}{w} \sum_{m \geq 1} \frac{\omega \xi(m)}{m} \sum_{a, b \in {\bf{Z}} \atop q_1(a, b) \neq 0} \frac{\lambda(q_1(a, b))}{q_1(a, b)^{\frac{1}{2}}}
V_{k+1} \left( \frac{m^2 q_1(a, b)}{Y} \right) &= L(1, \omega \xi) \cdot \mathfrak{L}^{(k)}(1, f)
+ O \left( Y^{-\kappa} \right).\end{align*}
\end{proof}

Putting this together with Lemmas \ref{haf2} and \ref{residue}, we obtain the following estimates. 

\begin{theorem}\label{SDave} Let $\alpha \geq 0$ be any integer. We have the following estimates for the averages $H^{(k)}(\alpha, 0)$: \\

\begin{itemize}

\item[(i)] If $k= 0$, then for some constant $\eta_0 >0$ we have  
\begin{align*} H^{(0)}(\alpha, 0) &= \frac{4}{w} \cdot L(1, \omega) \cdot \frac{L^{(p^{\alpha})}(1, \operatorname{Sym}^2 f)}{\zeta^{(p^{\alpha})}(2)} 
+ O \left( (N \vert D \vert p^{2 \alpha})^{-\eta_0} \right). \end{align*}

\item[(ii)] If $k=1$, then for some constant $\eta_1 >0$ we have 
\begin{align*} H^{(1)}(\alpha, 0) &= \frac{4}{w} L(1, \omega) \cdot \frac{L^{(p^{\alpha})}(1, \operatorname{Sym}^2 f )}{\zeta^{(p^{\alpha})}(2)} 
\left[ \frac{1}{2}\log(N\vert D \vert p^{2 \alpha}) + \frac{L'}{L}(1, \omega)  \right. \\  & \left. + \frac{L'^{(p^{\alpha})}}{L^{(p^{\alpha})}}(1, \operatorname{Sym}^2 f) 
- \frac{\zeta'^{(p^{\alpha})}}{\zeta^{(p^{\alpha})}}(2) - \gamma- \log 2 \pi \right] + O((N \vert D \vert p^{2 \alpha})^{-\eta_1} ). \end{align*}

\end{itemize}

In particular, if the quantity $N \vert D \vert p^{2\alpha}$ is sufficiently large, then the average $H^{(k)}(\alpha, 0)$ does not vanish. \end{theorem}

\section{Non self-dual estimates}  

Fix integers $\alpha \geq 0$ and $\beta \geq 4$. We now estimate the averages $H^{(0)}(\alpha, \beta)$ 
for a fixed ring class exponent $\alpha \geq 0$, with varying cyclotomic exponent $\beta \geq 4$. 
Let us retain the setup of Propositions \ref{RT} and \ref{CAF}. We shall derive bounds for the respective sums 
$\widetilde{D}_1(\alpha, \beta; Z)$ (and $\widetilde{D}_1(\rho, \beta; Z)$) and $D_1(\alpha, \beta; Z)$ with unbalancing parameter $Z = p^{ \frac{\beta}{8}}$.
To be more precise, we shall evaluate the hyper-Kloosterman sums in the contragredient term $\widetilde{D}_1(\alpha, \beta; p^{\frac{\beta}{8}})$ 
explicitly to derive an equivalent expression in terms of quadratic Weyl sums, which allows us to use Weyl differencing to obtain a sufficient bound. 
Having reduced the length sufficiently, we can then give a relatively elementary (albeit delicate) argument to estimate the leading term 
$D_1(\alpha, \beta; p^{\frac{\beta}{8}})$. 

\subsection{The twisted sums (Weyl differencing of $p$-adic phase)} 

Fix integers $\alpha \geq 0$ and $\beta \geq 2$. Let us also fix a primitive ring class character $\rho$ of conductor $p^{\alpha}$. 
We now consider the twisted sums
\begin{align*}\widetilde{D}_{1}(\alpha, \beta; p^{\frac{\beta}{8}}) = \frac{\omega(N) \tau(\omega)^4}{(\vert D \vert p^{\beta})^2} \left( \frac{p}{\varphi(p)} \right)
\sum_{m \geq 1 \atop (m,  p)=1} 
&\frac{\omega(m)}{m} \sum_{ n \geq 1 \atop (n, p)=1} \frac{r(n)\lambda(n)}{n^{\frac{1}{2}}} 
V_1 \left(\frac{ m^2 n }{ N \vert D \vert p^{2 \alpha} p^{\frac{17}{8} \beta} } \right) \\
&\times \operatorname{Kl}_4( \pm (m^2 n \overline{N}^2)^{\frac{1}{2}} \overline{D}, p^{\beta}). \end{align*} 
and 
\begin{align*}\widetilde{D}_{1}(\rho, \beta; p^{\frac{\beta}{8}}) 
= \frac{\omega(N) \tau(\omega)^4}{(\vert D \vert p^{\beta})^2} \left( \frac{p}{\varphi(p)} \right) 
\sum_{m \geq 1 \atop (m, p)=1} &\frac{\omega(m)}{m} \sum_{ n \geq 1 \atop (n, p)=1} \frac{c_{\rho}(n)\lambda(n)}{n^{\frac{1}{2}}} 
V_{1} \left(\frac{m^2 n }{N \vert D \vert p^{2\alpha} p^{\frac{17}{8} \beta} } \right) \\
&\times \operatorname{Kl}_4( \pm (m^2 n \overline{N}^2)^{\frac{1}{2}} \overline{D}, p^{\beta}) \end{align*} 
from the formulae of Propositions \ref{haf} and \ref{CAF} respectively. Here again, we use the shorthand notation
\begin{align*} c_{\rho}(n) &= \sum_{A \in \operatorname{Pic}(\mathcal{O}_{p^{\alpha}})} r_A(n) \rho(A) \end{align*}
to denote the corresponding coefficient in the Dirichlet series expansion $(\ref{integralDirichlet})$
of the $L$-function $L(s, f \times \rho \chi \circ {\bf{N}})$ (this being independent of the choice of Dirichlet character $\chi$). 
Note that we have for each $n \geq 1$ the relation\footnote{Note that this can be viewed as a Siegel-Weil type relation, 
linking the sum over ring class theta series $\sum_A \theta_A$ with some Eisenstein series. In this way, the sums 
$D_1(\alpha, \beta; p^{\frac{\beta}{8}}) + \widetilde{D}_1(\alpha, \beta; p^{\frac{\beta}{8}})$ 
can be thought of as averages over primitive Dirichlet characters of certain Eisenstein series,
where we are exploiting the simple nature of the coefficients ($r(n)$) to derive suitable estimates.} 
\begin{align}\label{S-W} r(n) 
&= \frac{1}{\# \operatorname{Pic}(\mathcal{O}_{p^{\alpha}})} 
\sum_{\rho \in \operatorname{Pic}(\mathcal{O}_{p^{\alpha}})^{\vee}} c_{\rho}(n).\end{align}
Indeed, this is easy to see from the definition of $c_{\rho}(n)$, using orthogonality relations: 
\begin{align*} \sum_{\rho \in \operatorname{Pic}(\mathcal{O}_{p^{\alpha}})^{\vee}} c_{\rho}(n) 
&= \sum_{\rho \in \operatorname{Pic}(\mathcal{O}_{p^{\alpha}})^{\vee}} 
\sum_{A \in \operatorname{Pic}(\mathcal{O}_{p^{\alpha}})} r_A(n) \rho(A) 
= \sum_{A \in \operatorname{Pic}(\mathcal{O}_{p^{\alpha}})} r_A(n) 
\sum_{\rho \in \operatorname{Pic}(\mathcal{O}_{p^{\alpha}})^{\vee}} \rho(A) 
= \# \operatorname{Pic}( \mathcal{O}_{p^{\alpha}}) r_{\bf{1}}(n). \end{align*}
Here again, we use that $r_{\bf{1}}(n) = r(n)$ for any integer $n \geq 1$ prime to $p$ 
(see e.g.~\cite[Proposition 7.20]{Cox}).

\begin{proposition}[``Sali\'e"]\label{salie} 

Assume that $\beta \geq 4$, and without loss of generality that $\beta$ is even, say $\beta = 2 b$ for some integer $b \geq 2$. 
Then for any integer $c \geq 1$ prime to $p$, we have the formula
\begin{align*} \operatorname{Kl}_4(c, p^{\beta}) 
&=  p^{\frac{3 \beta}{2}} \sum_{w \operatorname{mod} p^{b} \atop w^4 \equiv c \operatorname{mod} p^{b}} 
e \left( \frac{3 w + c \overline{w}}{p^{\beta}}\right).\end{align*}
Here, the sum runs fourth roots of $c \operatorname{mod} p^b$. \end{proposition}

\begin{proof} The result, attributed to Sali\'e, is considered to be classical. However, the main written reference seems to be 
\cite[Theorem C.1, Lemma C.2]{BB}, where there is a minor error with the formulation of the final statement 
(which needs to be given in terms of liftings of the roots $r^{1/n}$). We therefore indicate a proof of the stated formula for the convenience of the reader. 
Let us lighten notation by writing $x = (x_1, x_2, x_3)$ to denote the $3$-tuple of classes $\operatorname{mod} p^{2 b}$. We then write 
$h = h_c$ denote the function defined on $x$ by $h(x) = x_1 + x_2 + x_3 + c \overline{x_1 x_2 x_2}$, 
and $\nabla h(x) = 1 - c \overline{x}^2 = (1 - c\overline{(x_1 x_2 x_3)x_1}, 1 - c\overline{(x_1 x_2 x_3)x_2}1 - c\overline{(x_1 x_2 x_3)x_2})$. 
It is easy to show (cf.~\cite[Lemme C4]{BB}) that we have the expansion $h(x) = h(y) + p^b \nabla h(y) \boldsymbol{\cdot} z$, 
where $\boldsymbol{\cdot}$ denotes the dot product. 
Substituting this expansion into the definition of the hyper-Kloosterman sum $\operatorname{Kl}_4(c, p^{\beta})$ then gives the relation 
\begin{align*} \operatorname{Kl}_4(c, p^{\beta}) = \sum_{x \operatorname{mod} p^{2b} \atop (x, p^{2b})=1} e \left( \frac{h(x)}{p^{2b}}\right) &=
\sum_{y \operatorname{mod} p^{b} \atop (y, p^{b})=1} \sum_{z \operatorname{mod} p^b} e \left(\frac{h(y) + p^b  \nabla h(y) \boldsymbol{\cdot} z}{p^{2b}} \right) \\
&= \sum_{y \operatorname{mod} p^{2b} \atop (y, p^{2b})=1} e \left(\frac{h(y) }{p^{2b}} \right) \sum_{z \operatorname{mod} p^b } 
e \left( \frac{ \nabla h(y) \boldsymbol{\cdot} z}{p^{b}}\right).\end{align*}
Since the inner sum runs over all ($3$-tuples of) classes $z \operatorname{mod} p^{b}$, 
we can use orthogonality of additive characters to evaluate the inner sum, and hence to obtain the relation
\begin{align*} \operatorname{Kl}_4(c, p^{2b}) 
&= p^{3b} \sum_{ {y \operatorname{mod} p^{2b} \atop (y, p^{2b})=1} \atop \nabla h(y) \equiv 0 \operatorname{mod} p^b} 
e \left(\frac{h(y) }{p^{2b}} \right).\end{align*}
Now, it is easy to see that the solutions to the congruence $\nabla h(y) \equiv 0 \operatorname{mod} p^b$ take the form
of invertible classes $y \operatorname{mod} p^b$ for which $y y_j \equiv c \operatorname{mod} p^b$ for each of $j = 1,2,3$. 
This proves the stated formula. \end{proof}

\begin{corollary} 

If $\beta = 2 b \geq 4$, then we have the following equivalent expressions for our twisted sums:
\begin{align*}\widetilde{D}_{1}(\alpha, \beta; p^{\frac{\beta}{8}}) 
= \frac{\omega(N) \tau(\omega)^4}{\vert D \vert^2 p^{\frac{\beta}{2}}} \frac{p}{\varphi(p)} 
\sum_{m \geq 1 \atop (m, p)=1} &\frac{\omega(m)}{m} \sum_{ n \geq 1 \atop (n, p)=1} \frac{r(n)\lambda(n)}{n^{\frac{1}{2}}} 
V_{1} \left(\frac{m^2 n }{N \vert D \vert p^{2\alpha} p^{\frac{17}{8} \beta} } \right) 
\sum_{w \operatorname{mod} p^{b} \atop w^8 \equiv \pm m^2 n \overline{N}^2 \overline{D}^2 \operatorname{mod} p^{b} } 
e \left( \frac{3 w + w^4 \overline{w}}{p^{\beta}} \right) \end{align*} 
and 
\begin{align*}\widetilde{D}_{1}(\rho, \beta; p^{\frac{\beta}{8}}) 
= \frac{\omega(N) \tau(\omega)^4}{\vert D \vert^2 p^{\frac{\beta}{2}}} \frac{p}{\varphi(p)}
 \sum_{m \geq 1 \atop (m, p)=1} 
&\frac{\omega(m)}{m} \sum_{ n \geq 1 \atop (n, p)=1} \frac{c_{\rho}(n)\lambda(n)}{n^{\frac{1}{2}}} 
V_{1} \left(\frac{m^2 n }{N \vert D \vert p^{2\alpha} p^{\frac{17}{8} \beta}} \right) 
\sum_{w \operatorname{mod} p^{b} \atop w^8 \equiv \pm m^2 n \overline{N}^2 \overline{D}^2 \operatorname{mod} p^{b} } 
e \left( \frac{3 w + w^4 \overline{w}}{p^{\beta}} \right).\end{align*} 

\end{corollary}

\begin{proof} The result follows by direct substitution, squaring the classes $w^4 \operatorname{mod} p^b$ to simplify the congruences. \end{proof}

Using this result, we can now express our twisted sums $\widetilde{D}_1(\alpha, \beta; p^{\frac{\beta}{8}})$ 
and $\widetilde{D}_1(\rho, \beta; p^{\frac{\beta}{8}})$ in terms of Weyl sums.
Here, to simplify expressions, we write $e(\pm x) := e(x) + e(-x) = \exp(2 \pi i  x) + \exp(- 2 \pi i x)$.

\begin{proposition}\label{DtildeWS} 

Suppose that $\beta \geq 4$, and without loss of generality that $\beta$ is even, say $\beta = 2 b$ with $\beta \geq 2$. 
Let $1 \leq s \leq \beta$ be any integer. Given a class $u \bmod p^s$, let $(\frac{u}{p})_8$ denote the octic residue symbol.
Given an integer $m \geq 1$, let $\mu_m$ denote the class $\mu_m \equiv m^2 \overline{N}^2 \overline{D}^2 \bmod p^s$.
Let $F_{u, m, s}(t)$ denote the polynomial 
\begin{align*} F_{u, m, s}(t) &= \frac{\xi_{u, m}}{ p^{\beta} } \left( 
\sum_{j = 0}^{\lfloor \beta / s \rfloor}  { { \frac{1}{8}}\choose{j}} \left[ 3 (\overline{u} \mu_m)^j    
+ \xi_{u,m}^3 \overline{\xi}_{u, m} { { \frac{1}{2}}\choose{j}} \right]p^{s j} t^j  \right). \end{align*} 
Here, 
\begin{align*} { \frac{1}{n}}\choose{j} &= \frac{\frac{1}{n} (\frac{1}{n} - 1) \cdots (\frac{1}{n} - j +1)}{j!}\end{align*} 
for each integer $n \geq 2$, $\xi_{u, m}$ is some fixed eight root of  $(m \overline{N} \overline{D}) u \overline{\mu}_m \bmod p^{\beta}$,
and $\lfloor \beta / s \rfloor$ denotes the nearest integer $\leq \beta/s$.
We have the following equivalent descriptions for the sums $\widetilde{D}_1(\alpha, \beta; p^{\frac{\beta}{8}})$ 
and $\widetilde{D}_1(\rho, \beta; p^{\frac{\beta}{8}})$: 
\begin{align*} &\widetilde{D}_1(\alpha, \beta; p^{\frac{\beta}{8}}) \\ 
&= \frac{\omega(N) \tau(\omega)^4}{\vert D \vert^2 p^{\frac{\beta}{2}}} \frac{p}{\varphi(p)} 
\sum_{1 \leq u \leq p^s \atop  (\frac{u}{p})_8=1} \sum_{m \geq 1} \frac{\omega(m)}{m} \sum_{t \geq 0} 
\frac{\lambda(u \overline{\mu}_m + p^s t) r(u \overline{\mu}_m + p^s t)}{ (u \overline{\mu}_m + p^s t)^{\frac{1}{2}}} 
V_1 \left( \frac{m^2 (u \overline{\mu}_m + p^s t)}{N \vert D \vert p^{2\alpha} p^{\frac{17}{8} \beta}}\right) 
e \left( \pm F_{u, m, s}(t) \right)  \end{align*} 
and 
\begin{align*} &\widetilde{D}_1(\rho, \beta; p^{\frac{\beta}{8}}) \\ 
&=\frac{\omega(N) \tau(\omega)^4}{\vert D \vert^2 p^{\frac{\beta}{2}}} \frac{p}{\varphi(p)} 
\sum_{1 \leq u \leq p^s \atop  (\frac{u}{p})_8=1} \sum_{m \geq 1} \frac{\omega(m)}{m} \sum_{t \geq 0} 
\frac{\lambda(u \overline{\mu}_m + p^s t) c_{\rho}(u \overline{\mu}_m + p^s t)}{ (u \overline{\mu}_m + p^s t)^{\frac{1}{2}}} 
V_1 \left( \frac{m^2 (u \overline{\mu}_m + p^s t)}{N \vert D \vert p^{2 \alpha} p^{\frac{17}{8} \beta}  }\right) e \left( \pm F_{u, m, s}(t) \right).
\end{align*} \end{proposition}

\begin{proof} 

Fix an integer $1 \leq s \leq \beta$. We then divide $mn$-sums of into congruence classes modulo $p^s$ to obtain 
\begin{align*} &\widetilde{D}_1(\alpha, \beta; p^{\frac{\beta}{8}}) \\
&= \frac{\omega(N) \tau(\omega)^4}{\vert D \vert^2 p^{\frac{\beta}{2}}} \frac{p}{\varphi(p)} \sum_{1 \leq u \leq p^s \atop  (\frac{u}{p})_8=1}
&\sum_{m, n \geq 1 \atop u \equiv m^2 n \overline{N}^2 \overline{D}^2 \bmod p^s} 
\frac{\omega(m) \lambda(n) r(n)}{ m n^{\frac{1}{2}}} V_1 \left( \frac{m^2 n}{N \vert D \vert p^{2 \alpha} p^{\frac{17}{8} \beta}}\right) 
\sum_{ w \bmod p^b \atop w^8 \equiv \pm u \bmod p^b } e\left( \frac{ 3 w + w^4 \overline{w}}{p^{\beta}}\right) \end{align*} 
and 
\begin{align*} &\widetilde{D}_1(\rho, \beta; p^{\frac{\beta}{8}}) \\
&= \frac{\omega(N) \tau(\omega)^4}{\vert D \vert^2 p^{\frac{\beta}{2}}} \frac{p}{\varphi(p)}
\sum_{1 \leq u \leq p^s \atop  (\frac{u}{p})_8=1}&\sum_{m, n \geq 1 \atop u \equiv m^2 n \overline{N}^2 \overline{D}^2 \bmod p^s} 
\frac{\omega(m) \lambda(n) c_{\rho}(n)}{ m n^{\frac{1}{2}}} V_1 \left( \frac{m^2 n}{N \vert D \vert p^{2\alpha} p^{\frac{17}{8} \beta}}\right) 
\sum_{ w \bmod p^b \atop w^8 \equiv \pm u \bmod p^b } e\left( \frac{ 3 w + w^4 \overline{w}}{p^{\beta}}\right), \end{align*}
where the condition $(\frac{u}{p})_8 = 1$ in the $u$-sum comes from the $w$-sum via Hensel's lemma. Note that here,
each of the $w$-sums consists one a single pair of eighth roots $w \operatorname{mod} p^b$. 

To give a more explicit description of the eighth roots $w^8 \equiv \pm u \bmod p^b$ appearing in these expressions, 
fix a class $u \bmod p^{s}$ for which $(\frac{u}{p})_8 = 1$, and consider the corresponding inner $mn$-sum. 
Let $\mu_m \equiv m^2 \overline{N}^2 \overline{D}^2 \bmod p^s$. Hence $n \mu_m \equiv m^2 \overline{N}^2 \overline{D}^2 n \bmod p^{s}$, 
from which it follows that $n \mu_m \equiv u \bmod p^{s}$. Hence, we can expand each integer $n \geq 1$ in the second  
sum in terms of the congruence condition $u \equiv n \mu_m \bmod p^s$ as $n = u \overline{\mu}_m + p^s t$, 
with $t \geq 0$ varying over positive integers. This in turn gives us the expansion
\begin{align*} w^8 &\equiv \pm m^2 \overline{N}^2 \overline{D}^2(u \overline{\mu}_m + p^s t) 
\equiv \pm (m \overline{N}\overline{D})^2 (u \overline{\mu}_m + p^s t) 
\equiv \pm (m \overline{N}\overline{D})^2 u \overline{\mu}_m(1 + \kappa p^s t) \bmod p^{\beta}, \end{align*} 
where $\kappa = \overline{u} \mu_m$ denotes the multiplicative inverse of $u \overline{\mu}_m \bmod p^{\beta}$. 
Now, Hensel's lemma ensures that we can find an integer $\xi = \xi_{u, m}$ such that 
$\xi^8 \equiv (m \overline{N} \overline{D}) u \overline{\mu}_m \bmod p^{\beta}$. 
(In fact, there exist $O_p(1)$ many such roots, and we choose one implicitly). 
Hence, we can express the roots appearing in the corresponding $w$-sum as 
\begin{align}\label{eighth}w^8 \equiv \pm \xi^8 (1 + \kappa p^s t) \bmod p^{\beta}, \end{align}
so that $w \equiv \xi_u(1 + \kappa p^s t)^{\frac{1}{8}}$ and $w^4 \equiv \xi_u^4 (1 + \kappa p^s t)^{\frac{1}{2}}$.
Now, to give an even more explicit description of these classes $w$ and $w^4$, 
we can use the classical fact that for any $x \in p{\bf{Z}}_p$, the power series 
\begin{align*} \sum_{j \geq 0}  { { \frac{1}{n}}\choose{j}} x^j  ~~~~~ \in {\bf{Z}}_p [[x]]\end{align*} 
converges in the $p$-adic norm to the $n$-th root $(1 + x)^{\frac{1}{n}}$ for any integer $n \geq 2$ (see e.g.~\cite[p. 173]{Rob}) to obtain
\begin{align*} (1 + \kappa p^s t)^{\frac{1}{n}} &= \sum_{j \geq 0}  { { \frac{1}{n}}\choose{j}} \kappa^j p^{s j} t^j, \end{align*} 
so that 
\begin{align*} w \equiv \xi_u \sum_{j \geq 0}  { { \frac{1}{8}}\choose{j}} \kappa^j p^{s j} t^j  \bmod  p^{\beta} \end{align*}
and 
\begin{align*} w^4 \equiv \xi_u^4 \sum_{j \geq 0}  { { \frac{1}{2}}\choose{j}} \kappa^j p^{s j} t^j \bmod p^{\beta}. \end{align*}
Hence, we derive the more explicit presentation 
\begin{align*} \frac{3 w + w^4 \overline{w}}{p^{\beta}} &= 
\frac{ 3 \xi (1 + \kappa p^s t)^{\frac{1}{8}} + \xi^4 (1 + \overline{\kappa} p^s t)^{\frac{1}{2}} \overline{\xi} (1 + \kappa p^s t)^{\frac{1}{8}}}{p^{\beta}} \\
&= \frac{\xi}{p^{\beta}} \left(  3 (1 + \kappa p^s t)^{\frac{1}{8}} + \xi^3 \overline{\xi}(1 + \overline{\kappa} p^s t)^{\frac{1}{2}}  (1 + \kappa p^s t)^{\frac{1}{8}} \right) \\
&= \frac{\xi_u}{ p^{\beta} } \left( 3 \sum_{j = 0}^{\lfloor \beta / s \rfloor}  { { \frac{1}{8}}\choose{j}} \kappa^j p^{s j} t^j    
 + \xi^3 \overline{\xi} \sum_{j = 0}^{\lfloor \beta / s \rfloor}  { { \frac{1}{8}}\choose{j}} { { \frac{1}{2}}\choose{j}} p^{ 2s j} t^{2j}   \right), \end{align*} 
which simplifies to give the stated expressions for the exponential sums described at the start of the proof. \end{proof}

Using this description of the twisted sums, we may now use Weyl differencing to derive bounds as follows. 

\begin{theorem}\label{D1tilde} 

Assume that $\beta \geq 4$. We have for any choice of $\varepsilon >0$ (in either case) the upper bounds 
\begin{align*} \widetilde{D}_1(\alpha, \beta; p^{\frac{\beta}{8}}) 
&= O_{p, f, D, \alpha, \varepsilon} \left( p^{-\beta (\frac{1}{4} - \varepsilon)} \right) \end{align*}
and 
\begin{align*} \widetilde{D}_1(\rho, \beta; p^{\frac{\beta}{8}}) 
&= O_{p, f, D, \rho, \varepsilon} \left( p^{-\beta (\frac{1}{4} - \varepsilon)} \right). \end{align*}
\end{theorem}

\begin{proof} Taking $s = b = \frac{\beta}{2}$ in Proposition \ref{DtildeWS} above, we have the identities
\begin{align}\label{exq1} \widetilde{D}_1(\alpha, \beta; p^{\frac{\beta}{8}}) = 
&\frac{\omega(N) \tau(\omega)^4}{\vert D \vert^2 p^{b}} \frac{p}{\varphi(p)}
\sum_{1 \leq u \leq p^b \atop  (\frac{u}{p})_8=1} \sum_{m \geq 1} \frac{\omega(m)}{m} \sum_{t \geq 0} 
\frac{\lambda(u \overline{\mu}_m + p^b t) r(u \overline{\mu}_m + p^b t)}{ (u \overline{\mu}_m + p^b t)^{\frac{1}{2}}} 
V_1 \left( \frac{m^2 (u \overline{\mu}_m + p^b t)}{N \vert D \vert p^{2 \alpha} p^{\frac{17}{8} \beta}}\right) e \left( \pm F_{u, m, b}(t) \right)  \end{align} 
and 
\begin{align}\label{exq2} \widetilde{D}_1(\rho, \beta; p^{\frac{\beta}{8}}) =  
&\frac{\omega(N) \tau(\omega)^4}{\vert D \vert^2 p^{b}} \frac{p}{\varphi(p)}
\sum_{1 \leq u \leq p^b \atop  (\frac{u}{p})_8=1} \sum_{m \geq 1} \frac{\omega(m)}{m} \sum_{t \geq 0} 
\frac{\lambda(u \overline{\mu}_m + p^b t) c_{\rho}(u \overline{\mu}_m + p^b t)}{ (u \overline{\mu}_m + p^b t)^{\frac{1}{2}}} 
V_1 \left( \frac{m^2 (u \overline{\mu}_m + p^b t)}{N \vert D \vert p^{2 \alpha} p^{\frac{17}{8} \beta}} \right) e \left( \pm F_{u, m, b}(t) \right), \end{align}
where each of the quadratic polynomials 
\begin{align*} F_{u, m, b}(t) &= \frac{\xi_{u, m}}{ p^{\beta} } \left( \sum_{j = 0}^{2}  { { \frac{1}{8}}\choose{j}} \left[ 3 (\overline{u} \mu_m)^j   
+ \xi_{u,m}^3 \overline{\xi}_{u, m} { { \frac{1}{2}}\choose{j}} \right]p^{s j} t^j  \right). \end{align*}
has leading coefficient given by 
\begin{align*} \theta = \theta_{u,m} &= \xi_{u, m} { { \frac{1}{8}}\choose{2}} \left[ 3 (\overline{u} \mu_m)^2   
+ \xi_{u,m}^3 \overline{\xi}_{u, m} { { \frac{1}{2}}\choose{2}} \right] \\ 
&= -\frac{7}{128} \xi_{u, m} \left[ 3 (\overline{u} \mu_m)^2  + \xi_{u,m}^3 \overline{\xi}_{u, m} \frac{1}{8} \right] = 
-\frac{7}{1024} \left[ 3 \cdot 8 (\overline{u} \mu_m)^2  + \xi_{u,m}^3 \overline{\xi}_{u, m} \right] \in {\bf{Q}} - \lbrace 0 \rbrace. \end{align*}

Let $T \geq 1$ be any integer, and $F(t)$ any quadratic polynomial of the form 
$F(t) = \theta t^2 + c_1 t + c_0$ with $\theta \in {\bf{R}}$. Dirichlet's approximation theorem implies
that there exist (infinitely many) pairs of coprime integers $a$ and $q$ such that 
\begin{align}\label{DAT} \left\vert 2\theta - \frac{a}{q} \right\vert &\leq \frac{1}{2 T q} \end{align}
with $1 \leq q \leq 2T$. If $\theta$ is sufficiently well-approximable to a rational number in this sense, 
then a classical argument due to Weyl can be used to estimate the exponential sum
\begin{align*} S_F(T) &:= \sum_{1 \leq t \leq T} e(F(t)) \end{align*} as 
\begin{align} \vert S_F(T) \vert &\leq 2 T q^{-\frac{1}{2}} + q^{\frac{1}{2}} \log q \end{align}
for any choice of $1 \leq q \leq 2T$ for which the inequality $(\ref{DAT})$ holds. 
We refer to \cite[Proposition 8.1]{IK} for more details of this argument. Here,
since our leading coefficient $\theta = \theta_{u, m}$ is rational, the inequality $(\ref{DAT})$ 
is satisfied trivially for any choice of $1 \leq q \leq 2T$ (e.g.~by taking $a = 2 \theta q$). 
Taking $q = 2T$ then gives us 
\begin{align}\label{Weylapp} \left\vert  S_{F_{u, m, b}}(T) \right\vert 
&\ll_{\varepsilon} T^{\frac{1}{2} + \varepsilon} \end{align}
for any choice of $\varepsilon >0$, uniformly in the choice of class $u \operatorname{mod} p^b$ 
or integer $m \geq 1$ in the sums $(\ref{exq1})$ and $(\ref{exq2})$.

Let us now return to the sums $(\ref{exq1})$ and $(\ref{exq2})$, starting with the expression $(\ref{exq1})$ for $\widetilde{D}_1(\alpha, \beta; p^{\frac{\beta}{8}})$.
Fixing a residue class $u \operatorname{mod} p^b$ (of which there are $O_p(p^b)$ many), it will do to bound the corresponding inner $t$-sum. 
Moreover, we argue using partial summation that it will do to consider the contribution from $m=1$, and by symmetry that it will do to consider the 
contribution of the exponential terms $e(F_{u, m, b}(t))$ (say) in the expansions $e(\pm F_{u, m, b}(t)) = e(F_{u, m, b}(t)) + e(- F_{u, m, b}(t))$. 
Hence, it remains to bound the contribution of 
\begin{align}\label{exq1r} \sum_{t \geq 0} \frac{\lambda(u \overline{\mu}_1 + p^b t) r(u \overline{\mu}_1 + p^b t)}{ (u \overline{\mu}_1 + p^b t)^{\frac{1}{2}}} 
V_1 \left( \frac{ u \overline{\mu}_1 + p^b t }{N \vert D \vert p^{2 \alpha } p^{\frac{17}{8} \beta}}\right) e \left(F_{u, 1, b}(t) \right). \end{align}
Using the rapid decay of the cutoff function $V_1$, we argue that it will suffice to consider the truncated sum of length 
$T := N \vert D \vert p^{2 \alpha + \frac{17}{8} \beta - b}$ corresponding to the region of moderate decay for $V_1$: 
\begin{align}\label{exq1rT} \sum_{0 \leq t \leq T} \frac{\lambda(u \overline{\mu}_1 + p^b t) r(u \overline{\mu}_1 + p^b t)}{ (u \overline{\mu}_1 + p^b t)^{\frac{1}{2}}} 
V_1 \left( \frac{u \overline{\mu}_1 + p^b t }{N \vert D \vert p^{2 \alpha} p^{\frac{17}{8} \beta}}\right) e \left(F_{u, 1, b}(t) \right), \end{align}
i.e.~as $t \leq T$ implies that $u \overline{\mu}_1 + p^b t \leq N \vert D \vert p^{2 \alpha + \frac{17}{8} \beta}$ 
(with $V_1(u \overline{\mu}_1 + p^b t) = O_C( (u \overline{\mu}_1 + p^b t)^{-C})$ when $t > T$).
To estimate this latter sum $(\ref{exq1rT})$ via the bound $(\ref{Weylapp})$, let us first apply a discrete version of partial summation
(cf.~Lemma \ref{ps} above). Hence, given an integer $t \geq 0$, let us write 
\begin{align*} h(t) &= \frac{\lambda(u \overline{\mu}_1 + p^b t) r(u \overline{\mu}_1 + p^b t)}{ (u \overline{\mu}_1 + p^b t)^{\frac{1}{2}}} 
V_1 \left( \frac{ u \overline{\mu}_1 + p^b t }{N \vert D \vert p^{2 \alpha} p^{\frac{17}{8} \beta} }\right).\end{align*}
It is then easy to see that the sum $(\ref{exq1rT})$ is given by the difference 
\begin{align*}S_{F_{u, 1, b}}(T) h(T) - \sum_{1 \leq t \leq T-1} S_{F_{u, 1, b}}(t) \left[ h(t) - h(t-1) \right], \end{align*}
and that it will suffice to bound the modulus of the leading term $S_{F_{u, 1, b}}(T) h(T)$. 
Using the Deligne/classical bounds $\lambda(n), r(n) \ll_{\varepsilon} n^{\varepsilon}$ together with the decay bound
$V_1(y) = O(y^{\frac{1}{2}})$ for $1 \leq y \leq 1$ and the Weyl bound $(\ref{Weylapp})$, we then derive the bound
\begin{align*} S_{F_{u, 1, b}}(T) h(T) 
&\ll_{\varepsilon} \left\vert S_{F_{u, 1, b}}(T) \right\vert \cdot \frac{(p^b T)^{\varepsilon}}{(p^b T)^{ \frac{1}{2} }} 
\frac{ (p^b T)^{\frac{1}{2}}}{ (N \vert D \vert p^{2 \alpha} p^{\frac{17}{8} \beta})^{\frac{1}{2}} } \\
&\ll_{\epsilon, \varepsilon'}  (N \vert D \vert p^{2\alpha} p^{\frac{17}{8} \beta})^{-\frac{1}{2}} (p^b T)^{\varepsilon} T^{\frac{1}{2} + \varepsilon'} 
=(N \vert D \vert p^{2 \alpha} p^{\frac{17}{8} \beta})^{-\frac{1}{2} + \varepsilon + \varepsilon'} T^{\frac{1}{2} + \varepsilon'}  \\
&\ll_{\varepsilon} (N \vert D \vert p^{2\alpha} p^{\frac{17}{8} \beta})^{\varepsilon}  (p^{- b})^{\frac{1}{2} + \varepsilon},  \end{align*}
which (by our various reductions above) also bounds the modulus of the sums $(\ref{exq1rT})$ and $(\ref{exq1r})$.
Taking the sum over residues $u \bmod p^{b}$ (of which there are $O_p(p^b)$ many), 
and taking into account the scaling factor $p^{-b}$ at the front of the sum, 
we then derive the claimed upped bound $\widetilde{D}_1(\alpha, \beta; p^{\frac{\beta}{8}}) 
= O_{p, f, D, \alpha, \varepsilon} \left(  p^{-\beta (\frac{1}{4} - \varepsilon)} \right) $.
A completely analogous argument (replacing $r$ by $c_{\rho}$)\footnote{Using that $c_{\rho}$ is the Fourier coefficient of the theta series $\theta(\rho)$}
works to bound the sum $\widetilde{D}_1(\rho, \beta; p^{\frac{\beta}{8}})$ as stated. \end{proof}

\begin{remark} It should be possible to derive sharper bounds for these sums (in the style of Khan \cite{Kh} say) 
by taking Weyl polynomials of higher degree. However, we have not pursued the idea in this work. \end{remark}

\subsection{The leading sums} 

Fix an integer $\beta \geq 2$. Fix an integer $\alpha \geq 1$, as well as a primitive ring class character 
$\rho$ of conductor $p^{\alpha}$. Let us now consider the leading sums 
\begin{align*} D_{1}(\alpha, \beta; p^{\frac{\beta}{8}}) = \sum_{m \geq 1} \frac{\omega(m)}{m} 
&\sum_{ {n \geq 1 \atop (n, p) =1} \atop m^2 n \equiv 1 \bmod p^{\beta}} \frac{r(n)\lambda(n)}{n^{\frac{1}{2}}} 
V_{1}\left(\frac{m^2 n }{ N \vert D \vert p^{2\alpha} p^{\frac{15}{8} \beta} } \right) \\ 
& - \frac{1}{\varphi(p)} \sum_{m \geq 1} \frac{\omega(m)}{m} 
\sum_{ {n \geq 1 \atop m^2 n \equiv 1 \bmod p^{\beta-1}} \atop m^2 n \not \equiv 1 \bmod p^{\beta}} \frac{r(n)\lambda(n)}{n^{\frac{1}{2}}} 
V_{1}\left(\frac{m^2 n }{ N \vert D \vert p^{2 \alpha} p^{\frac{15}{8} \beta} } \right) \end{align*} 
and 
\begin{align*} D_1(\rho, \beta; p^{\frac{\beta}{8}}) = \sum_{m \geq 1} \frac{\omega(m)}{m} 
&\sum_{ {n \geq 1 \atop (n, p) =1} \atop m^2 n \equiv 1 \bmod p^{\beta}} \frac{c_{\rho}(n)\lambda(n)}{n^{\frac{1}{2}}} 
V_{1}\left(\frac{m^2 n }{ N \vert D \vert p^{2 \alpha} p^{\frac{15}{8} \beta} } \right) \\ 
& - \frac{1}{\varphi(p)} \sum_{m \geq 1} \frac{\omega(m)}{m} 
\sum_{ {n \geq 1 \atop m^2 n \equiv 1 \bmod p^{\beta-1}} \atop m^2 n \not \equiv 1 \bmod p^{\beta}} \frac{c_{\rho}(n)\lambda(n)}{n^{\frac{1}{2}}} 
V_{1}\left(\frac{m^2 n }{ N \vert D \vert p^{2 \alpha}p^{\frac{15}{8} \beta} } \right).\end{align*}
We now estimate these sums. Let us lighten notation by writing 
\begin{align*} \mathfrak{D} = q_1(0, 1) = 
&= \begin{cases} \frac{\vert D \vert}{4} &\text{ if $D \equiv 0 \bmod 4$} \\ \frac{1+\vert D \vert}{4} &\text{ if $D \equiv 1 \bmod 4$}. \end{cases} \end{align*}

\begin{proposition}\label{D1} 

We have for the setup described above the estimate 
\begin{align*} D_1(\alpha, \beta; p^{\frac{\beta}{8}}) &= \frac{2}{w}\left( 1 + \frac{\lambda(\mathfrak{D})}{\mathfrak{D}^{\frac{1}{2}}} \right) 
\cdot \frac{2}{\varphi^\star(p^{\beta})} \sum_{\chi \bmod p^{\beta} \atop \chi(-1)=1, \operatorname{primitive}} 
L(1, \omega \chi^2) \cdot \frac{L(1, \operatorname{Sym^2} f \otimes \chi)}{L(2, \chi^2)} \\
&+ O_{\varepsilon} \left( (N \vert D \vert p^{2 \alpha})^{\frac{1}{2} + \varepsilon} (p^{\beta})^{- \frac{1}{16} + 30 \varepsilon} \right). \end{align*} 

\end{proposition}

\begin{proof} 

Recall that after unwinding, the sum $D_1(\alpha, \beta; p^{\frac{\beta}{8}})$ can be expressed equivalently as 
\begin{align*} D_1(\alpha, \beta; p^{\frac{\beta}{8}}) &= \frac{2}{\varphi^{\star}(p^{\beta})} 
\sum_{\chi \bmod p^{\beta} \atop \chi(-1) = 1, \operatorname{primitive}} 
\sum_{m \geq 1 \atop (m, p)=1} \frac{\omega \chi^2 (m)}{m} \sum_{n \geq 1 \atop (n, p)=1} 
\frac{\lambda(n) \chi(n) r(n)}{n^{\frac{1}{2}}} V_1\left( \frac{m^2 n}{N \vert D \vert p^{2 \alpha} p^{\frac{15}{8} \beta}} \right), \end{align*}
which after expanding in terms of the parametrization $(\ref{count})$ of the counting function $r(n)$ gives the expression 
\begin{align}\label{expression} D_1(\alpha, \beta; p^{\frac{\beta}{8}}) &= \frac{1}{w} \cdot \frac{2}{\varphi^{\star}(p^{\beta})} 
\sum_{\chi \bmod p^{\beta} \atop \chi(-1)= 1\operatorname{primitive}} 
\sum_{m \geq 1} \frac{\omega \chi^2 (m)}{m} \sum_{a, b \in {\bf{Z}} \atop q_1(a, b) \neq 0} 
\frac{\lambda(q_1(a, b)) \chi(q_1(a, b)) }{q_1(a, b)^{\frac{1}{2}}} 
V_1\left( \frac{m^2 q_1(a, b)}{N \vert D \vert p^{2 \alpha} p^{\frac{15}{8} \beta}} \right).\end{align}

Let us first consider the contributions from the $b=0$ terms in $(\ref{expression})$. Opening up contours, we see that
\begin{align*} &D_1(\alpha, \beta; p^{\frac{\beta}{8}})\vert_{b=0} \\ &= \frac{1}{w} \cdot \frac{2}{\varphi^{\star}(p^{\beta})} 
\sum_{\chi \bmod p^{\beta} \atop \chi(-1)= 1\operatorname{primitive}} \sum_{m \geq 1} \frac{\omega \chi^2 (m)}{m} \sum_{a \neq 0 \in {\bf{Z}}} 
\frac{\lambda(q_1(a, 0)) \chi(q_1(a, 0)) }{q_1(a, 0)^{\frac{1}{2}}} V_1\left( \frac{m^2 q_1(a, 0)}{N \vert D \vert p^{2 \alpha} p^{\frac{15}{8} \beta}} \right) \\ 
&= \frac{2}{w} \cdot \frac{2}{\varphi^{\star}(p^{\beta})} \sum_{\chi \bmod p^{\beta} \atop \chi(-1)= 1\operatorname{primitive}} 
\int_{\Re(s)=2} \frac{L_{\infty}(s + 1/2)}{L_{\infty}(1/2)} \frac{g^*(s)}{s} L(2s, \omega \chi^2) \frac{L(s + 1, \operatorname{Sym}^2 f \otimes \chi)}{L(2(s + 1), \chi^2)} 
(N \vert D \vert p^{2 \alpha} p^{\frac{15}{8} \beta})^s \frac{ds}{s}. \end{align*}
Shifting the contour in the latter integral to to $\Re(s)=-2$, we cross a simple pole at $s=0$ of residue
\begin{align*} \frac{2}{w} \cdot \frac{2}{\varphi^{\star}(p^{\beta})} \sum_{\chi \bmod p^{\beta} \atop \chi(-1)= 1\operatorname{primitive}}
L(1, \omega \chi^2) \frac{L(1, \operatorname{Sym}^2 f \otimes \chi)}{L(2, \chi^2)}. \end{align*}
Shifting the remaining contour leftward, we then see that the remaining contribution is bounded above by 
$O_C \left( (N \vert D \vert p^{2 \alpha} p^{\frac{15}{8} \beta})^{-C} \right)$ for any choice of constant $C >0$.
A completely analogous calculation for the $a = 0$ contributions gives the estimate 
\begin{align*} &D_1(\alpha, \beta; p^{\frac{\beta}{8}})\vert_{a=0} \\ 
&= \frac{1}{w} \cdot \frac{2}{\varphi^{\star}(p^{\beta})} \sum_{\chi \bmod p^{\beta} \atop \chi(-1)= 1\operatorname{primitive}} 
\sum_{m \geq 1} \frac{\omega \chi^2 (m)}{m} \sum_{b \neq 0 \in {\bf{Z}}} 
\frac{\lambda(q_1(0, b)) \chi(q_1(0, b)) }{q_1(0, b)^{\frac{1}{2}}} V_1\left( \frac{m^2 q_1(0, b)}{N \vert D \vert p^{2 \alpha} p^{\frac{15}{8} \beta}} \right) \\ 
&= \frac{2}{w} \left( \frac{\lambda(\mathfrak{D})}{\mathfrak{D}^{\frac{1}{2}}} \right) 
\cdot \frac{2}{\varphi^{\star}(p^{\beta})} \sum_{\chi \bmod p^{\beta} \atop \chi(-1)= 1\operatorname{primitive}} 
\int_{\Re(s)=2} \frac{L_{\infty}(s + 1/2)}{L_{\infty}(1/2)} \frac{g^*(s)}{s} L(2s, \omega \chi^2) 
\frac{L(s + 1, \operatorname{Sym}^2 f \otimes \chi)}{L(2(s + 1), \chi^2)} (N \vert D \vert p^{2 \alpha} p^{\frac{15}{8} \beta})^s \frac{ds}{s} \\ 
&= \frac{2}{w} \left( \frac{\lambda(\mathfrak{D})}{\mathfrak{D}^{\frac{1}{2}}} \right)
\cdot \frac{2}{\varphi^{\star}(p^{\beta})} \sum_{\chi \bmod p^{\beta} \atop \chi(-1)= 1\operatorname{primitive}} 
L(1, \omega \chi^2) \frac{L(1, \operatorname{Sym}^2 f \otimes \chi)}{L(2, \chi^2)} 
+ O_C \left( (N \vert D \vert p^{2 \alpha} p^{\frac{15}{8} \beta})^{-C} \right) \end{align*}
for any choice of constant $C>0$. Putting these contributions together, we then derive the residual estimate 
\begin{align*} &D_1(\alpha, \beta; p^{\frac{\beta}{8}})\vert_{b=0} + D_1(\alpha, \beta; p^{\frac{\beta}{8}})\vert_{a=0} \\ 
&= \frac{2}{w} \left( 1 + \frac{\lambda(\mathfrak{D})}{\mathfrak{D}^{\frac{1}{2}}} \right)
\cdot \frac{2}{\varphi^{\star}(p^{\beta})} \sum_{\chi \bmod p^{\beta} \atop \chi(-1)= 1\operatorname{primitive}} 
L(1, \omega \chi^2) \frac{L(1, \operatorname{Sym}^2 f \otimes \chi)}{L(2, \chi^2)} 
+ O_C \left( (N \vert D \vert p^{2 \alpha} p^{\frac{15}{8} \beta})^{-C} \right) \end{align*}
for any choice of constant $C>0$. Let us now estimate the remaining contribution
\begin{align*} D_1^{\star}(\alpha, \beta; p^{\frac{\beta}{8}}) &= \frac{1}{w} \cdot \frac{2}{\varphi^{\star}(p^{\beta})} 
\sum_{\chi \bmod p^{\beta} \atop \chi(-1)= 1\operatorname{primitive}} 
\sum_{m \geq 1} \frac{\omega \chi^2 (m)}{m} \sum_{a, b \in {\bf{Z}} \atop ab \neq 0} 
\frac{\lambda(q_1(a, b)) \chi(q_1(a, b)) }{q_1(a, b)^{\frac{1}{2}}} 
V_1\left( \frac{m^2 q_1(a, b)}{N \vert D \vert p^{2 \alpha} p^{\frac{15}{8} \beta}} \right),\end{align*}
which after evaluating via the orthogonality relation $(\ref{QO})$ is the same as 
\begin{align*} D^{\star}_1(\alpha, \beta; p^{\frac{\beta}{8}}) &= \frac{1}{w} \sum_{m \geq 1} \frac{\omega(m)}{m}
\sum_{ a, b \in {\bf{Z}}, ab \neq 0  \atop m^2 q_1(a, b) \equiv \pm 1 \bmod p^{\beta}   } \frac{ \lambda(q_1(a,b))}{q_1(a, b)^{\frac{1}{2}}} 
V_1 \left( \frac{m^2 q_1(a, b)}{N \vert D \vert p^{2 \alpha} p^{ \frac{15}{8} \beta }} \right) \\ 
&- \frac{1}{\varphi(p)} \frac{1}{w} \sum_{m \geq 1} \frac{\omega(m)}{m}
\sum_{ { a, b \in {\bf{Z}}, ab \neq 0 \atop m^2 q_1(a, b) \equiv \pm 1 \bmod p^{\beta-1}}Ê\atop m^2 q_1(a, b) \not\equiv \pm 1 \bmod p^{\beta} }
\frac{\lambda(q_1(a, b))}{q_1(a, b)^{\frac{1}{2}}} V_1 \left( \frac{m^2 q_1(a, b)}{N \vert D \vert p^{2 \alpha} p^{ \frac{15}{8} \beta }} \right). \end{align*}
Expanding out in terms of congruences and writing
\begin{align*} r^{\star}(n) &= \frac{1}{w} \cdot \# \left\lbrace (a, b) \in {\bf{Z}}^2, ab \neq 0, q_1(a, b) = n  \right\rbrace \end{align*}
to denote the corresponding counting function, we see that this is the same as 
\begin{align*} D^{\star}_1(\alpha, \beta; p^{\frac{\beta}{8}}) &= \frac{1}{w} \sum_{m \geq 1} \frac{\omega(m)}{m} \sum_{t \geq 1} 
\frac{\lambda( \overline{m}^2 + t p^{\beta}) r^{\star}(\overline{m}^2 + tp^{\beta})}{(\overline{m}^2 + t p^{\beta})} 
V_1 \left( \frac{m^2 (\overline{m}^2 + tp^{\beta})}{N \vert D \vert p^{2 \alpha} p^{ \frac{15}{8} \beta}} \right) \\
&- \frac{1}{\varphi(p)} \frac{1}{w} \sum_{m \geq 1} \frac{\omega(m)}{m} 
\sum_{l=1}^{p-1} \frac{\lambda(\overline{m}^2 + l p^{\beta-1}) r^{\star}( \overline{m}^2 + lp^{\beta-1})}{(\overline{m}^2 + l p^{\beta-1})}
V_1 \left( \frac{m^2 (\overline{m}^2 + lp^{\beta-1})}{N \vert D \vert p^{2 \alpha} p^{\frac{15}{8} \beta }} \right). \end{align*}
Let us first estimate the tail sum in this expression. For each $m \geq 1$, we use Deligne's bound $\lambda(n) \ll_{\varepsilon} n^{\varepsilon}$
and the classical estimate $r^{\star}(n) \ll r(n) \ll_{\varepsilon} n^{\varepsilon}$ to see that the corresponding contribution is 
\begin{align*} &\frac{1}{\varphi(p)} \frac{\omega(m)}{m} \sum_{l=1}^{p-1} 
\frac{\lambda(\overline{m}^2 + l p^{\beta-1}) r^{\star}(\overline{m}^2 + l p^{\beta-1})}{(\overline{m}^2 + l p^{\beta-1})^{\frac{1}{2}}}
V_1 \left( \frac{ m^2 (\overline{m}^2 + l p^{\beta-1})}{N \vert D \vert p^{2 \alpha} p^{\frac{15}{8} \beta}} \right) \\
&\ll_{p, \varepsilon} \frac{1}{m} \sum_{l=1}^{p-1} \frac{(\overline{m}^2 + l p^{\beta-1})^{\varepsilon}}{(\overline{m}^2 + l p^{\beta-1})^{\frac{1}{2}}} 
\frac{m (\overline{m}^2 + l p^{\beta-1})^{\frac{1}{2}}}{(N \vert D \vert p^{2 \alpha} p^{\frac{15}{8} \beta})^{\frac{1}{2}}} \\ 
&\ll_{p, \varepsilon} (N \vert D \vert p^{2 \alpha} p^{\frac{15}{8} \beta})^{- \frac{1}{2}} \varphi(p) (p^{\beta})^{\varepsilon}
= O_{p, f, D, \alpha, \varepsilon} \left( (p^{\beta})^{- \frac{15}{16} + \varepsilon} \right). \end{align*}
Since this estimate is uniform in the choice of $m$, we deduce that the contribution of the tail sum is bounded above by 
$O_{p, f, D, \alpha, \varepsilon} \left( (p^{\beta})^{- \frac{15}{16} + \varepsilon} \right)$. To estimate the main sum in this latter 
expression $D^{\star}_1(\alpha, \beta; p^{\frac{\beta}{8}})$, we consider for each integer $m \geq 1$ the corresponding contribution 
\begin{align*} \frac{\omega(m)}{m} 
\sum_{t \geq 0} \frac{\lambda(\overline{m}^2 + t p^{\beta}) r^{\star}(\overline{m}^2 + t p^{\beta})}{(\overline{m}^2 + tp^{\beta})^{\frac{1}{2}} }
V_1 \left( \frac{m^2 (\overline{m}^2 + t p^{\beta})}{N \vert D \vert p^{2 \alpha} p^{\frac{15}{8} \beta}} \right). \end{align*}
Now, observe that by the rapid decay of $V_1(y)$ for $y \rightarrow \infty$, it will suffice to estimate the contribution 
of the truncated sum determined by the condition $t p^{\beta} \leq N \vert D \vert p^{2 \alpha} p^{\frac{15}{8} \beta}$, 
equivalently $t \leq N \vert D \vert p^{2 \alpha} p^{\frac{7}{8} \beta}$. Using the moderate decay of 
$V_1(y)$ for $y \rightarrow 0$, it is easy to see via a similar argument that 
\begin{align*} &\frac{\omega(m)}{m} \sum_{t \leq N \vert D \vert p^{2 \alpha} p^{\frac{7}{8} \beta}} 
\frac{\lambda(\overline{m}^2 + t p^{\beta}) r^{\star}(\overline{m}^2 + t p^{\beta})}{(\overline{m}^2 + tp^{\beta})^{\frac{1}{2}} }
V_1 \left( \frac{m^2 (\overline{m}^2 + t p^{\beta})}{N \vert D \vert p^{2 \alpha} p^{\frac{15}{8} \beta}} \right) \\
&\ll_{\varepsilon} \frac{1}{m} \sum_{t \leq N \vert D \vert p^{2 \alpha} p^{\frac{7}{8} \beta}} 
\frac{(\overline{m}^2 + t p^{\beta})^{\varepsilon}}{(\overline{m}^2 + tp^{\beta})^{\frac{1}{2}}} 
\frac{m (\overline{m}^2 + tp^{\beta})^{\frac{1}{2}}}{(N \vert D \vert p^{2 \alpha} p^{\frac{15}{8} \beta})^{\frac{1}{2}}} \\ 
&\ll_{\varepsilon} (N \vert D \vert p^{2 \alpha} p^{\frac{15}{8} \beta})^{- \frac{1}{2}} \sum_{t \leq N \vert D \vert p^{2 \alpha} p^{\frac{7}{8} \beta}} 
(\overline{m}^2 + tp^{\beta})^{\varepsilon} \\ &\ll_{\varepsilon} (N \vert D \vert p^{2 \alpha})^{\frac{1}{2} + \varepsilon} 
p^{- \frac{15}{16} \beta} p^{\frac{7}{8} \beta} p^{\frac{15}{8} \varepsilon} 
= (N \vert D \vert p^{2 \alpha})^{\frac{1}{2} + \varepsilon} (p^{\beta})^{- \frac{1}{16} + 30 \varepsilon} 
= O_{f, D, \alpha, \varepsilon}\left( (p^{\beta})^{-\frac{1}{16} + \varepsilon} \right).  \end{align*}
Since the estimate is again seen easily to be uniform in $m$, we deduce that the corresponding sum is bounded in this way, and hence that 
\begin{align*} D_1^{\star}(\alpha, \beta; p^{\frac{\beta}{8}}) \ll_{p, \varepsilon} (N \vert D \vert p^{2 \alpha}) (p^{\beta})^{- \frac{1}{16} + 30 \varepsilon}
= O_{f, D, \alpha, p, \varepsilon} \left( (p^{\beta})^{- \frac{1}{16} + \varepsilon} \right). \end{align*} 
The result then follows after putting together this estimate with that for the contributions of $b=0$ and $a=0$ terms described above. \end{proof}

\begin{remark}[A remark on the sums $D_1(\rho, \beta; p^{\frac{\beta}{8}})$.] 

Let us now consider the sum $D_1(\rho, \beta; p^{\frac{\beta}{8}})$, which we can expand as 
\begin{align*} &D_1(\rho, \beta; p^{\frac{\beta}{8}}) \\ &= \sum_{A \in \operatorname{Pic}(\mathcal{O}_{p^{\alpha}})} 
\rho(A) \sum_{m \geq 1 \atop (m, p)=1} \frac{\omega(m)}{m} \sum_{ {n \geq 1 \atop (n, p) =1} \atop m^2 n \equiv 1 \bmod p^{\beta}} 
\frac{r_A(n) \lambda(n)}{n^{\frac{1}{2}}} V_{1}\left(\frac{ m^2 n }{ N \vert D \vert p^{2\alpha} p^{\frac{15}{8}\beta} } \right) \\ 
& - \frac{1}{\varphi(p)} \sum_{A \in \operatorname{Pic}(\mathcal{O}_{p^{\alpha}})} \rho(A) \sum_{m \geq 1 \atop (m, p)=1} \frac{\omega(m)}{m} 
\sum_{ {n \geq 1 \atop m^2 n \equiv 1 \bmod p^{\beta-1}} \atop m^2 n \not \equiv 1 \bmod p^{\beta}} 
\frac{r_A(n) \lambda(n)}{n^{\frac{1}{2}}} V_1 \left(\frac{m^2 n }{ N \vert D \vert p^{2 \alpha} p^{\frac{15}{8} \beta} } \right). \end{align*}
As before, $r_A(n)$ denotes the number of ideals of norm $n$ in $A \in \operatorname{Pic}(\mathcal{O}_{p^{\alpha}})$, 
and we can use the relation $(\ref{QO})$ to evaluate this equivalently as  
\begin{align*} &D_1(\rho, \beta; p^{\frac{\beta}{8}}) \\
&= \frac{2}{\varphi^{\star}(p^{\beta})} \sum_{\chi \bmod p^{\beta} \atop \chi(-1)=1, \operatorname{primitive}} 
\sum_{A \in \operatorname{Pic}(\mathcal{O}_{p^{\alpha}})} \rho(A) \sum_{m \geq 1} \frac{\omega \chi^2(m)}{m} 
\sum_{ n \geq 1 \atop (n, p) =1} \frac{r_A(n)\chi(n) \lambda(n)}{n^{\frac{1}{2}}} 
V_1 \left(\frac{m^2 n }{ N \vert D \vert p^{2\alpha} p^{\frac{15}{8} \beta} } \right). \end{align*}
Let us for each $A \in \operatorname{Pic}(\mathcal{O}_{p^{\alpha}})$ write $q_A = q_A(x, y)$ 
to denote the corresponding positive definite binary quadratic form. This latter expression can then be written equivalently as 
\begin{align*} &D_1(\rho, \beta; p^{\frac{\beta}{8}}) \\
&= \frac{2}{\varphi^{\star}(p^{\beta})} \sum_{\chi \operatorname{mod} p^{\beta} \atop \chi(-1)=1, \operatorname{primitive}} 
\sum_{A \in \operatorname{Pic}(\mathcal{O}_{p^{\alpha}})} \rho(A) \sum_{m \geq 1} \frac{\omega \chi^2(m)}{m} 
\frac{1}{w} \sum_{ a, b \in {\bf{Z}} \atop q_A(a, b) \neq 0 } \frac{\lambda(q_A(a, b) \chi(q_A(a, b)) }{q_A(a, b)^{\frac{1}{2}}} 
V_{1}\left(\frac{m^2 q_A(a, b) }{ N \vert D \vert p^{2 \alpha } p^{\frac{15}{8} \beta} } \right). \end{align*}
Notice that we could in principal attempt a similar argument for these sums, although the contribution 
of the fixed character $\rho$ makes it difficult a priori to use the corresponding estimate to derive nonvanishing results. 
However, we shall see later that it is possible to average the corresponding sum over primitive ring class characters $\rho$
of an exact given order (in the style of Rohrlich \cite{Ro2}) to reduce to a similar setup as we have here for the principal class. \end{remark}

\subsection{Main estimates} 

Using Propositions \ref{D1tilde} and \ref{D1}, we now derive the following estimates for the averages 
\begin{align*} H^{(0)}(\alpha, \beta) &= \frac{2}{\# \operatorname{Pic}(\mathcal{O}_{p^{\alpha}}) \varphi^{\star}(p^{\beta})} 
\sum_{\rho \in \operatorname{Pic}(\mathcal{O}_p^{\alpha})^{\vee}} \sum_{\chi \bmod p^{\beta} \atop  \chi(-1)=1
\operatorname{primitive}} L(1/2, f \times \rho \chi \circ {\bf{N}}) \end{align*} computed in Propositions \ref{haf} above. 

\begin{theorem}\label{cycharm} 

Fix integers $\alpha \geq 0$ and $\beta \geq 4$. Let us again write $h(\mathcal{O}_{p^{\alpha}}) = \# \operatorname{Pic}(\mathcal{O}_{p^{\alpha}})$ 
to lighten notation. If the pair $(f, \mathcal{W})$ is generic, then the average $H^{(0)}(\alpha, \beta)$ is estimated as 
\begin{align*} H^{(0)}(\alpha, \beta) &= \frac{2}{w} \left( 1 + \frac{\lambda(\mathfrak{D})}{\mathfrak{D}^{\frac{1}{2}}} \right) 
\cdot \frac{2}{\varphi^{\star}(p^{\beta})} \sum_{\chi \bmod p^{\beta} \atop \chi(-1) =1, \operatorname{primitive}} 
L(1, \omega\chi^2) \cdot \frac{ L(1, \operatorname{Sym^2} f \otimes \chi) }{L(2, \chi^2)} \\
&+ O_{\varepsilon} \left( (N \vert D \vert p^{2 \alpha})^{\frac{1}{2} + \varepsilon} (p^{\beta})^{- \frac{1}{16} + 30 \varepsilon} \right). \end{align*}
Moreover, the average $H^{(0)}(\alpha, \beta)$ does not vanish if $\beta \geq 4$ is sufficiently large. \end{theorem} 

\begin{proof} 

We start with the formula of Proposition \ref{haf}, taking the unbalancing parameter $Z = p^{\frac{\beta}{8}}$, and 
using a variation of the argument of Lemma \ref{haf2} to reduce to estimating the sum of leading sums 
\begin{align*} D_1(\alpha, \beta; p^{\frac{\beta}{8}}) + \widetilde{D}_1(\alpha, \beta; p^{\frac{\beta}{8}}). \end{align*} 
We then deduce the stated estimate after putting together the results of Propositions \ref{D1tilde} and \ref{D1}. 
To deduce the nonvanishing of the average, note that we can ignore the error term (which tends to zero with $\beta$). 
To estimate the leading term, let us first consider the Dirichlet series in $s \in {\bf{C}}$ (first for $\Re(s) > 1$),
\begin{align*} \frac{2}{\varphi^{\star}(p^{\beta})} \sum_{\chi \bmod p^{\beta} \atop \chi(-1), \operatorname{primitive}} 
L(s, \omega \chi^2) \cdot \frac{L(s, \operatorname{Sym^2} f \otimes \chi)}{L(2s, \chi^2)} &= 
\frac{2}{\varphi^{\star}(p^{\beta})} \sum_{\chi \bmod p^{\beta} \atop \chi(-1), \operatorname{primitive}}
\sum_{m \geq 1} \frac{\omega\chi^2(m)}{m^s} \sum_{n \geq 1} \frac{\lambda(n^2) \chi(n)}{n^s}. \end{align*}
Fix a primitive (even) Dirichlet character $\chi \bmod p^{\beta}$. 
Recall the P\'olya-Vinogradov inequality shows that for any real parameter $X \geq 1$, we have the estimate 
\begin{align}\label{PV} L(s, \omega \chi^2) &= \sum_{m \leq X} \frac{\omega \chi^2(m)}{m^s} 
+ O_D \left( (p^{\beta})^{\frac{1}{2}} \log p^{\beta} X^{-\Re(s)} \right). \end{align}
On the other hand, as explained in \cite[Lemma 4.1]{CM} (we vary their setup),
we can use the automorphy of $\operatorname{Sym}^2 f \otimes \chi$ 
to derive the following contour estimate for any $\frac{3}{2} > \Re(s) \geq 1$ and real parameter $Y \geq 1$:
\begin{align*} &\sum_{n \geq 1} \frac{\lambda(n^2) \chi(n)}{n^s} e^{- \frac{n}{Y}} 
= \int_{\Re(u) = 2} \frac{L(s + u, \operatorname{Sym}^2 f \otimes \chi)}{L(2(s + u), \chi^2)} \Gamma(u) Y^u \frac{du}{2\pi i} \\
&= \frac{L(s, \operatorname{Sym}^2 f \otimes \chi)}{L(2s, \chi^2)} + \int_{\Re(u) = (1/2 - \Re(s))} 
\frac{L(s + u, \operatorname{Sym}^2 f \otimes \chi)}{L(2(s + u), \chi^2)} \Gamma(u) Y^u \frac{du}{2\pi i}. \end{align*}
Writing $q$ to denote the conductor of $\operatorname{Sym}^2 f \otimes \chi$, and using the subconvexity bound 
\begin{align*} L(u, \operatorname{Sym}^2 f \otimes \chi) \ll_{\varepsilon} q^{\frac{1}{2} + \varepsilon} \vert u \vert^{\frac{3}{4} + \varepsilon} \end{align*}
for $\Re(u) = 1/2$, we derive the estimate  
\begin{align*} \sum_{n \geq 1} \frac{\lambda(n^2) \chi(n)}{n^s} e^{- \frac{n}{Y}} &= \frac{L(s, \operatorname{Sym}^2 f \otimes \chi)}{L(2s, \chi^2)} 
+ O_{\varepsilon}\left( (q \vert s \vert) ^{\varepsilon}  q^{\frac{1}{2}} \vert s \vert^{\frac{3}{4}} Y^{\frac{1}{2} - \Re(s)}\right). \end{align*}
Taking $Y = q^2$ (as in \cite[Lemma 4.1]{CM}) and $s = 1$, we then derive the estimate
\begin{align}\label{ib1} \frac{L(1, \operatorname{Sym}^2 f \otimes \chi)}{L(2, \chi^2)} &= 
\sum_{n \geq 1} \frac{\lambda(n^2) \chi(n)}{n} e^{- \frac{n}{q^2}} + O_{\varepsilon} \left( q^{-\frac{1}{2} + \varepsilon} \right), \end{align}
and via $\vert \lambda(n^2) \vert \ll d_2(n)$ the individual bounds 
\begin{align}\label{ib2} \sum_{n \geq 1} \frac{\lambda(n^2) \chi(n)}{n} e^{- \frac{n}{q^2}}, 
~~~~~ \frac{L(1, \operatorname{Sym}^2 f \otimes \chi)}{L(2, \chi^2)} &\ll_{\varepsilon} q^{\varepsilon}. \end{align}

Let us now consider the product of $L$-values 
\begin{align}\label{lead} L(1, \omega \chi^2) \cdot \frac{L(1, \operatorname{Sym}^2 f \otimes \chi)}{L(2, \chi^2)}
&= \left( \sum_{m \leq X} \frac{\omega \chi^2(m)}{m} + O_D \left( p^{\frac{\beta}{2}} (\log p^{\beta})X^{-1} \right)  \right) \cdot
\frac{L(1, \operatorname{Sym}^2 f \otimes \chi)}{L(2, \chi^2)},\end{align}
where we have used the P\'olya-Vinogradov inquality $(\ref{PV})$ at $s=1$ in to estimate the factor $L(1, \omega \chi^2)$. 
Note that since $\operatorname{Sym}^2 f$ determines some automorphic form on $\operatorname{GL}_3$, 
the conductor $q = O_f(p^{3 \beta})$ is of size $p^{3 \beta}$. Using $(\ref{ib1})$ to estimate the first term in the expression $(\ref{lead})$, we find that 
\begin{align*} \sum_{m \leq X} \frac{\omega \chi^2(m)}{m} \frac{L(1, \operatorname{Sym}^2 f \otimes \chi)}{L(2, \chi^2)} &= 
\sum_{m \leq X} \frac{\omega \chi^2(m)}{m} \left( \sum_{n \geq 1} \frac{\lambda(n^2)\chi(n)}{n} e^{- \frac{n}{q^2}} 
+ O_{\varepsilon}(q^{-\frac{1}{2} + \varepsilon}) \right). \end{align*}
Let us now take $X = (p^{\beta})^{\frac{1}{2} + \eta}$ for some small $\varepsilon < \eta < 1/2$, so that this latter sum is estimated as 
\begin{align*} \sum_{m \leq X} \frac{\omega \chi^2(m)}{m} \frac{L(1, \operatorname{Sym}^2 f \otimes \chi)}{L(2, \chi^2)}  &=
\sum_{m \leq X} \frac{\omega \chi^2(m)}{m} \sum_{n \geq 1} \frac{\lambda(n^2)\chi(n)}{n} e^{- \frac{n}{q^2}} 
+ O_{\varepsilon}\left( \log ((p^{\beta})^{\frac{1}{2} + \eta})  q^{-\frac{1}{2} + \varepsilon} \right). \end{align*}
Note that by using $(\ref{ib1})$ again, we can also estimate the first term in $(\ref{lead})$ as  
\begin{align*} \sum_{m \leq (p^{\beta})^{\frac{1}{2} + \eta}} \frac{\omega \chi^2(m)}{m} \frac{L(1, \operatorname{Sym}^2 f \otimes \chi)}{L(2, \chi^2)} 
&= \sum_{m \leq (p^{\beta})^{\frac{1}{2} + \eta}} \frac{\omega \chi^2(m)}{m} \sum_{n \geq 1} \frac{\lambda(n^2) \chi(n)}{n} 
+ O_{\varepsilon}\left( \log ((p^{\beta})^{\frac{1}{2} + \eta} ) q^{-\frac{1}{2} + \varepsilon} \right). \end{align*} 
To estimate the second term in the expression $(\ref{lead})$, we simply use the contour bound $(\ref{ib2})$ 
for the symmetric square $L$-value at $s=1$ to obtain an error of $O_{D, \varepsilon}\left( (p^{\beta})^{\varepsilon - \eta} \right)$. 
In this way, we derive the preliminary estimate 
\begin{align}\label{prelim} L(1, \omega \chi^2) \cdot \frac{L(1, \operatorname{Sym}^2 f \otimes \chi)}{L(2, \chi^2)} 
&= \sum_{m \leq (p^{\beta})^{\frac{1}{2} + \eta}} \frac{\omega \chi^2(m)}{m} 
\sum_{n \geq 1} \frac{\lambda(n^2) \chi(n)}{n} e^{-\frac{n}{q^2}} +O_{D, \varepsilon}((p^{\beta})^{-\eta + \varepsilon}). \end{align}
Taking the weighted sum over primitive even Dirichlet characters $\chi \bmod p^{\beta}$ in the latter estimate, we derive 
\begin{align*} \frac{2}{\varphi^{\star}(p^{\beta})} &\sum_{\chi \bmod p^{\beta} \atop \chi(-1)= 1, \operatorname{primitive}} 
L(1, \omega\chi^2) \cdot \frac{ L(1, \operatorname{Sym^2} f \otimes \chi) }{L(2, \chi^2)} 
\\&= \frac{2}{\varphi^{\star}(p^{\beta})} \sum_{\chi \bmod p^{\beta} \atop \chi(-1) = 1 \operatorname{primitive}} \sum_{m \leq (p^{\beta})^{\frac{1}{2} + \eta}} 
\frac{\omega\chi^2(m)}{m} \sum_{n \geq 1} \frac{\lambda(n^2) \chi(n)}{n} e^{- \frac{n}{q^2}}+ O_{D, \varepsilon}\left( (p^{\beta})^{-\eta + \varepsilon} \right). \end{align*}

Let us now consider the leading term in this latter expression, which after unwinding via $(\ref{QO})$ equals 
\begin{align*} \frac{2}{\varphi^{\star}(p^{\beta})} &\sum_{\chi \bmod p^{\beta} \atop \chi(-1)=1, \operatorname{primitive}} 
\sum_{1 \leq m \leq (p^{\beta})^{\frac{1}{2} + \eta}} \frac{ \omega \chi^2(m) }{ m } \sum_{n \geq 1} \frac{\lambda(n^2) \chi(n)}{n} e^{- \frac{n}{q^2}}\\
&= \sum_{n \geq 1} \frac{\lambda(n^2)}{n} e^{- \frac{n}{q^2}} 
\left( \sum_{m \leq (p^{\beta})^{\frac{1}{2} + \eta} \atop m \equiv \pm \overline{n}^{\frac{1}{2}} \bmod p^{\beta}} \frac{\omega(m)}{m} - \frac{1}{\varphi(p)} 
\sum_{ {m \leq (p^{\beta})^{\frac{1}{2} + \eta} \atop m^2 \equiv \pm \overline{n} \bmod p^{\beta-1}} \atop m \not\equiv \pm \overline{n}^{\frac{1}{2}} \bmod p^{\beta}} 
\frac{\omega(m)}{m} \right). \end{align*} 
Here, $\overline{n}^{\frac{1}{2}}$ denotes a square root of $\overline{n} \bmod p^{\beta}$ (when it exists). 
Now, taking the length $X = (p^{\beta})^{\frac{1}{2} + \eta}$ as we do, the congruence in the definition of the second $m$-sum 
is never satisfied, and there is only one solution to each congruence in the definition of the first $m$-sum. 
Hence, it remains to estimate the contribution of 
\begin{align}\label{RE} \sum_{n \geq 1} \frac{\lambda(n^2)}{n} e^{-\frac{n}{q^2}} 
\sum_{m \leq (p^{\beta})^{\frac{1}{2} + \eta} \atop m \equiv \pm \overline{n}^{\frac{1}{2}} \bmod p^{\beta}} \frac{\omega(m)}{m}
&= \sum_{n \geq 1} \frac{\omega(\overline{n}^{\frac{1}{2}})}{\overline{n}^{\frac{1}{2}}} \frac{\lambda(n^2)}{n} e^{- \frac{n}{q^2}}.\end{align} 
It is easy to see that $n=1$ term in $(\ref{RE})$ contributes $e^{-\frac{1}{q^2}} \gg 1$. To estimate remaining sum over $n \geq 2$, 
we partition it into residue classes modulo $p^{\beta}$ of integers $1 \leq u \leq (p^{\beta})^{\frac{1}{2} + \eta}$ 
which (via Hensel's lemma) are quadratic residues modulo $p$, then expand out in terms of progressions over integers $t \geq 0$:
\begin{align*} \sum_{n \geq 2}  \frac{\omega(\overline{n}^{\frac{1}{2}})}{ \overline{n}^{\frac{1}{2}} }  \frac{\lambda(n^2)}{n} e^{- \frac{n}{q^2}}
&= \sum_{2 \leq u \leq (p^{\beta})^{\frac{1}{2} + \eta} \atop \left(\frac{u}{p} \right)_p = 1} \frac{\omega(u)}{u} 
\sum_{t \geq 0} \frac{\lambda(\overline{u}^4 + tp^{\beta})}{(\overline{u}^2 + tp^{\beta})} e^{- \frac{(\overline{u}^2 + t p^{\beta})}{q^2}} \\
&=  \sum_{2 \leq u \leq (p^{\beta})^{\frac{1}{2} + \eta} \atop \left(\frac{u}{p} \right)_p = 1} \frac{\omega(u)}{u} e^{- \frac{\overline{u}^2}{q^2}}
\sum_{t \geq 0} \frac{\lambda(\overline{u}^4 + tp^{\beta})}{(\overline{u}^2 + tp^{\beta})} e^{- \frac{tp^{\beta}}{q^2}}. \end{align*} 
Let us first consider the sum over integers  $t \geq 1$ terms in this latter expression, 
which by Deligne's theorem $\lambda(n) \ll_{\varepsilon} n^{\varepsilon}$ 
together with the basic lower bound $e^{\frac{n}{q^2}} \gg n e^{\frac{1}{q^2}}$ can be bounded above in modulus by 
\begin{align*} \sum_{2 \leq u \leq (p^{\beta})^{\frac{1}{2} + \eta} \atop \left(\frac{u}{p} \right)_p = 1} &\frac{\omega(u)}{u} e^{- \frac{\overline{u}^2}{q^2}} 
\sum_{t \geq 1} \frac{\lambda(\overline{u}^4 + tp^{\beta})}{(\overline{u}^2 + tp^{\beta})} e^{- \frac{tp^{\beta}}{q^2}} \ll_{\varepsilon} 
\sum_{2 \leq u \leq (p^{\beta})^{\frac{1}{2} + \eta} \atop \left(\frac{u}{p} \right)_p = 1} \frac{1}{u} e^{-\frac{\overline{u}^2}{q^2}}  
\sum_{t \geq 1} \frac{(\overline{u}^4 + tp^{\beta})^{\varepsilon}}{(\overline{u}^2 + tp^{\beta})} e^{- \frac{t p^{\beta}}{q^2}} \\ 
&\ll e^{-\frac{1}{q^2}} \sum_{2 \leq u \leq (p^{\beta})^{\frac{1}{2} + \eta} \atop \left(\frac{u}{p} \right)_p = 1} \frac{1}{u \overline{u}^2} 
\sum_{t \geq 1} \frac{1}{(tp^{\beta})^{2 - \varepsilon}} = \left( \frac{1}{p^{\beta}} \right)^{2 - \varepsilon} e^{-\frac{1}{q^2}} 
\sum_{2 \leq u \leq (p^{\beta})^{\frac{1}{2} + \eta} \atop \left(\frac{u}{p} \right)_p = 1} \frac{1}{u \overline{u}^2} \sum_{t \geq 1}  \frac{1}{t^{2 - \varepsilon}} \\ 
&= O_{\varepsilon} \left( e^{-\frac{1}{q^2}} \left( p^{\beta} \right)^{ \eta + \varepsilon - \frac{3}{2} } \right). \end{align*} 
Let us now consider the remaining contributions from $t=0$,
\begin{align}\label{leading} \sum_{2 \leq u \leq (p^{\beta})^{\frac{1}{2} + \eta} \atop \left(\frac{u}{p} \right)_p = 1} 
\frac{\omega(u)}{u} \frac{\lambda(\overline{u}^4)}{\overline{u}^{2}} e^{- \frac{\overline{u}^2}{q^2}}. \end{align} 
Observe that for each term $u \geq 2$ in the latter sum, we have that  $\overline{u} u = 1 + l p^{\beta}$ for some integer $l \geq 1$. 
Hence, for each such pair in the sum, we have $\overline{u} u \geq 1 + p^{\beta}$. Using Deligne's bound again, it is then easy to see that
\begin{align*} \sum_{2 \leq u \leq (p^{\beta})^{\frac{1}{2} + \eta} \atop \left(\frac{u}{p} \right)_p = 1}  
\frac{\omega(u)}{u}  \frac{\lambda(\overline{u}^4)}{\overline{u}^{2}} e^{- \frac{ \overline{u}^2 }{q^2}}
&\ll_{\varepsilon} e^{-\frac{1}{q^2}}
\sum_{2 \leq u \leq (p^{\beta})^{\frac{1}{2} + \eta} \atop \left(\frac{u}{p} \right)_p = 1} \frac{1}{ \overline{u}^{3-\varepsilon}(1 + p^{\beta})} \\
&= e^{-\frac{1}{q^2}} \frac{1}{(1 + p^{\beta})} \sum_{2 \leq u \leq (p^{\beta})^{\frac{1}{2} + \eta} \atop \left(\frac{u}{p} \right)_p = 1} 
\frac{1}{\overline{u}^{3-\varepsilon}} = 
O_{f, p, \varepsilon}\left(e^{-\frac{1}{q^2}} (p^{\beta})^{\eta + \varepsilon - \frac{1}{2}}  \right). \end{align*}
In this way, we deduce that the leading term $(\ref{leading})$ is estimated as 
$e^{-\frac{1}{q^2}} + O_{f, p, \varepsilon}\left( e^{-\frac{1}{q^2}}(p^{\beta})^{\eta + \varepsilon - \frac{1}{2}}  \right)$, and hence  
\begin{align*} \frac{2}{\varphi^{\star}(p^{\beta})} &\sum_{\chi \bmod p^{\beta} \atop \chi(-1) = 1, \operatorname{primitive}} 
L(1, \omega\chi^2) \cdot \frac{ L(1, \operatorname{Sym^2} f \otimes \chi) }{L(2, \chi^2)} 
=  e^{-\frac{1}{q^2}} + O_{f, p, \varepsilon}\left( e^{-\frac{1}{q^2}}(p^{\beta})^{\eta + \varepsilon - \frac{1}{2}}  \right) 
+ O_{D, \varepsilon}\left( (p^{\beta})^{-\eta + \varepsilon} \right). \end{align*}
Is then easy to see (as $q = O_f(p^{3\beta})$) that this sum of $L$-values converges with $\beta$ to $1$, 
and hence that the average $H^{(0)}(\alpha, \beta)$ converges to the nonvanishing constant for 
$1 + \lambda(\mathfrak{D}) \mathfrak{D}^{- \frac{1}{2}}$ with $\beta$. \end{proof} 

\section{Galois conjugate ring class characters}

We now derive the following refinement of Theorems \ref{SDave} and Theorem \ref{cycharm} for Galois conjugate ring class characters, 
or more simply ring class characters of the same exact order, 
analogous to the settings considered by Rohrlich \cite{Ro2}, \cite{Ro} Vatsal \cite{Va} and Cornut \cite{Cor}.

Fix an integer $\alpha \geq 1$, as well as a primitive ring class character $\rho$ of conductor $p^{\alpha}$.
Let us lighten notation in writing $G(\alpha) = \operatorname{Pic}(\mathcal{O}_{p^{\alpha}})$ to denote the class group of the order $\mathcal{O}_{p^{\alpha}}$.
Given an integer $d \geq 1$ dividing the class number $h(\mathcal{O}_{p^{\alpha}}) = \# G(\alpha)$, let us also write 
$G(\alpha)^d = \left\lbrace g^d : g \in G(\alpha) \right\rbrace$ to denote the subgroup of $d$-th powers in $G(\alpha)$.
Writing $G(\alpha)^{\vee}$ to denote the characters of $G(\alpha)$, the characters $\rho \in G(\alpha)^{\vee}$ of exact 
order $d$ are precisely those for which $\rho^d = {\bf{1}}$, where ${\bf{1}}$ denotes the trivial character of $G(\alpha)$. 
Note that there are precisely $\varphi(d)$ many such characters, and also that such characters factor through the quotient $G(\alpha) / G(\alpha)^d$. 
Now, it is well-known and classical (see e.g.~\cite[$\S 3.1$]{IK}) that such characters detect $d$-th powers via the following orthogonality relation:
\begin{align}\label{QOGC} \sum_{\rho \in G(\alpha)^{\vee} \atop \rho^d \equiv {\bf{1}}} \rho(A) 
&= \begin{cases} [G(\alpha): G(\alpha)^d] &\text{ if $A \in G(\alpha)^d$} \\ 0 &\text{ if $A \notin G(\alpha)^d$}. \end{cases} \end{align}

Given an integer $1 \leq x \leq \alpha$, let us now define (in either case $k \in \lbrace 0, 1 \rbrace$ on the generic root number) the 
corresponding average over ring class characters $\rho \in \operatorname{Pic}(\mathcal{O}_{p^{\alpha}})$ of exact order $p^x$:
\begin{align}\label{Ga}\Gamma^{(k)}(\alpha(x)) 
&= \frac{1}{[G(\alpha): G(\alpha)^{p^x}]} \sum_{\rho \in G(\alpha)^{\vee} \atop \rho^{p^x} = {\bf{1}}} L^{(k)}(1/2, f \times \rho). \end{align} 
Given an integer $\beta \geq 2$, let us also define the corresponding double average of central values  
\begin{align}\label{Gab} \Gamma(\alpha(x), \beta) 
&= \frac{1}{[G(\alpha): G(\alpha)^{p^x}]} \frac{2}{\varphi^{\star}(p^{\beta})} 
\sum_{\rho \in G(\alpha)^{\vee} \atop \rho^{p^x} = {\bf{1}}} 
\sum_{\chi \bmod p^{\beta} \atop \chi(-1)=1, \operatorname{primitive}} L(1/2, f \times \rho \chi \circ {\bf{N}}). \end{align} 
Here, both averages are defined with the understanding that the ring class characters $\rho$ of exact order $p^x$
are primitive of conductor $\alpha$, i.e.~so that the choice of $x$ is constrained by the choice of $\alpha$. Also, since
the index $[G(\alpha): G(\alpha)^{p^x}]$ is equivalent to the number of primitive ring class characters of $G(\alpha)$ 
of exact order $p^x$ ($=\varphi(p^x)$), we argue that $\# G(\alpha)^{p^x} = \#\operatorname{Pic}(\mathcal{O}_{p^{\alpha}})/\varphi(p^x)$, 
where $\#\operatorname{Pic}(\mathcal{O}_{p^{\alpha}})$ can be described by Dedekind's class number formula $(\ref{Dedekind})$. 

\begin{lemma}\label{GAF} 

Let $\alpha \geq 1$ and $1 \leq x \leq \alpha$ and $\beta \geq 2$ be integers. 
We have the following formulae for the averages $\Gamma^{(k)}(\alpha(x))$ and $\Gamma(\alpha(x), \beta)$,
given in terms of the Dirichlet series expansion $(\ref{integralDirichlet})$: We have the self-dual average formula
\begin{align*}\Gamma^{(k)}(\alpha(x)) = 2 \sum_{A \in G(\alpha) \atop A \in G(\alpha)^{p^x}}
\sum_{m \geq 1} \frac{\omega(m)}{m} \sum_{n \geq 1 \atop (n, p)=1} \frac{\lambda(n) r_A(n)}{n^{\frac{1}{2}}} 
V_{k+1} \left(\frac{m^2 n}{N \vert D \vert p^{2(\alpha + \beta)}} \right), \end{align*}
and for any choice of real parameter $Z>0$ the non self-dual average formula
\begin{align*} \Gamma(\alpha(x), \beta) &= \sum_{A \in G(\alpha) \atop A \in G(\alpha)^{p^x}} 
\left( D_A(\alpha, \beta; Z) + \widetilde{D}_A (\alpha, \beta; Z)\right), \end{align*}
where 
\begin{align*} D_A(\alpha, \beta; Z) = \sum_{m \geq 1} \frac{\omega(m)}{m} &\sum_{n \geq 1 \atop m^2 n \equiv \pm 1 \mod p^{\beta}} 
\frac{\lambda(n) r_A(n)}{n^{\frac{1}{2}}} V_1 \left( \frac{Z m^2 n}{N \vert D \vert p^{2(\alpha + \beta)}} \right) \\ 
&- \frac{1}{\varphi(p)} \sum_{m \geq 1} \frac{\omega(m)}{m} 
\sum_{ {n \geq 1 \atop m^2 n \equiv \pm 1 \bmod p^{\beta-1}} \atop m^2n \not\equiv \pm 1 \bmod p^{\beta}  } 
\frac{\lambda(n) r_A(n)}{n^{\frac{1}{2}}} V_1 \left( \frac{Z m^2 n}{N \vert D \vert p^{2(\alpha + \beta)}}\right) \end{align*}
and 
\begin{align*} \widetilde{D}_A(\alpha, \beta; Z) &= \frac{\omega(N) \tau(\omega)^4}{(\vert D \vert p^{\beta})^{\frac{1}{2}}}  
\frac{p}{\varphi(p)} \sum_{m \geq 1} \frac{\omega(m)}{m} 
\sum_{n \geq 1 \atop (n, p)=1} \frac{\lambda(n)r_A(n)}{n^{\frac{1}{2}}} V_1 \left( \frac{m^2 n}{Z N \vert D \vert p^{2(\alpha + \beta)}} \right) 
\operatorname{Kl}_4( \pm (m^2 n \overline{N}^2)^{\frac{1}{2}} \overline{D}, p^{\beta}) . \end{align*}
Here (again), we write $r_A(n)$ to denote the number of ideals in the class $A \in G(\alpha) = \operatorname{Pic}(\mathcal{O}_{p^{\alpha}})$ 
of norm equal to $n$, and also $\operatorname{Kl}_4(\pm c, p^{\beta}) = \operatorname{Kl}_4(c, p^{\beta}) + \operatorname{Kl}_4(-c, p^{\beta})$ 
for $c$ an integer prime to $p$ to lighten notation.

\end{lemma}

\begin{proof} 

To compute the average $\Gamma^{(k)}(\alpha(x))$, we start with the fact that each of the values in the sum $(\ref{Ga})$ 
can be expressed via the approximate functional equation of Proposition \ref{cvformula} (using the expansion $(\ref{integralDirichlet})$) as 
\begin{align*} L^{(k)}(1/2, f \times \rho) 
&= 2 \sum_{m \geq 1} \frac{\omega(m)}{m} \sum_{n \geq 1 \atop (n, p)=1} \frac{\lambda(n) c_{\rho}(n)}{n^{\frac{1}{2}}}
V_{k+1} \left( \frac{m^2 n}{N \vert D \vert p^{2(\alpha + \beta)}} \right) \\ &= 
2 \sum_{m \geq 1} \frac{\omega(m)}{m} \sum_{n \geq 1 \atop (n, p)=1} \frac{\lambda(n) r_A(n)}{n^{\frac{1}{2}}}
V_{k+1} \left( \frac{m^2 n}{N \vert D \vert p^{2(\alpha + \beta)}} \right)  \sum_{A \in G(\alpha)} \rho(A). \end{align*}
Notice that the conductor $p^{\alpha}$ does not change in the sum $(\ref{Ga})$, and so there is no need to use partial summation
in taking the sum over ring class characters $\rho \in G(\alpha)$ of exact order $p^x$. Hence, we may simply use $(\ref{QOGC})$ 
to evaluate the inner $A$-sum in this latter expression, from which we derive the stated formula.  

To compute the average $\Gamma(\alpha(x), \beta)$, observe that we have (in the setup of Proposition \ref{CAF}) the relation 
\begin{align*} \Gamma(\alpha(x), \beta) 
&= \sum_{\rho \in G(\alpha)^{\vee} \atop \rho^{p^x} = {\bf{1}}} \left( D_1(\rho, \beta; Z) + \widetilde{D}_1(\rho, \beta; Z) \right).\end{align*}
The stated average formula is then deduced in a similar way, opening up the coefficients $c_{\rho}(n)$ in the sums $D_1(\rho, \beta; Z)$
and $\widetilde{D}_1(\rho, \beta; Z)$ (as defined in Proposition \ref{CAF} above), and using the relation $(\ref{QOGC})$. \end{proof}

Let us now describe the estimates we can derive for these Galois averages, using the same arguments and calculations as given for 
Theorems \ref{SDave} and \ref{cycharm} above. We start with the self-dual averages $\Gamma^{(k)}(\alpha(x))$. Here, we apply the 
argument of Theorem \ref{SDerror} (ii) above to each class $A \in G(\alpha)^{p^x}$ to derive the following result.

\begin{theorem}\label{SDGA} 

Let $\alpha \geq 1$ and $1 \leq x \leq \alpha$ be integers. 
We have the following estimates for the averages $\Gamma^{(k)}(\alpha(x))$ in either case $k \in \lbrace 0, 1 \rbrace$ on the generic root number: \\

\begin{itemize}

\item[(i)] If $k= 0$, then for some constant $\eta_0 >0$ we have the estimate
\begin{align*} \Gamma^{(0)}(\alpha(x))&= 
\frac{4}{w} \cdot L(1, \omega) \cdot \frac{L^{(p^{\alpha})}(1, \operatorname{Sym^2} f)}{\zeta^{(p^{\alpha})}(2)} 
+ O \left( (N \vert D \vert p^{2 \alpha})^{-\eta_0} \right). \end{align*}

\item[(ii)] If $k=1$, then for some constant $\eta_1 >0$ we have 
\begin{align*} \Gamma^{(1)}(\alpha(x))&= 
\frac{4}{w} \cdot L(1, \omega) \cdot \frac{L^{(p^{\alpha})}(1, \operatorname{Sym^2} f )}{\zeta^{(p^{\alpha})}(2)} 
\left[ \frac{1}{2}\log(N\vert D \vert p^{2 \alpha}) + \frac{L'}{L}(1, \omega)  \right. \\  
& \left. + \frac{L'^{(p^{\alpha})}}{L^{(p^{\alpha})}}(1, \operatorname{Sym^2} f) - \frac{\zeta'^{(p^{\alpha})} }{\zeta^{(p^{\alpha})}}(2) 
- \gamma- \log 2 \pi\right]  + O((N \vert D \vert p^{2 \alpha})^{-\eta_1} ). \end{align*}
\end{itemize}

In particular, in either case $k \in \lbrace 0, 1 \rbrace$ on the generic root number, the averages $\Gamma^{(k)}(\alpha(x))$ 
converge to nonzero constants with the quantity $N \vert D \vert p^{2 \alpha}$.

\end{theorem}

\begin{proof} 

Let us for each class $A \in G(\alpha)^{p^x}$ write $q_A$ to denote the corresponding binary quadratic form. 
Let us then write $\theta_{q_A}(z)$ to denote the binary theta series associated to $q_A$, 
viewed as an automorphic form on $\operatorname{GL}_2({\bf{A}}_{\bf{Q}})$ (as in the proof of Theorem \ref{SDerror} above).
Hence, for $x \in {\bf{A}}_{\bf{Q}}$ an adele, and $y = y_f y_{\infty} \in {\bf{A}}_{\bf{Q}}^{\times}$ 
an idele with finite component $y_f \in {\bf{A}}_{ {\bf{Q}}, f}^{\times}$
and archimedean component $y_{\infty} \in {\bf{R}}^{\times}$, $\theta_{q_A}$ has the Fourier-Whittaker expansion 
\begin{align*} \theta_{q_A} \left( \left(\begin{array} {cc} y &  x \\ 0 & 1 \end{array}\right) \right) 
&= \vert y \vert^{\frac{1}{2}} \frac{1}{w} \sum_{\gamma_1, \gamma_2 \in {\bf{Z}} } e \left( q_A(\gamma_1, \gamma_2)(x + i y) \right). \end{align*}
Replacing the binary quadratic form $q_1$ associated to the trivial class by $q_A$ in the argument of Theorem \ref{SDerror} above, 
we see that for $Y >0$ any real parameter, we have the same convergent limit 
\begin{align*} \lim_{Y \rightarrow \infty} Y^{\frac{1}{2}} \sum_{m \geq 1} \frac{\omega (m) }{m^2} \rho_{\phi \overline{\theta}_{q_A}, 0} \left( \frac{m^2}{Y} \right)
&= \lim_{Y \rightarrow \infty} \sum_{m \geq 1} \frac{\omega(m)}{m} \frac{1}{w} \sum_{a, b \in {\bf{Z}} \atop q_A(a, b) \neq 0}
\frac{\lambda(q_A(a, b))}{q_A(a, b)^{\frac{1}{2}}} V_{k+1} \left( \frac{m^2 q_A(a,b)}{Y} \right) \\
&= L(1, \omega) \cdot \mathfrak{L}^{(k)}(1, f). \end{align*} 
Here, $\rho_{\phi \overline{\theta}_{q_A}, 0}$ denotes the constant coefficient of $\phi \overline{\theta}_{q_A}$, 
which viewed as a function on $y = y_f y_{\infty} \in {\bf{A}}_{\bf{Q}}^{\times}$ is 
\begin{align*} \rho_{\phi \overline{\theta}_{q_A}, 0}(y) &= \int_{ {\bf{A}}_{\bf{Q}} / {\bf{Q}}} \phi \overline{\theta}_{q_A} 
\left( \left(\begin{array} {cc} y &  x \\ 0 & 1 \end{array}\right) \right) dx, \end{align*}
and the product of $L$-values $\mathfrak{L}^{(k)}(1, f)$ has the same definition as in Theorem \ref{SDerror} (ii) above. 
The claimed estimates are then easy to deduce after considering the limiting behaviour of the sum over classes $A \in G(\alpha)^{p^x}$,
divided by the index $[G(\alpha): G(\alpha)^{p^x}]$. \end{proof} 

Note that after interpreting the weighted averages $\Gamma^{(k)}(\alpha(x))$ over ring class characters of exact order 
as Galois averages (in the sense of \cite{Sh} and \cite{GZ}, cf.~\cite{Ro2}), we derive completely analytic proofs of the 
theorems of Vatsal \cite{Va} and Cornut \cite{Cor} respectively (in the style of Rohrlich \cite{Ro2}), as stated in the introduction. 
Let us now consider the weighted double average $\Gamma^{(k)}(\alpha(x), \beta)$ over primitive Dirichlet characters $\chi \bmod p^{\beta}$ with $\beta \geq 4$
(so that we can evaluate hyper-Kloosterman sums in the style of Sali\'e).

\begin{theorem} \label{NSDGA} 

Fix integers $\alpha \geq 1$, $1 \leq x \leq \alpha$, and $\beta \geq 4$. We have the estimate
\begin{align*} \Gamma(\alpha(x), \beta) &= \frac{2}{w}  \prod_{A \in G(\alpha)^{p^x}} \left( 1 + \frac{\lambda(q_A(0, 1))}{q_A(0, 1)^{\frac{1}{2}}} \right) 
\cdot \frac{2}{\varphi^{\star}(p^{\beta})} 
\sum_{\chi \bmod p^{\beta} \atop \chi(-1) = 1, \operatorname{primitive}} L(1, \omega \chi^2)
\cdot \frac{ L(1, \operatorname{Sym^2} f \otimes \chi^2) }{L(2, \chi^2)} \\
&+ O_{p, \varepsilon}\left( \varphi(p^x) (N \vert D \vert p^{2 \alpha})^{\frac{1}{2} + \varepsilon} (p^{\beta})^{- \frac{1}{16} + 30 \varepsilon} \right). \end{align*}
In particular, the average $\Gamma(\alpha(x), \beta)$ converges to a nonzero constant for $\beta \geq 4$ sufficiently large. 
Hence, if $\rho$ is any primitive ring class character of conductor $p^{\alpha}$ (for $\alpha \geq 1$) 
and exact order $p^x$, then for each sufficiently larger integer $\beta \gg x$ there exists a primitive Dirichlet character 
$\chi \operatorname{mod} p^{\beta}$ such that $L(1/2, f \times \rho \chi \circ {\bf{N}}) \neq 0$.\end{theorem} 

\begin{proof} We start with the expression of Proposition \ref{GAF}, taking $Z = p^{\frac{\beta}{8}}$, 
and opening up the counting functions $r_A(n)$ for each class $A \in G(\alpha)^{p^x}$. 
We then apply the arguments of Theorem \ref{D1tilde}, Proposition \ref{D1}, and Theorem \ref{cycharm} 
to estimate each of the corresponding $A$-sums $D_A(\alpha, \beta; p^{\frac{\beta}{8}})$ and $D_A(\alpha, \beta; p^{\frac{\beta}{8}})$. \end{proof}

\end{document}